\documentclass[11pt]{article}
\usepackage{amsmath,amssymb,amsthm,float}
\usepackage{latexsym, yfonts}
\usepackage{graphicx,color}
\usepackage[all]{xy}
\usepackage{longtable}
\usepackage{enumitem}
\usepackage{tikz}
\usepackage{array}
\usepackage{calc}
\usepackage{adjustbox}
\usepackage[bf,width=.9\textwidth,hang]{caption}
\usepackage[text={7in,9in},centering]{geometry}
\usepackage[title]{appendix}
\usepackage[parfill]{parskip}
\begingroup
    \makeatletter
    \@for\theoremstyle:=definition,remark,plain\do{%
        \expandafter\g@addto@macro\csname th@\theoremstyle\endcsname{%
            \addtolength\thm@preskip\parskip
            }%
        }
\endgroup

\newtheorem{thm*}{Theorem}
\newtheorem{thm}{Theorem}[section]

\newtheorem{lem}[thm]{Lemma}
\newtheorem{cor}[thm]{Corollary}
\newtheorem{cor*}{Corollary}

\newenvironment{pf}[1]{\begin{trivlist}
\item[\hskip \labelsep {\bfseries #1}]}{\end{trivlist}}

\title{Irreducible $A_1$ Subgroups of Exceptional Algebraic Groups}
\author{Adam R. Thomas}

\begin{document}

\maketitle

\begin{abstract}
A closed subgroup of a semisimple algebraic group is called irreducible if it lies in no proper parabolic subgroup. In this paper we classify all irreducible $A_1$ subgroups of exceptional algebraic groups $G$. Consequences are given concerning the representations of such subgroups on various $G$-modules: for example, the conjugacy classes of irreducible $A_1$ subgroups are determined by their composition factors on the adjoint module of $G$.  
\end{abstract}

\section{Introduction} \label{intro}

Let $G$ be a reductive connected algebraic group over an algebraically closed field $K$ of characteristic $p$. A subgroup $X$ of $G$ is called \emph{$G$-irreducible} (or just \emph{irreducible} if $G$ is clear from the context) if it is closed and not contained in any proper parabolic subgroup of $G$. This definition, as given by Serre in \cite{ser04}, generalises the standard notion of an irreducible subgroup of $\text{GL}(V)$. Indeed, if $G = \text{GL}(V)$, a subgroup $X$ is $G$-irreducible if and only if $X$ acts irreducibly on $V$. Similarly, the notion of complete reducibility can be generalised (see \cite{ser04}): a subgroup $X$ of $G$ is said to be \emph{$G$-completely reducible} (or \emph{$G$-cr} for short) if, whenever it is contained in a parabolic subgroup $P$ of $G$, it is contained in a Levi subgroup of $P$.      

Now let $G$ be a simple algebraic group of exceptional type. In \cite{tho1}, the author classified the simple, $G$-irreducible connected subgroups of rank at least 2. In this paper we classify the $G$-irreducible $A_1$ subgroups, completing the classification of simple, $G$-irreducible connected subgroups. We note that the $G$-irreducible $A_1$ subgroups can be deduced from \cite[Theorem~1]{law} when $p > N(A_1,G)$ (see the preceding table to Theorem \ref{A1samecomp} for the definition), in particular when $p > 7$. In low characteristics there are fewer classes of irreducible $A_1$ subgroups but the existence of non-$G$-cr subgroups complicates the proof. We also note that if $G \neq E_8$ then partial results can be found in \cite{bon}; we require a set of conjugacy class representatives without repeat for the $G$-irreducible $A_1$ subgroups for the $E_8$ case and therefore classify the irreducible $A_1$ subgroups independently. The following theorem summarises the individual cases for each exceptional algebraic group $G$; it classifies the $G$-irreducible $A_1$ subgroups of $G$. 

\begin{thm*} \label{mainthm}
Suppose $X$ is a $G$-irreducible subgroup $A_1$ of a simple exceptional algebraic group $G$. Then $X$ is conjugate to exactly one subgroup of Tables \ref{G2tab} to \ref{E8tab}, found in Sections \ref{secG2} to \ref{secE8}, respectively and each subgroup in the tables is $G$-irreducible. 
\end{thm*}

The validity of Theorem \ref{mainthm} will be established by proving Theorems \ref{thmG2} to \ref{thmE8} found below, which classify the $G$-irreducible $A_1$ subgroups of $G$ when $G$ is of type $G_2$ through $E_8$. Each subgroup in Tables \ref{G2tab} to \ref{E8tab} is described by its embedding in some reductive, maximal connected subgroup, given in Theorem \ref{maximalexcep}. When we say a reductive, maximal connected subgroup we mean a subgroup that is maximal among all closed connected subgroups and is reductive. We note that the case $p=0$ can be recovered by simply removing any subgroup in the tables for which a non-zero field twist is necessary and assuming $p= \infty$ when reading inequalities; this yields only finitely many classes of irreducible $A_1$ subgroups when $p=0$.   

A natural question to ask is whether $G$-irreducible subgroups of a certain type exist, especially in small characteristics. As an immediate consequence of Theorems \ref{thmG2} to \ref{thmE8}, we reprove the following corollary, first proved by Liebeck and Testerman in \cite{LT} with a correction by Amende in \cite{bon}, showing that $G$-irreducible $A_1$ subgroups almost always exist.   

\begin{cor*} \label{A1subgroups}
Let $G$ be an exceptional algebraic group. Then $G$ contains an irreducible subgroup $A_1$, unless $G = E_6$ and $p=2$. 
\end{cor*}

Given the existence of irreducible $A_1$ subgroups, we can use the proofs of Theorems \ref{thmG2} to \ref{thmE8} to study their overgroups. The next result shows the existence of a reductive, maximal connected subgroup that contains representatives of each conjugacy class of $G$-irreducible $A_1$ subgroups in small characteristics, with one exception. 

\vspace{0.1cm}

\begin{cor*} \label{A1overgroups}
Let $G$ be an exceptional algebraic group and $p=2$ or $3$. Then there exists a reductive, maximal connected subgroup $M$ containing representatives of every $G$-conjugacy class of $G$-irreducible $A_1$ subgroups, unless $G=F_4$ and $p=3$ (in which case two reductive, maximal connected subgroups are required). The following table lists such subgroups $M$.   
\end{cor*}
\vspace{0.1cm}
\begin{longtable}{p{0.07\textwidth - 2\tabcolsep}>{\raggedright\arraybackslash}p{0.18\textwidth-2\tabcolsep}>{\raggedright\arraybackslash}p{0.1\textwidth-\tabcolsep}@{}}

\caption{Maximal connected overgroups for $G$-irreducible $A_1$ subgroups. \label{A1overgroupstab}} \\

\hline

$G$ & $p=3$ & $p=2$ \\

\hline

$G_2$ & $A_1 \tilde{A}_1$ & $A_1 \tilde{A}_1$ \\

$F_4$ & $B_4$ and $\bar{A}_1 C_3$ & $B_4$ \\

$E_6$ & $C_4$ & \textbf{---} \\

$E_7$ & $\bar{A}_1 D_6$ & $\bar{A}_1 D_6$ \\

$E_8$ & $D_8$ & $D_8$ \\
 
\hline

\end{longtable}

The choice of $M$ is not unique; for example, if $G = F_4$ and $p=2$ then $C_4$ also contains a representative of every $G$-conjugacy class of $G$-irreducible subgroup $A_1$. We also note that in larger characteristics more reductive, maximal subgroups are required. In particular, when $p \geq 19$ we need seven such subgroups for $G = E_7$. 

The next corollary shows that the $G$-conjugacy class of a $G$-irreducible subgroup $A_1$ is determined by its composition factors on the adjoint module for $G$. This is similar to \cite[Theorems 4, 6]{LS3} and extends part of Theorem \ref{A1samecomp} to low characteristics for irreducible $A_1$ subgroups. 

We must first explain a definition we will use throughout the paper. Let $X$ and $Y$ be semisimple subgroups of a semisimple algebraic group $G$ and let $V$ be a $G$-module. Then we say that $X$ and $Y$ \emph{have the same composition factors on $V$} if there exists an isomorphism from $X$ to $Y$ sending the set of composition factors of $V \downarrow X$ to the set of composition factors $V\downarrow Y$ (counted with multiplicity).

\begin{cor*} \label{corconj}
Let $G$ be an exceptional algebraic group and $X$ and $Y$ be irreducible $A_1$ subgroups. If $X$ and $Y$ have the same composition factors on $L(G)$ then $X$ is conjugate to $Y$. 
\end{cor*}

We also deduce that the $G$-conjugacy class of a simple connected subgroup of $G$ is determined by its composition factors on a smallest dimensional non-trivial module for $G$, which we will abbreviate to ``minimal module'' throughout. The dimensions of such a module are $7$ ($6$ if $p=2$), $26$ ($25$ if $p=3$), $27$, $56$ and $248$ for $G = G_2$, $F_4$, $E_6$, $E_7$ and $E_8$, respectively. 

\begin{cor*} \label{corconjmin}
Let $G$ be an exceptional algebraic group and $X$ and $Y$ be irreducible $A_1$ subgroups. If $X$ and $Y$ have the same composition factors on a minimal module for $G$ then $X$ is conjugate to $Y$. 
\end{cor*}

The next corollary lists some of the interesting $A_1$ subgroups that are $M$-irreducible but not $G$-irreducible for some reductive, maximal connected subgroup $M$. Here ``interesting'' means that the $M$-irreducible subgroup is not obviously $G$-reducible, i.e. $M'$-reducible for some other reductive, maximal connected subgroup $M'$ or contained in a proper Levi subgroup. 

To describe one of the subgroups we first define a piece of notation from \cite{tho1}. Suppose $G=E_8$ and $p=2$. There are two $D_8$-conjugacy classes of $B_4$ subgroups in $D_8$ acting irreducibly on the natural module for $D_8$. Since $p=2$, one is $E_8$-irreducible (by \cite[Lemma 7.5]{tho1}) and denoted by $B_4(\dagger)$ and the other is $E_8$-reducible (by \cite[Lemma 7.4]{tho1}) and denoted by $B_4(\ddagger)$. Furthermore, there are two $D_8$-conjugacy classes of $A_1$ subgroups acting irreducibly on the natural module for $D_8$ and one of these is contained in $B_4(\dagger)$ and the other in $B_4(\ddagger)$; thus we can differentiate them by giving the $B_4$ overgroup they are contained in. 

\begin{cor*} \label{nongcr}
Let $G$ be an exceptional algebraic group and $X$ be a subgroup $A_1$ of $G$. Suppose that whenever $X$ is contained in a reductive, maximal connected subgroup $M$ it is $M$-irreducible and assume that such an overgroup $M$ exists. Assume further that $X$ is not contained in a proper Levi subgroup of $G$. Then either:
\begin{enumerate}[leftmargin=*,label=\normalfont(\arabic*)]
\item $X$ is $G$-irreducible, or
\item $X$ is conjugate to a subgroup in Table \ref{cortab} below. Such $X$ are non-$G$-cr and satisfy the hypothesis.  
\end{enumerate}
\end{cor*}

\begin{longtable}{p{0.07\textwidth - 2\tabcolsep}p{0.16\textwidth - 2\tabcolsep}p{0.09\textwidth - 2\tabcolsep}>{\raggedright\arraybackslash}p{0.68\textwidth-\tabcolsep}@{}}

\caption{Non-$G$-cr subgroups that are irreducible in every (and at least one) maximal, reductive overgroup \label{cortab}} \\

\hline 

$G$ & Max. sub. $M$ & $p$ & $M$-irreducible subgroup $X$ \\

\hline 

$G_2$ & $A_1 \tilde{A}_1$ & $p=2$ & $A_1 \hookrightarrow A_1 \tilde{A}_1$ via $(1,1)$ \\

\hline

$E_7$ & $A_7$ & $p=2$ & $V_{A_7}(\lambda_1) \downarrow A_1 = 1 \otimes 1^{[r]} \otimes 1^{[s]}$ $(0 < r < s)$ \\

& $A_1 G_2$ & $p=7$ & $A_1 \hookrightarrow A_1 A_1 < A_1 G_2$ via $(1,1)$ where $A_1 < G_2$ is maximal \\

\hline

$E_8$ & $D_8$ & $p=2$ & $A_1 < B_4 (\ddagger)$ where $V_{D_8}(\lambda_1) \downarrow A_1 = 1 \otimes 1^{[r]} \otimes 1^{[s]} \otimes 1^{[t]}$ $(0<r<s<t)$ \newline and $V_{B_4}(\lambda_1) \downarrow A_1 = 2 \oplus 2^{[r]} \oplus 2^{[s]} \oplus 2^{[t]}$  \\

\hline

\end{longtable}

The notation in the fourth column of Table \ref{cortab} is explained in Section \ref{nota}. 

\section{Notation} \label{nota}

Let $G$ be a simple algebraic group over an algebraically closed field $K$ of characteristic $p$. Let $\Phi$ be the root system of $G$ and $\Phi^+$ be the set of positive roots in $\Phi$. Write $\Pi = \{ \alpha_1, \ldots, \alpha_l \}$ for the simple roots of $G$ and $\lambda_1, \ldots, \lambda_l$ for the fundamental dominant weights of $G$, both with respect to the ordering of the Dynkin diagram as given in \cite[p. 250]{bourbaki}. We sometimes use $a_1 a_2 \ldots a_l$ to denote a dominant weight $a_1 \lambda_1 + a_2 \lambda_2 + \cdots + a_l \lambda_l$. We denote by $V_G(\lambda)$ (or just $\lambda$) the irreducible $G$-module of dominant high weight $\lambda$. Similarly, the Weyl module of high weight $\lambda$ is denoted $W(\lambda) = W_G(\lambda)$ and the tilting module of highest weight $\lambda$ is denoted by $T(\lambda)$. Another module we refer to frequently is the adjoint module for $G$, which we denote $L(G)$. We let $V_7 : = W_{G_2}(10)$, $V_{26}:=W_{F_4}(0001)$, $V_{27} := V_{E_6}(\lambda_1)$ and $V_{56}:=V_{E_7}(\lambda_7)$. We note that $V_{7}$ ($V_{26}$) is irreducible unless $p=2$ ($p=3$). For $G$-modules $V$, $W$ we write $V + W$ for the module $V \oplus W$ and let $V^*$ denote the dual module of $V$. If $Y = Y_1 Y_2 \ldots Y_k$, a commuting product of simple algebraic groups, then $(V_1, \ldots, V_k)$ denotes the $Y$-module $V_1 \otimes \dots \otimes V_k$  where each $V_i$ is an irreducible $Y_i$-module. The notation $\bar{X}$ denotes a subgroup of $Y$ that is generated by long root subgroups of $Y$. If $Y$ has short root elements then $\tilde{X}$ means $\tilde{X}$ is generated by short root subgroups.  

Now suppose $p > 0$. Let $F: G \rightarrow G$ be the standard Frobenius endomorphism (acting on root groups $U_\alpha = \{ u_\alpha(c) | c \in K\}$ by $u_{\alpha}(c) \mapsto u_{\alpha}(c^p)$) and $V$ be a $G$-module afforded by a representation $\rho: G \rightarrow \text{GL}(V)$. Then $V^{[r]}$ denotes the module afforded by the representation $\rho^{[r]} : = \rho \circ F^r$. Let $M_1, \ldots, M_k$ be $G$-modules and $n_1, \ldots, n_k$ be positive integers. Then $M_1^{n_1} / \ldots / M_k^{n_k}$ denotes a $G$-module having the same composition factors as $M_1^{n_1} \oplus \dots \oplus M_k^{n_k}$. Furthermore, $V = M_1 | \ldots | M_k$ denotes a $G$-module with a socle series as follows: $M_k \cong \text{Soc}(V) = \text{Soc}^1(V)$ and for $i > 0$, we have $M_{k-i}$ is $\text{Soc}^{i+1}(V) = \text{Soc}(V/N_{i})$ where $N_{i}$ is the inverse image in $V$ of $\text{Soc}^{i}(V)$ under the quotient mapping $V \rightarrow V / N_{i-1}$ (so $N_0 = 0$ and $N_1 = M_k$). Sometimes, to make things clearer, we will use a tower of modules $$\cfrac{M_1}{\cfrac{M_2}{M_3}}$$ to mean the same as $M_1 | M_2 | M_3$.

We need a notation for diagonal subgroups of $Y = H_1 H_2 \ldots H_k$, a commuting product of subgroups of type $A_1$. Let $H$ be a simply connected subgroup $A_1$ and $\hat{Y} = H \times H \ldots \times H$, the direct product of $k$ copies of $H$. Then we may regard $Y$ as $\hat{Y} / Z$ where $Z$ is a subgroup of the centre of $\hat{Y}$ and $H_i$ is the image of the $i$th projection map. A diagonal subgroup of $\hat{Y}$ is a subgroup $\hat{X} \cong H$ of the following form: $\hat{X} = \{ (\phi_1(h), \ldots, \phi_k(h)) | h \in H \}$ where each $\phi_i$ is a non-trivial endomorphism of $H$. A diagonal subgroup $X$ of $Y$ is the image of a diagonal subgroup of $\hat{Y}$ under the natural map $\hat{Y} \rightarrow Y$. To describe such a subgroup it therefore suffices to give a non-trivial endomorphism, $\phi_i$, of $H$ for each $i$. By \cite[Chapter 12]{stei}, $\phi_i = \alpha \theta_i  F^{r_i}$ where $\alpha$ is an inner automorphism, $\theta_i$ is a graph morphism and $F^{r_i}$ is a power of the standard Frobenius endomorphism. We only wish to distinguish these diagonal subgroups up to conjugacy and therefore assume $\alpha$ is trivial. Moreover, there are no non-trivial graph automorphisms of $A_1$. It therefore suffices to give a non-negative integer $r_i$ for each $i$. Such a diagonal subgroup $X$ is denoted ``$X \hookrightarrow H_1 H_2 \ldots H_k$ via $(1^{[r_1]}, 1^{[r_2]},  \ldots, 1^{[r_k]})$''. We often abbreviate this to ``$X$ via $(1^{[r_1]}, \ldots, 1^{[r_k]})$'' if the group $Y$ is clear. We note that to avoid any redundancies we always take the minimum of the $r_1, \dots, r_k$ to be zero. 

Now let $G$ be a simple exceptional algebraic group. In Tables \ref{G2tab} to \ref{E8tab} we give an ID number to each of the conjugacy classes of $G$-irreducible $A_1$ subgroups in Theorems \ref{thmG2} to \ref{thmE8}. The notation $G(\#a)$ (or simply $a$ if $G$ is clear from the context) means the $G$-irreducible subgroup corresponding to the ID number $a$. Sometimes $G(\#a)$ will refer to infinitely many conjugacy classes of $G$-irreducible subgroups. This only occurs for diagonal subgroups and the conjugacy class will depend on field twists $r_1, \ldots, r_k$. Sometimes we need to refer to a subset of the conjugacy classes that $G(\#a)$ represents, described by an ordered set of field twists $s_1, \ldots, s_k$ and this will be denoted by $G(\#a^{\{s_1, \ldots, s_k\}})$. Let us give a concrete example to make this clearer. Consider $G_2(\#1)$, the conjugacy classes of diagonal subgroups $A_1 \hookrightarrow A_1 \tilde{A}_1$ via $(1^{[r]},1^{[s]})$ ($rs=0$; $r \neq s$) (see Table \ref{G2tab}). Then the notation $G_2(\#1^{\{r,0\}})$ refers to the conjugacy classes with $s=0$ and the notation $G_2(\#1^{\{1,0\}})$ refers to the single conjugacy class $A_1 \hookrightarrow A_1 \tilde{A}_1$ via $(1^{[1]},1)$.

In Tables \ref{G2tab} to \ref{E8tab} we also need a notation to be able to describe $M$-irreducible $A_1$ subgroups $X$ of reductive, maximal connected subgroups $M$ of $G$. Suppose $M = M_1 M_2 \ldots M_r$, with each $M_i$ simple. If all of the factors are simple classical algebraic groups then we define $$V_M := V_{M_1}(\lambda_1) \otimes V_{M_2}(\lambda_1) \otimes \cdots \otimes V_{M_r}(\lambda_1)$$  and let $V_M \downarrow X$ be the usual restriction of the $M$-module $V_M$ to $X$. If $M = F_4$ then we define $V_M \downarrow X$ to be $F_4(\#a)$, where $a$ is the ID number of the subgroup $X$ of $F_4$. The final case we need to describe the notation for is $M = M_1 M_2$ where $M$ is one of $A_1 G_2, A_2 G_2$, $G_2 C_3$, $A_1 F_4$, $G_2 F_4$ or $A_1 E_7$. The projection of $X$ to both $M_1$ and $M_2$ is an $M_i$-irreducible subgroup $A_1$, say $X_i$ (by Lemma \ref{easy}) and therefore $X$ is a diagonal subgroup of $X_1 X_2$ via $(1^{[r_1]},1^{[r_2]})$. We need to give $X_1$, $X_2$ and the field twists $r_1$, $r_2$. If $M_i$ is classical then write $V_{M_i}(\lambda_1) \downarrow X_i$ for $X_i$ and otherwise write $M_i(\#a)$ for $X_i$, where $a$ is the ID number of $X_i$ in $M_i$. Then we define $V_M \downarrow X = (X_1^{[r]},X_2^{[s]})$. We make a slight modification if $X_1$ (or similarly $X_2$ but not both) is a diagonal subgroup of some subgroup $Y$ of $M_1$, which is of exceptional type, so  $X_1 = M_1(\#a) < Y$ via $(1^{[s_1]}, \ldots 1^{[s_k]})$. In this case, $X$ is a diagonal subgroup of $Y X_2$ and we define $V_M \downarrow X = (M_1(\#a^{\{s_1, \ldots, s_k\}}),X_2^{[s_{k+1}]})$. 

 Again, let us give a concrete example to make this clearer. Suppose $G = E_7$ and $M = A_1 F_4$. Then $F_4$ has a maximal subgroup $A_1$ when $p \geq 13$, which is of course $F_4$-irreducible and denoted by $F_4(\#10)$. Letting the factor $A_1$ of $M$ be $X_1$ and the maximal subgroup of $F_4$ be $X_2$ we have an $M$-irreducible subgroup $X \hookrightarrow X_1 X_2$ via $(1^{[r]},1^{[s]})$ $(rs=0)$. Then $V_M \downarrow X = (1^{[r]},F_4(\#10)^{[s]})$. The notation changes slightly when we consider another $F_4$-irreducible subgroup. Let $X_1$ be as before but this time let $X_2$ be the subgroup $F_4(\#8)$, i.e. $A_1 \hookrightarrow A_1 A_1 < A_1 C_3$ via $(1^{[u]},1^{[v]})$ ($p \geq 7$; $uv=0$) where the second $A_1$ factor is maximal in $C_3$. Then $X \hookrightarrow X_1 X_2$ via $(1^{[r]},1^{[w]})$ $(rw=0)$ is $M$-irreducible and represents $X \hookrightarrow X_1 A_1 A_1$ via $(1^{[r]},1^{[s]},1^{[t]})$ $(rst=0)$ where $s = u+w$ and $t = v+w$. We then write $V_M \downarrow X = (1^{[r]},F_4(\#8^{\{s,t\}}))$ $(rst=0)$.  
 
Let $J = \{\alpha_{j_1}, \alpha_{j_2}, \ldots, \alpha_{j_r}\} \subseteq \Pi$ and define $\Phi_J = \Phi \cap \mathbb{Z}J$. Then the standard parabolic subgroup corresponding to $J$ is the subgroup $P = \left< T, U_\alpha \, : \, \alpha \in \Phi_J \cup \Phi^+ \right>$. The Levi decomposition of $P$ is $P = Q L$ where $Q = R_{u}(P) = \left< U_\alpha \, | \, \alpha \in \Phi^+ \setminus \Phi_J \right>$, and $L = \left< T, U_\alpha \, | \, \alpha \in \Phi_J \right>$. For $i \geq 1$ we define
\[ Q(i) = \left< U_\alpha \, \middle| \, \alpha = \sum_{j \in \Pi} c_j \alpha_j \mathrm{ \ where \ } \sum_{j \in \Pi \setminus J} c_j \geq i \right>, \]
which is a subgroup of $Q$. The \emph{$i$th level} of $Q$ is $Q(i)/Q(i+1)$, and this is central in $Q/Q(i+1)$.

\section{Preliminaries} \label{prelims}

Let $G$ be a simple algebraic group over an algebraically closed field of characteristic $p$. The first result needed for the proofs of Theorems \ref{thmG2} to \ref{thmE8} is the classification of reductive, maximal connected subgroups of exceptional algebraic groups. 

\begin{thm}[{\cite[Corollary 2]{LS1}}] \label{maximalexcep}
The following tables give the conjugacy classes of reductive, maximal connected subgroups $M$ for $G$ a simple exceptional algebraic group. We also give the composition factors of the restrictions to $M$ of $V_7$, $V_{26}$, $V_{27}$, $V_{56}$ and $L(G)$. 
\end{thm}
\pagebreak
$G = G_2$

\vspace{-0.15cm}
\begin{longtable*}{p{0.16\textwidth - 2\tabcolsep}>{\raggedright\arraybackslash}p{0.32\textwidth-2\tabcolsep}>{\raggedright\arraybackslash}p{0.52\textwidth-\tabcolsep}@{}}

\hline

$M$ & Comp. factors of $V_7 \downarrow M$ & Comp. factors of $L(G_2) \downarrow M$ \\

\hline

$A_2$ & $10 / 01 / 00$ & $W(11) / 10 / 01$ \\

$\tilde{A}_2$ $(p=3)$ & $11$ & $11 / 30 / 03 / 00$ \\

$A_1 \tilde{A}_1$ & $(1,1) / (0,W(2))$ & $(W(2),0) / (0,W(2)) / (1,W(3))$ \\

$A_1$ $(p \geq 7)$ & 6 & $W(10) / 2$ \\

\hline

\end{longtable*}

$G = F_4$
\vspace{-0.15cm}
\begin{longtable*}{p{0.16\textwidth - 2\tabcolsep}>{\raggedright\arraybackslash}p{0.32\textwidth-2\tabcolsep}>{\raggedright\arraybackslash}p{0.52\textwidth-\tabcolsep}@{}}

\hline
$M$ & Comp. factors of $V_{26} \downarrow M$ & Comp. factors of $L(F_4) \downarrow M$ \\

\hline
$B_4$ & $W(1000) / 0001 / 0000$ & $W(0100) / 0001$ \\

$C_4$ $(p=2)$ & $0100$ & $2000 / 0100 / 0001 / 0000^2$ \\

$\bar{A}_1 C_3$ $(p \neq 2)$ & $(1,100) / (0,W(010))$ & $(2,000) / (0,200) / (1,001)$ \\

$A_1 G_2$ $(p \neq 2)$ & $(2,10) / (W(4),00)$ & $(2,00) / (0,W(01)) / (W(4),10)$ \\

$A_2 \tilde{A}_2$ & $(10,10) / (01,01) / (00,W(11))$ & $(W(11),00) / (00,W(11)) / (10,W(02)) / (01,W(20))$ \\

$G_2$ $(p=7)$ & $20$ & $01 / 11$ \\

$A_1$ $(p \geq 13)$ & $W(16) / 8$ & $W(22) / W(14) / 10 / 2$ \\

\hline
\end{longtable*}

$G = E_6$
\vspace{-0.15cm}
\begin{longtable*}{p{0.21\textwidth - 2\tabcolsep}>{\raggedright\arraybackslash}p{0.29\textwidth-2\tabcolsep}>{\raggedright\arraybackslash}p{0.5\textwidth-\tabcolsep}@{}}

\hline
$M$ & Comp. factors of $V_{27} \downarrow M$ & Comp. factors of $L(E_6) \downarrow M$ \\

\hline

$\bar{A}_1 A_5$ & $(1,\lambda_1) /$ $\!\! (0,\lambda_4)$ & $(W(2),0) /$ $\!\! (0,W(\lambda_1 + \lambda_5)) /$ $\!\! (1,\lambda_3)$ \\
 $F_4$ & $W(0001) /$ $\!\! 0000$ & $W(1000) /$ $\!\! W(0001)$ \\
 $C_4$ $(p \neq 2)$ & $0100$ & $2000 /$ $\!\! W(0001)$ \\
 $A_2^3$ & $(10,01,00) /$ $\!\! (00,10,01) /$ $\!\! (01,00,10)$ & $(W(11),00,00) /$ $\!\! (00,W(11),00) /$ $\!\! (00,00,W(11)) /$ $\!\! (10,10,10) /$ $\!\! (01,01,01)$ \\
 $A_2 G_2$ & $(10,W(10)) /$ $\!\! (W(02),00)$ & $(W(11),W(10)) /$ $\!\! (W(11),00) /$ $\!\! (00,W(01))$  \\
 $G_2$ ($2$ classes; $p \neq 7)$ & $W(20)$ & $W(01) /$ $\!\! W(11)$ \\
 $A_2$ ($2$ classes; $p \geq 5)$ & $W(22)$ & $11 /$ $\!\! 41 /$ $\!\! 14$ \\

\hline

\end{longtable*}

$G = E_7$
\vspace{-0.15cm}
\begin{longtable*}{p{0.16\textwidth - 2\tabcolsep}>{\raggedright\arraybackslash}p{0.34\textwidth-2\tabcolsep}>{\raggedright\arraybackslash}p{0.5\textwidth-\tabcolsep}@{}}

\hline
$M$ & Comp. factors of $V_{56} \downarrow M$ & Comp. factors of $L(E_7) \downarrow M$ \\

\hline
$\bar{A}_1 D_6$ & $(1,\lambda_1) /$ $\!\! (0,\lambda_5)$ & $(W(2),0) /$ $\!\! (0,W(\lambda_2)) /$ $\!\! (1,\lambda_6)$ \\
 $A_7$ & $\lambda_2 /$ $\!\! \lambda_6$ & $W(\lambda_1 + \lambda_7) /$ $\!\! \lambda_4$ \\
 $A_2 A_5$ & $(10,\lambda_1) /$ $\!\! (01,\lambda_5) /$ $\!\! (00,\lambda_3)$ & $(W(11),0) /$ $\!\! (00,W(\lambda_1 + \lambda_5)) /$ $\!\! (10,\lambda_4) /$ $\!\! (01,\lambda_2)$ \\
 $G_2 C_3$ & $(W(10),100) /$ $\!\! (00,W(001))$ & $(W(10),W(010)) /$ $\!\! (W(01),000) /$ $\!\! (00,W(200))$  \\
$A_1 G_2$ $(p \neq 2)$ & $(1,W(01)) /$ $\!\! (W(3),10)$ & $(2,00) /$ $\!\! (0,W(01)) /$ $\!\! (W(4),10) /$ $\!\! (2,W(20))$ \\
$A_1 F_4$ & $(1,W(0001)) /$ $\!\! (W(3),0000)$ & $(W(2),0000) /$ $\!\! (0,W(1000)) /$ $\!\! (W(2),W(0001))$  \\
 $A_2$ $(p \geq 5)$ & $W(60) /$ $\!\! W(06)$ & $W(44) /$ $\!\! 11$ \\
$A_1 A_1$ $(p \geq 5)$ & $(W(6),3) /$ $\!\! (4,1) /$ $\!\! (2,W(5))$ & $(2,0) /$ $\!\! (0,2) /$ $\!\! (2,W(8)) /$ $\!\! (W(6),4) /$ $\!\! (4,W(6)) /$ $\!\! (4,2) /$ $\!\! (2,4)$ \\
$A_1$ $(p \geq 17)$ & $W(21) / 15 / 11 / 5$ & $W(26) /$ $\!\! W(22) /$ $\!\! W(18) /$ $\!\! 16 /$ $\!\! 14 /$ $\!\! 10^2 /$ $\!\! 6 /$ $\!\! 2 $ \\
$A_1$ $(p \geq 19)$ & $W(27) / 17 / 9$  & $W(34) /$ $\!\! W(26) /$ $\!\! W(22) /$ $\!\! 18 /$ $\!\! 14 /$ $\!\! 10 /$ $\!\! 2$ \\
\hline

\end{longtable*}

$G = E_8$
\vspace{-0.15cm}
\begin{longtable*}{p{0.16\textwidth - 2\tabcolsep}>{\raggedright\arraybackslash}p{0.84\textwidth-\tabcolsep}@{}}

\hline
$M$ & Comp. factors of $L(E_8) \downarrow M$ \\

\hline

 $D_8$ & $W(\lambda_2) /$ $\!\! \lambda_7$  \\
 $A_8$ & $W(\lambda_1 + \lambda_8) /$ $\!\! \lambda_3 /$ $\!\! \lambda_5$ \\
$\bar{A}_1 E_7$ & $(W(2),0) /$ $\!\! (0,\lambda_1) /$ $\!\! (1, \lambda_7)$ \\
 $A_2 E_6$  & $(W(11),0) /$ $\!\!(00,W(\lambda_2)) /$ $\!\! (10,\lambda_6) /$ $\!\! (01,\lambda_1)$  \\
 $A_4^2$  & $(W(1001),0000) /$ $\!\! (0000,W(1001)) /$ $\!\! (1000,0100) /$ $\!\! (0001,0010) /$ $\!\! (0100,0001) /$ $\!\! (0010,1000)$ \\
 $G_2 F_4$  & $(W(10),W(0001)) /$ $\!\! (W(01),0000) /$ $\!\! (00,W(1000))$ \\
 $B_2$ $(p \geq 5)$  & $02 /$ $\!\! W(06) /$ $\!\! W(32)$ \\ 
$A_1 A_2$ $(p \geq 5)$ & $(W(6),11) /$ $\!\! (W(2),W(22)) /$ $\!\! (4,30) /$ $\!\! (4,03) /$ $\!\! (2,00) /$ $\!\! (0,11)$ \\
$A_1$ $(p \geq 23)$ & $W(38) /$ $\!\! W(34) /$ $\!\! W(28) /$ $\!\! W(26) /$ $\!\! 22^2 /$ $\!\! 18 /$ $\!\! 16 /$ $\!\! 14 /$ $\!\! 10 /$ $\!\! 6 /$ $\!\! 2$ \\
$A_1$ $(p \geq 29)$ & $W(46) /$ $\!\! W(38) /$ $\!\! W(34) /$ $\!\! 28 /$ $\!\! 26 /$ $\!\! 22 /$ $\!\! 18 /$ $\!\! 14 /$ $\!\! 10 /$ $\!\! 2$ \\
$A_1$ $(p \geq 31)$ & $W(58) /$ $\!\! W(46) /$ $\!\! W(38) /$ $\!\! W(34) /$ $\!\! 26 /$ $\!\! 22 /$ $\!\! 14 /$ $\!\! 2$ \\

\hline
\end{longtable*}

Note that in the cases in Theorem \ref{maximalexcep} where $M$ is of maximal rank, the composition factors are not given in \cite{LS1} but can be found in \cite[Lemmas 11.2(iii), 11.8, 11.10, 11.11, 11.12(ii)]{LS9}; moreover, for $(G,M) = (E_6,A_2^3)$, $(E_7, A_2 A_5)$, $(E_8,D_8)$ and $(E_8,A_4^2)$ we have made a choice of simple system within each factor. 

We next recall the algorithm of Borel and de Siebenthal, described in  \cite[14.2]{don}. The algorithm describes a way of using the extended Dynkin diagram to find subsystem subgroups of $G$. When $G$ is exceptional it finds all such subsystem subgroups unless $(G,p) = (G_2,3)$ or $(F_4,2)$, in which case certain subgroups containing short root subgroups need to be added. The subgroups of maximal rank in Theorem \ref{maximalexcep} come from this algorithm, although a maximal rank subsystem subgroup need not be maximal amongst connected subgroups. For example, we have the following inclusion of maximal rank subsystem subgroups $A_1^8 < A_1^4 D_4 <  D_4^2 < D_8 < E_8$. Throughout the proofs of Theorems \ref{thmG2} to \ref{thmE8} we will implicitly make use of this algorithm to describe maximal rank subsystem subgroups. 

For the following lemmas let $G$ be a semisimple connected algebraic group. We describe some elementary results about $G$-irreducible subgroups. 

\begin{lem} [{\cite[Lemma 2.1]{LT}}] \label{semirr}
If $X$ is a $G$-irreducible connected subgroup of $G$, then $X$ is semisimple and $C_G(X)$ is a finite subgroup.   
\end{lem}

\begin{lem} [{\cite[Lemma 3.6]{tho1}}] \label{easy}
Suppose a $G$-irreducible connected subgroup $X$ is contained in $K_1 K_2$, a commuting product of connected non-trivial subgroups $K_1$, $K_2$ of $G$. Then $X$ has a non-trivial projection to both $K_1$ and $K_2$. Moreover, each projection is a $K_i$-irreducible subgroup. 
\end{lem}

We now need some results that allow us to deduce whether or not a subgroup  $A_1$ of $G$ is $G$-irreducible. The first result allows us to do that when $G$ is a classical simple group. Recall that if $G$ is not of type $A_n$, then $G$ has a natural non-degenerate bilinear form on $V_{G}(\lambda_1)$ (noting that we are factoring out the radical when $(G,p)$ is $(B_n,2)$) and the notation $V = V_1 \perp V_2$ denotes an orthogonal decomposition with respect to this form. 

\begin{lem}[{\cite[Lemma 2.2]{LT}}] \label{class}
Suppose $G$ is a classical simple algebraic group, with natural module $V = V_G(\lambda_1)$. Let $X$ be a semisimple connected closed subgroup of $G$. If $X$ is $G$-irreducible then one of the following holds:

\begin{enumerate}[label=\normalfont(\roman*),leftmargin=*,widest=iii, align=left]

\item $G = A_n$ and $X$ is irreducible on $V$. 

\item $G = B_n, C_n$ or $D_n$ and $V \downarrow X = V_1 \perp \ldots \perp V_k$ with the $V_i$ all non-degenerate, irreducible and inequivalent as $X$-modules.  

\item $G = D_n$, $p=2$, $X$ fixes a non-singular vector $v \in V$, and $X$ is a $G_v$-irreducible subgroup of $G_v = B_{n-1}$.
\end{enumerate}
\end{lem}

Applications of this lemma often implicitly invoke some facts about the representation theory of $X$. For instance, suppose that $X$ is of type $A_1$ and that $n < p$. Then we implicitly use that the dimension of $V = V_{X}(n)$ is $n+1$; and that $X$ preserves a symplectic form on $V$ when $n$ is odd and preserves an orthogonal form when $n$ is even.     

The next lemma and corollary are used heavily in the proofs of Theorems \ref{thmG2} to \ref{thmE8} to show a subgroup is $G$-irreducible for a simple exceptional algebraic group $G$. 

\begin{lem}  [{\cite[Lemma 3.8]{tho1}}] \label{wrongcomps}
Let $X$ be a semisimple connected subgroup of $G$ and let $V$ be a $G$-module. Suppose that $X$ does not have the same composition factors as any semisimple connected subgroup $H$ of the same type as $X$ with $H \leq L'$ and $L$-irreducible, for some proper Levi subgroup $L$ of $G$. If $X$ is of type $B_n$ and $p=2$ then assume further that there is no subgroup $H$ of type $C_n$ with $H \leq L'$ and $L$-irreducible, for some Levi subgroup $L$ of $G$, such that there is an isogeny $\phi: X \rightarrow H$ inducing a mapping which takes the composition factors of $V \downarrow X$ to those of $V \downarrow H$. Then $X$ is $G$-irreducible.  
\end{lem}

\begin{cor}  [{\cite[Corollary 3.9]{tho1}}] \label{notrivs}
Suppose $X < G$ is semisimple and $L(G) \downarrow X$ has no trivial composition factors. Then $X$ is $G$-irreducible. 
\end{cor}

\begin{proof}
Suppose $X$ is $G$-reducible. Then by Lemma \ref{wrongcomps} (with $V = L(G)$) there exists a subgroup $H$ of some Levi subgroup $L_1$ such that the composition factors of $L(G) \downarrow H$ are the same as $L(G) \downarrow X$.  But $H < L_1$, so $L(G) \downarrow H$ has trivial composition factors coming from $L(Z(L_1))$, a contradiction. 
\end{proof}

The next result is \cite[Prop. 1.4]{LS7} with $S$ allowed to be any closed subgroup of $X$; the proof is the same. 

\begin{lem} \label{subspaces}
Let $X$ be a linear algebraic group over $K$ and let $S$ be a closed subgroup of $X$. Suppose $V$ is a finite-dimensional $X$-module satisfying the following conditions:

\begin{enumerate}[leftmargin=*,label=\normalfont(\roman*),widest=iii, align=left]
\item every $X$-composition factor of $V$ is an irreducible $S$-module;

\item for any $X$-composition factors $M$, $N$ of $V$, the restriction map $\mathrm{Ext}^1_X(M,N) \rightarrow \mathrm{Ext}^1_S(M,N)$ is injective;

\item for any $X$-composition factors $M,N$ of $V$, if $M \downarrow S \cong N \downarrow S$, then $M \cong N$ as $X$-modules.  

\end{enumerate}
Then $X$ and $S$ fix precisely the same subspaces of $V$. 
\end{lem}

The following well-known result will be used implicitly to prove certain extensions of $A_1$-modules exist. 

\begin{lem}[{\cite[Corollary 3.9]{Andersen}}] \label{h1fora1}
Suppose $X$ is an algebraic group of type $A_1$ and that $M$ is an irreducible $X$-module such that $H^{1}(X,M) \neq 0$. Then $M$ is a Frobenius twist of $(p-2) \otimes 1^{[1]}$ and $H^1(X,M) \cong K$.
\end{lem}

When considering the conjugacy class of a subgroup $A_1$ in an exceptional algebraic group, the following result is useful. We define a prime number $N(A_1,G)$ for each exceptional algebraic group $G$ as in the table below.  

\begin{longtable*}{>{\raggedleft\arraybackslash}p{0.1\textwidth - 2\tabcolsep}>{\raggedleft\arraybackslash}p{0.06\textwidth - 2\tabcolsep}>{\raggedleft\arraybackslash}p{0.06\textwidth - 2\tabcolsep}>{\raggedleft\arraybackslash}p{0.06\textwidth - 2\tabcolsep}>{\raggedleft\arraybackslash}p{0.06\textwidth - 2\tabcolsep}>{\raggedleft\arraybackslash}p{0.06\textwidth-\tabcolsep}@{}}

\hline

$G$ & $G_2$ & $F_4$ & $E_6$ & $E_7$ & $E_8$ \\

$N(A_1,G)$ & 3 & 3 & 5 & 7 & 7 \\

\hline

\end{longtable*}

\begin{thm}[{\cite[Theorem 4]{law}}] \label{A1samecomp}
Let $G$ be an exceptional algebraic group in characteristic $p > N(A_1,G)$ and $X_1$ and $X_2$ be $A_1$ subgroups of $G$ that have the same composition factors on $L(G)$. Then $X_1$ and $X_2$ are $G$-conjugate.  
\end{thm}

In the proofs of Theorems \ref{thmE7} and \ref{thmE8}, we use Lemma \ref{wrongcomps} to prove $A_1$ subgroups are $G$-irreducible when $p=2$ for $G = E_7$ and $E_8$. To do this, we need to know the $L'$-irreducible $A_1$ subgroups of Levi subgroups $L$ of $G$ when $p=2$, which we list in the following lemma. For a Levi subgroup $L$ such that all factors of $L' = L_1 \dots L_m$ are classical, we define $V_L$ to be the module $V_{L_1}(\lambda_1) \otimes \cdots \otimes V_{L_m}(\lambda_1)$.  

\begin{lem} \label{bada1p2}

Let $G=E_7$ or $E_8$ with $p=2$ and let $L$ be a proper Levi subgroup of $G$. Then the following table contains each $L$-irreducible subgroup $A_1$. 

\end{lem}

\begin{longtable}{p{0.1\textwidth - 2\tabcolsep}>{\raggedright\arraybackslash}p{0.9\textwidth-1\tabcolsep}}

\caption{$L$-irreducible $A_1$ subgroups of Levi factors $L$ of $E_7$ and $E_8$ when $p=2$} \\

\hline \noalign{\smallskip}

Levi $L'$ & $V_L \downarrow A_1$ \\

\hline \noalign{\smallskip}

$E_7$ & see Theorem \ref{thmE7} \\

$D_7$ & $0|(2 + 2^{[r]} + 2^{[s]} + 2^{[t]} + 2^{[u]} + 2^{[v]})|0$ ($0 < r < s < t < u < v$) \\

& $0|(2^{[r]} + 2^{[s]})|0 + 1^{[t]} \otimes 1^{[u]} \otimes 1^{[v]}$ ($rt=0$; $r < s$; $t < u < v$) \\

& $0|(2^{[r]} + 2^{[s]})|0 + 1^{[t]} \otimes 1^{[u]} + 1^{[v]} \otimes 1^{[w]}$ ($rt=0$; $r < s$; $t < u$; $t \leq v$; $v < w$)  \\

$D_6$ & $0|(2 + 2^{[r]} + 2^{[s]} + 2^{[t]} + 2^{[u]})|0$ ($0< r < s < t < u$) \\

& $0|(2^{[r]} + 2^{[s]} + 2^{[t]})|0 + 1^{[u]} \otimes 1^{[v]}$ ($ru=0$; $r < s < t$; $u < v$) \\

& $1^{[r]} \otimes 1^{[s]} \otimes 1^{[t]} + 1^{[u]} \otimes 1^{[v]}$  ($ru=0$; $r < s < t$; $u \leq v$) \\

& $1 \otimes 1^{[r]} + 1^{[s]} \otimes 1^{[t]} + 1^{[u]} \otimes 1^{[v]}$ ($s \leq \text{min}\{t,u,v\}$; $|\{0,r\}| + |\{s,t\}| + |\{u,v\}| \geq 5$; $|\{\{0,r\}, \{s,t\}, \{u,v\}\}|=3$)  \\

$A_1 D_5$ & $(1^{[r]}, 0|(2^{[s]} + 2^{[t]} + 2^{[u]} + 2^{[v]})|0)$ ($rs=0$; $s < t < u < v$)  \\

$D_5$ & $0|(2 + 2^{[r]} + 2^{[s]} + 2^{[t]})|0$ ($0 < r < s < t$)  \\

$A_1 D_4$ & $(1^{[r]}, 0|(2^{[s]} + 2^{[t]} + 2^{[u]})|0)$ ($rs=0$; $s < t < u$) \\

& $(1^{[r]}, 1^{[s]} \otimes 1^{[t]} \otimes 1^{[u]})$ ($rs=0$; $s < t < u$) \\

& $(1^{[r]}, 1^{[s]} \otimes 1^{[t]} + 1^{[u]} \otimes 1^{[v]})$ ($rs=0$; $s < t$; $u \leq v$; if $s=u$ then $t < v$) \\

$D_4$ & $0|(2 + 2^{[r]} + 2^{[s]})|0$ ($0 < r < s$)\\

& $1 \otimes 1^{[r]} \otimes 1^{[s]} $ ($2$ classes; $0 < r < s$) \\

& $1 \otimes 1^{[r]} + 1^{[s]} \otimes 1^{[t]}$ ($0 < r$; $s \leq t$; if $s=0$ then $r < t$) \\

$A_7$ & $1 \otimes 1^{[r]} \otimes 1^{[s]}$ ($0 < r < s$) \\

$A_3^2$ & $(1^{[r]} \otimes 1^{[s]}, 1^{[t]} \otimes 1^{[u]})$ ($rt=0$; $r < s$; $t < u$) \\

$A_1^2 A_3$ & $(1^{[r]},1^{[s]}, 1^{[t]} \otimes 1^{[u]})$ ($rst=0$; $t < u)$ \\

$A_1 A_3$ & $(1^{[r]},1^{[s]} \otimes 1^{[t]})$ $(s < t)$ \\

$A_1^4$ & $(1^{[r]},1^{[s]},1^{[t]},1^{[u]})$ $(rstu=0)$  \\

$A_3$ & $1 \otimes 1^{[r]}$ $(r \neq  0)$  \\

$A_1^3$ & $(1^{[r]},1^{[s]},1^{[t]})$ $(rst=0)$   \\

$A_1^2$ & $(1^{[r]},1^{[s]})$ $(rs=0)$ \\

$A_1$ & 1 \\

\hline

\end{longtable}

\begin{proof}

Let $L$ be a Levi subgroup of $G$. Write $L' = L_1 \dots L_m$ where each $L_i$ is a simple factor. Given the $L_i$-irreducible subgroups of type $A_1$, then all $L'$-irreducible subgroups of type $A_1$ follow, since they are just diagonal subgroups (by Lemma \ref{easy}). We therefore give a brief description of the $L_i$-irreducible subgroups of type $A_1$ to conclude the proof. 

Suppose $L_i$ is of classical type. We use Lemma \ref{class} to find all $L_i$-irreducible $A_1$ subgroups. First let $L_i \cong A_n$, then $V_n := V_{A_n}(\lambda_1) \downarrow A_1$ is irreducible. By Steinberg's Tensor Product Theorem, it follows that $V_n \downarrow A_1 = 1^{[r_1]} \otimes 1^{[r_2]} \otimes \dots \otimes 1^{[r_l]}$ for distinct $r_1, \ldots, r_l$, since $p=2$. Therefore, $A_n$ has an $A_n$-irreducible subgroup $A_1$ if and only if $n+1 = 2^{[k]}$ for some $k \geq 1$. Now let $L_i \cong D_n$ ($4 \leq n \leq 7$). In all cases $L_i$ has an $L_i$-irreducible subgroup $A_1$, acting as $0 | (2^{[r_1]} + \dots + 2^{[r_{n-1}]}) | 0$ on $V_{D_n}(\lambda_1)$, coming from part (iii) of Lemma \ref{class}. If $n=4$ or $6$ then $L_i$ has an $L_i$-irreducible subgroup $A_1$ acting as $1^{[r_1]} \otimes 1^{[r_2]} + \dots + 1^{[r_{n-1}]} \otimes 1^{[r_n]}$ on $V_{D_n}(\lambda_1)$. Finally, if $n=4$ then there are two further classes of $L_i$-irreducible $A_1$ subgroups, acting as $1 \otimes 1^{[r]} \otimes 1^{[s]}$ on $V_{D_4}(\lambda_1)$. 

Now suppose $L_i$ is of exceptional type and hence isomorphic to $E_6$ or $E_7$. We use Theorem \ref{thmE6} and \ref{thmE7}, respectively, to find the $L_i$-irreducible $A_1$ subgroups. We note that we are permitted to do this since we prove Theorems \ref{thmE6} to \ref{thmE8} successively and so are only using Theorem \ref{thmE6} and \ref{thmE7} after they have been proved.  In particular, there are no $E_6$-irreducible $A_1$ subgroups when $p=2$. All $E_7$-irreducible $A_1$ subgroups are contained in $\bar{A}_1 D_6$ when $p=2$ and are listed in Table~\ref{E8tab}. 
\end{proof}

\section{Strategy for the proof of Theorem \ref{mainthm}} \label{strat}

To prove Theorem \ref{mainthm} we prove Theorems \ref{thmG2} to \ref{thmE8} in Sections \ref{secG2} to \ref{secE8}, respectively and successively. In this section we describe the strategy used in proving Theorems \ref{thmG2} to \ref{thmE8}. Let $G$ be an exceptional algebraic group over an algebraically closed field of characteristic $p$. Suppose $X$ is a $G$-irreducible subgroup $A_1$ of $G$. Then $X$ is contained in a maximal connected subgroup $M$ of $G$. Since $X$ is $G$-irreducible, $M$ is reductive. Furthermore, $X$ is $M$-irreducible as any parabolic subgroup of $M$ is contained in a parabolic subgroup of $G$ by the Borel-Tits Theorem \cite{BT}. Therefore, $X$ is contained $M$-irreducibly in some reductive, maximal connected subgroup $M$ of $G$ and the following strategy finds all such $X$. 

Take a reductive, maximal connected subgroup $M$ from Theorem \ref{maximalexcep} and find all $M$-irreducible $A_1$ subgroups of $M$, up to $M$-conjugacy. To do this we use Lemma \ref{class} for classical simple components of $M$, and Theorems \ref{thmG2} to \ref{thmE7} for exceptional simple components of $M$ of smaller rank than $G = F_4, E_6, E_7, E_8$. For each class of $M$-irreducible $A_1$ subgroups $X$ we then check whether there exists another reductive, maximal connected subgroup containing $X$ that we have already considered. If there is, then we have already considered $X$ and are done. If not, we must then decide whether $X$ is $G$-irreducible or not. To do this we  heavily use Lemma \ref{wrongcomps} and Corollary \ref{notrivs}. To apply these results we must find the composition factors of the action of $X$ on the minimal or adjoint module. These can be found by restricting the composition factors of $M$ to $X$. This can be done for all $M$-irreducible $A_1$ subgroups and the composition factors for the $G$-irreducible ones can be found in Section \ref{tabs}, Tables \ref{G2tabcomps} to \ref{E8tabcomps}. To apply Lemma \ref{wrongcomps} we also need the composition factors for the action of the Levi subgroups of $G$ on the minimal and adjoint modules. These can be found in Appendix \ref{app}, Tables \ref{levig2} to \ref{levie8}. In most cases, an $M$-irreducible subgroup $A_1$ is $G$-irreducible: Corollary \ref{nongcr} lists the subgroups which are irreducible in every reductive, maximal connected overgroup yet $G$-reducible. To prove an $M$-irreducible subgroup $X$ is $G$-reducible can be difficult and can require precise knowledge of the action of $X$ on the minimal or adjoint module for $G$ as well as computation in Magma \cite{magma}.  

\section{Proof of Theorem \ref{thmG2}: $G_2$-irreducible $A_1$ subgroups} \label{secG2}

In this section we find the irreducible $A_1$ subgroups of $G_2$, proving Theorem \ref{thmG2} below. We note that Theorem \ref{thmG2} is \cite[Theorem 5.4]{bon} and can also be deduced from \cite[Theorem 1]{g2dav}. We give a proof for completeness and also to show how the strategy described in Section \ref{strat} works. Recall the notation $V_M$ from Section \ref{nota}. 

\begin{thm*} \label{thmG2}
Suppose $X$ is an irreducible subgroup $A_1$ of $G_2$. Then $X$ is conjugate to exactly one subgroup of Table \ref{G2tab} and each subgroup in Table \ref{G2tab} is irreducible. 
\end{thm*}

\begin{longtable}{>{\raggedright\arraybackslash}p{0.05\textwidth - 2\tabcolsep}>{\raggedright\arraybackslash}p{0.18\textwidth - 2\tabcolsep}>{\raggedright\arraybackslash}p{0.5\textwidth - 2\tabcolsep}>{\raggedright\arraybackslash}p{0.07\textwidth-\tabcolsep}@{}}

\caption{The $G_2$-irreducible $A_1$ subgroups of $G_2$ \label{G2tab}} \\

\hline

ID & $M$ & $V_M \downarrow X$ & $p$ \\

\hline

1 & $A_1 \tilde{A}_1$ & $(1^{[r]},1^{[s]})$ $(rs=0; r \neq s)$ & any \\

2 & & $(1,1)$ & $\geq 3$ \\

\hline

3 &  $A_1$ & 1 & $\geq 7$ \\

\hline

\end{longtable}

We note that Table \ref{G2tabcomps} gives the composition factors of $V_{7}$ and $L(G_2)$ restricted to each irreducible subgroup $A_1$ in Table \ref{G2tab}. 

\begin{proof}

The conjugacy classes of reductive, maximal connected subgroups $M$ of $G_2$ are listed in Theorem~\ref{maximalexcep}. They are $\bar{A}_2$, $\tilde{A}_2$ $(p=3)$, $A_1 \tilde{A}_1$ and $A_1$ $(p \geq 7)$. Let $X$ be an $M$-irreducible subgroup $A_1$ of $M$. 

First suppose $M = A_1 \tilde{A}_1$. By Lemma \ref{easy}, $X$ has non-trivial projection to both $A_1$ and $\tilde{A}_1$. Hence $X$ is a diagonal subgroup of $M$ and we have $X \hookrightarrow A_1 \tilde{A}_1$ via $(1^{[r]},1^{[s]})$ ($rs=0$; $r \neq s$) or $(1,1)$. The diagonal subgroups with distinct field twists are $G_2$-irreducible for all $p$. Indeed, to show $X$ via $(1^{[r]},1^{[s]})$ ($rs=0$; $r \neq s$) is $G_2$-irreducible we use Lemma \ref{wrongcomps}. From Table \ref{G2tabcomps}, the restriction of $V_{G_2}(10)$ to $X$ has a 4-dimensional composition factor, namely $1^{[r]} \otimes 1^{[s]}$ (since $r \neq s$). Neither a Levi subgroup $A_1$ nor a Levi subgroup $\tilde{A}_1$ has a 4-dimensional composition factor on $V_{G_2}(10)$ (the composition factors are listed in Table \ref{levig2}) and hence $X$ does not have the same composition factors as either Levi subgroup on $V_{G_2}(10)$. Therefore, $X$ is $G_2$-irreducible by Lemma \ref{wrongcomps}. 

Now consider $X$ via $(1,1)$. When $p > 3$, we see from Table \ref{G2tabcomps} that $X$ has no trivial composition factors on $L(G_2)$ and hence $X$ is $G_2$-irreducible by Corollary \ref{notrivs}. When $p=3$, we see that $X$ has two $3$-dimensional composition factors on $V_{G_2}(10)$. Neither $A_1$ nor $\tilde{A}_1$ has two $3$-dimensional composition factors on $V_{G_2}(10)$ and hence $X$ is $G_2$-irreducible by Lemma \ref{wrongcomps}. Finally, let $p=2$. Then $V_{G_2}(10) \downarrow A_1 \tilde{A}_1 = (1,1) + (0,2)$ and so $V_{G_2}(10) \downarrow X = (0|2|0) + 2$ and $X$ fixes a non-zero vector of $V_{G_2}(10)$. The stabiliser of this non-zero vector in $G_2$ is a parabolic subgroup. Indeed, $G_2$ is transitive on $1$-spaces of $V_{G_2}(10)$ by \cite[Theorem B]{lss} and the stabiliser of some $1$-space is a parabolic subgroup. Hence $X$ is $G_2$-reducible. 

Now suppose $M = \bar{A}_2$. The only irreducible subgroup $A_1$ of $\bar{A}_2$ is embedded via the representation with high weight $2$, when $p \neq 2$, by Lemma \ref{class} and therefore $X$ is such a subgroup $A_1$. By \cite[Table 10.3]{LS1}, we have $\bar{A}_2.2$ is contained in $G_2$. The subgroup $\bar{A}_2.2$ contains an involution $t$ such that $X$ is the centraliser of $t$ in $\bar{A}_2.2$. The centraliser in $G_2$ of $t$ is $A_1 \tilde{A}_1$, by \cite[Table~4.3.1]{gls3}. Therefore $X < A_1 \tilde{A}_1$ and has already been considered. In fact, $X$ is conjugate to $G_2(\#2)$.  

Now let $M = \tilde{A}_2$ $(p =3)$. Then as before, $X$ is embedded in $\tilde{A}_2$ via the representation of high weight $2$. The same argument as for $M = \bar{A}_2$ shows that $X$ is contained in $A_1 \tilde{A}_1$. In particular, $X$ is $G_2$-irreducible and conjugate to $G_2(\#1^{\{0,1\}})$. 

Finally, when $M = A_1$ $(p \geq 7)$ we have $X = M$ and hence $X$ is $G_2$-irreducible.  

To finish the proof of Theorem \ref{thmG2}, we use the composition factors in Table \ref{G2tabcomps} to check that $G_2(\#1)$, $G_2(\#2)$ and $G_2(\#3)$ are pairwise non-conjugate.  
\end{proof}

\section{Proof of Theorem \ref{thmF4}: $F_4$-irreducible $A_1$ subgroups} \label{secF4}

In this section, we classify all $F_4$-irreducible $A_1$ subgroups of $F_4$, proving Theorem \ref{thmF4}. 

\begin{thm*} \label{thmF4}
Suppose $X$ is an irreducible subgroup $A_1$ of $F_4$. Then $X$ is conjugate to exactly one subgroup of Table \ref{F4tab} and each subgroup in Table \ref{F4tab} is irreducible. 
\end{thm*}

\begin{longtable}{>{\raggedright\arraybackslash}p{0.05\textwidth - 2\tabcolsep}>{\raggedright\arraybackslash}p{0.18\textwidth - 2\tabcolsep}>{\raggedright\arraybackslash}p{0.5\textwidth - 2\tabcolsep}>{\raggedright\arraybackslash}p{0.07\textwidth-\tabcolsep}@{}}

\caption{The $F_4$-irreducible $A_1$ subgroups of $F_4$ \label{F4tab}} \\

\hline

ID & $M$ & $V_M \downarrow X$ & $p$ \\

\hline

1 & $B_4$ & $1 \otimes 1^{[r]} + 1^{[s]} \otimes 1^{[t]} + 0$ $(0 < r < s < t)$ & any \\

2 & & $2^{[r]} + 2^{[s]} + 1^{[t]} \otimes 1^{[u]}$  ($rt=0$; $r < s$; $t < u$) & $=2$ \\

3 & & $2 + 2^{[r]} + 2^{[s]} + 2^{[t]}$  ($0 <  r < s < t$) & $=2$  \\

4 & & $2 + 2^{[r]} + 2^{[s]}$  $(0 <  r < s)$ & $\geq 3$ \\

5 & & $2 \otimes 2^{[r]}$ $(r \neq 0)$ & $\geq 3$ \\

6 & & $4^{[r]} + 1^{[s]} \otimes 1^{[t]}$  ($rs=0$; $s  \leq t$) & $\geq 5$  \\

7 & & $8$ & $\geq 11$ \\ 

\hline

8 &  $\bar{A}_1 C_3$ $(p \neq 2)$ & $(1^{[r]},5^{[s]})$ ($rs=0$) &  $\geq 7$ \\

9 & & $(1^{[r]},2^{[s]} \otimes 1^{[t]})$ $(rst=0; s \neq t)$ & $\geq 3$ \\

\hline

10 & $A_1$ & $1$ & $\geq 13$ \\

\hline

11 & $A_1 G_2 $ & $(1^{[r]}, G_2(\#3)^{[s]})$  ($rs=0$; $r \neq s$) & $\geq 7$ \\

\hline

\end{longtable}

The composition factors of $V_{26}$ and $L(F_4)$ restricted to each irreducible subgroup $A_1$ in Table \ref{F4tab} are found in Table \ref{F4tabcomps}.

\begin{proof}

The conjugacy classes of reductive, maximal connected subgroups $M$ of $F_4$ are listed in Theorem~\ref{maximalexcep}. They are $B_4$, $C_4$ $(p=2)$, $\bar{A}_1 C_3$ $(p \neq 2)$, $A_1 G_2$ $(p \neq 2)$, $A_2 \tilde{A}_2$, $G_2$ $(p=7$) and $A_1$ $(p \geq 13)$. Let $X$ be an $M$-irreducible subgroup $A_1$ of $M$. 

Firstly, let $M = B_4$. The $M$-irreducible $A_1$ subgroups are straightforward to find, using Lemma \ref{class}. They are the subgroups $F_4(\#1)$--$F_4(\#7)$ listed in Table \ref{F4tab} (without the constraints imposed on the field twists) as well as the subgroups $Y_1$ and $Y_2$ acting as $3^{[r]} \otimes 1^{[s]} + 0$ ($p \geq 5$) and $1 \otimes 1^{[r]} \otimes 1^{[s]}$ $(p=2; 0<r<s)$ on $V_{B_4}(\lambda_1)$, respectively. Note that $Y_1, Y_2 < D_4 < B_4$ and are conjugate by a triality automorphism to subgroups acting on $V_{B_4}(\lambda_1)$ as $4^{[r]} + 2^{[s]} + 0$ or $0 | (2 + 2^{[r]} + 2^{[s]}) | 0$, respectively. The first is $F_4(\#6^{\{r,s,s\}})$ and the second is $B_4$-reducible, by Lemma \ref{class}. Now consider $F_4(\#1)$ and $F_4(\#3)$. These are diagonal subgroups of the maximal rank subsystem subgroups $\bar{A}_1^4$ and $\tilde{A}_1^4$, respectively. Indeed, the subgroup $\bar{A}_1^4$ is a maximal subgroup of $D_4 < B_4$, corresponding to the chain $\text{SO}_4 \text{SO}_4 < \text{SO}_8 < \text{SO}_9$. The subgroup $\tilde{A}_1^4$ only exists when $p=2$ and is the image of $\bar{A}_1^4$ under the graph automoprhism of $F_4$. The Weyl group of $F_4$ induces an action of $S_4$ on both $\bar{A}_1^4$ and $\tilde{A}_1^4$. The field twists in the embeddings of $F_4(\#1)$ and $F_4(\#3)$ can hence be chosen such that $0 < r < s < t$, as in Table \ref{F4tab}. The constraints on the field twists in the remaining subgroups in Table \ref{F4tab} all come from considering the $M$-conjugacy classes of the $M$-irreducible $A_1$ subgroups.   

We must now prove that $F_4(\#1)$--$F_4(\#7)$ are $F_4$-irreducible. We first treat the cases when $p > 2$. 

Let $X$ be $F_4(\#1)$ $(p \neq 2)$ or $F_4(\#6)$, so $X$ is contained in $\bar{A}_1^2 B_2$. By restricting the composition factors of $L(F_4) \downarrow M$ (from Theorem \ref{maximalexcep}) to $\bar{A}_1^2 B_2$ we have \[ L(F_4) \downarrow \bar{A}_1^2 B_2 = (2,0,00) / (0,2,00) / (0,0,02) / (1,1,10) / (1,0,01) / (0,1,01).\] Let $X = F_4(\#1)$ $(p \neq 2)$. Then $X$ is contained in $\bar{A}_1^4$ and we have \begin{align*} L(F_4) \downarrow \bar{A}_1^4 = & (2,0,0,0) / (0,2,0,0) / (0,0,2,0) / (0,0,0,2) / (1,1,1,1) / (1,1,0,0) / \\  & (1,0,1,0) / (1,0,0,1) / (0,1,1,0) / (0,1,0,1) / (0,0,1,1).\end{align*} Since $0 < r < s < t$, there are no trivial composition factors occurring in $L(F_4) \downarrow X$ and thus $X$ is $F_4$-irreducible by Corollary~\ref{notrivs}. Now let $X = F_4(\#6)$, so the projection of $X$ to $B_2$ is a maximal subgroup $A_1$ $(p \geq 5)$ and $X$ is a subgroup of $\bar{A}_1^2 A_1$. The composition factors of $L(F_4)$ restricted to $\bar{A}_1^2 A_1$ are then as follows: \[L(F_4) \downarrow \bar{A}_1^2 A_1 = (2,0,0) / (0,2,0) / (0,0,2) / (0,0,W(6)) / (1,1,4) / (1,0,3) / (0,1,3).\] Therefore, Corollary \ref{notrivs} shows that $X$ is $F_4$-irreducible unless $p=5$ and $X = F_4(\#6^{\{0,0,1\}})$, in which case $L(F_4) \downarrow X = 10^2 /$ $\!\! 8^3 /$ $\!\! 6 /$ $\!\! 4 /$ $\!\! 2^4 /$ $\!\! 0$. To prove $X$ is $F_4$-irreducible in this case we use Lemma \ref{wrongcomps}. Suppose $Y$ is an $L'$-irreducible subgroup $A_1$ of a Levi subgroup $L$ having the same composition factors as $X$ on $L(F_4)$. Then using Table \ref{levif4}, we see that $L' = B_3$, $A_2 \tilde{A}_1$ and $\tilde{A}_2 A_1$ are the only possibilities since $X$ and hence $Y$, by definition, has only one trivial composition factor on $L(F_4)$. Suppose $L' = B_3$. Then from Table \ref{levif4}, we see that $V_{B_3}(100)$ occurs as a multiplicity two composition factor of $L(F_4) \downarrow B_3$. But it is impossible to construct two isomorphic $7$-dimensional modules from the composition factors of $L(F_4) \downarrow Y$, hence $Y$ is not contained in $B_3$. Now suppose $L' = A_2 \tilde{A}_1$ or $\tilde{A}_2 A_1$. Then from Table \ref{levif4}, we have that $(V_{A_2}(00),V_{A_1}(1))$ occurs as a composition factor of $L(F_4) \downarrow L'$ and hence $Y$ has a $2$-dimensional composition factor on $L(F_4)$. This is a contradiction. Hence we conclude that $Y$ does not exist and $X$ is $F_4$-irreducible by Lemma \ref{wrongcomps}.   

Next, if $X = F_4(\#4)$ or $F_4(\#7)$ then $X$ is $F_4$-irreducible by Corollary \ref{notrivs}, with the composition factors of $L(F_4) \downarrow X$ given in Table \ref{F4tabcomps}. 

The final case when $p \neq 2$ is when $X = F_4(\#5)$, acting on $V_{B_4}(\lambda_1)$ as $2 \otimes 2^{[r]}$ $(r \neq 0)$. From Table \ref{F4tabcomps}, if $p > 3$ then $X$ has no trivial composition factors on $L(F_4)$. Hence Corollary \ref{notrivs} applies and $X$ is $F_4$-irreducible. If $p=3$ then the composition factors of $V_{26} \downarrow X$ have dimensions $9,4^4,1$ or $9,4^3,3,1^2$ (if $r = 1$). From Table \ref{levif4}, we see that the only Levi subgroup with a composition factor of dimension at least 9 on $V_{26}$ is $L' = C_3$. Moreover, the composition factors of $C_3$ acting on $V_{26}$ have dimensions $13,6^2,1$. Therefore, no subgroup $A_1$ of $C_3$ has the same composition factors as $X$ on $V_{26}$. Hence $X$ is $F_4$-irreducible by Lemma \ref{wrongcomps}. 

We now assume $p=2$. If $X$ is either $F_4(\#1)$ or $F_4(\#3)$,  acting as $1 \otimes 1^{[r]} + 1^{[s]} \otimes 1^{[t]}$  or $2 + 2^{[r]} + 2^{[s]} + 2^{[t]}$  ($0 <  r < s < t$ in both cases), respectively, we use Lemma \ref{wrongcomps}. From Table \ref{F4tabcomps}, we see that $1 \otimes 1^{[r]} \otimes 1^{[s]} \otimes 1^{[t]}$ occurs as a composition factor of $L(F_4) \downarrow X$. Table \ref{levif4} shows that there are no Levi subgroups with $L(F_4) \downarrow L'$ having a composition factor of dimension at least 16 when $p=2$. Hence $X$ is $F_4$-irreducible by Lemma \ref{wrongcomps}. 

Finally, suppose $X = F_4(\#2)$ and so contained in $\bar{A}_1^2 \tilde{A}_1^2$, a maximal rank connected subgroup of $F_4$, corresponding, for example, to the chain $\text{SO}_4 \text{Sp}_2 \text{Sp}_2 < \text{Sp}_4 \text{Sp}_4 < \text{Sp}_8$. From $L(F_4) \downarrow B_4$ we have \begin{align*} L(F_4) \downarrow \bar{A}_1^2 \tilde{A}_1^2 = \hspace{0.1cm} & \hspace{0.1cm} (2,0,0,0) / (0,2,0,0) / (0,0,2,0) / (0,0,0,2) / (1,1,2,0) / (1,1,0,2) / (1,1,0,0) / \\ \hspace{0.1cm} & \hspace{0.1cm} (1,0,1,1) /  (0,1,1,1) / (0,0,2,2) / (0,0,0,0)^4. \end{align*} It follows that $X$ has at most 5 trivial composition factors on $L(F_4)$ with the possible extra one coming from $(1,1,2,0)$, $(1,1,0,2)$, $(1,0,1,1)$ or $(0,1,1,1)$. Using Table \ref{levif4}, we see that only $L'=B_3$, $C_3$ can have an irreducible subgroup $A_1$ with the same composition factors as $X$ on $L(F_4)$. The only $L'$-irreducible $A_1$ subgroups of $B_3$ or $C_3$ when $p=2$ are diagonal subgroups of $\bar{A}_1^3$, $\bar{A}_1^2 \tilde{A}_1$, $\bar{A}_1 \tilde{A}_1^2$ or $\tilde{A}_1^3$. Any such subgroup has at least 6 trivial composition factors on $L(F_4)$. Hence $X$ is $F_4$-irreducible by Lemma \ref{wrongcomps}. This completes the analysis of the $M$-irreducible $A_1$ subgroups contained in $M = B_4$. 

Now let $M = C_4$ $(p=2)$. Then another application of Lemma \ref{class} shows that $X$ acts as $1 \otimes 1^{[r]} + 1^{[s]} \otimes 1^{[t]}$ ($r \neq 0$; $s \neq t$; $\{0,r\} \neq \{s,t\}$), $1^{[r]} \otimes 1^{[s]} + 1^{[t]} + 1^{[u]}$ ($r < s$; $t < u$), $1 + 1^{[r]} + 1^{[s]} + 1^{[t]}$ ($0<r<s<t$) or $1 \otimes 1^{[r]} \otimes 1^{[s]}$ ($0 < r < s$). The first three are contained in the subsystem subgroup $C_2^2$ and the latter is contained in the subsystem subgroup $\tilde{D}_4$. Suppose $X$ is contained in $C_2^2$. There is only one $F_4$-conjugacy class of subgroups $C_2^2$ in $F_4$ because $C_{F_4}(C_2)^\circ = C_2$ (\cite[p.333, Table 2]{LS6}). Therefore $X$ is contained in $B_4$ and has already been considered. Now suppose $X$ is contained in $\tilde{D}_4$ acting on $V_{C_4}(\lambda_1)$ as $1 \otimes 1^{[r]} \otimes 1^{[s]}$ $(0 < r < s)$. In this case $X$ is contained in $A_1 B_2$ which is contained in a subgroup $B_3$ since $p=2$. By \cite[Table 8]{car}, we have $N_{F_4}(\bar{D}_4) / \bar{D}_4 \cong S_3$ and hence applying the graph automorphism of $F_4$, we have $N_{F_4}(\tilde{D}_4) / \tilde{D}_4 \cong S_3$. It follows that all three $\tilde{D}_4$-conjugacy classes of $B_3$ are conjugate in $F_4$. Therefore $X$ is contained in a $C_4$-reducible $B_3$ acting as $000|100|000$ on $V_{C_4}(\lambda_1)$ and hence $X$ is $F_4$-reducible. 

Next, we consider the case $M = \bar{A}_1 C_3$ $(p \neq 2)$. Using Lemma \ref{class}, we see that the projection of $X$ to $C_3$ acts as $5$ $(p \geq 5)$, $3^{[r]} + 1^{[s]}$ $(p \geq 7)$, $2^{[r]} \otimes 1^{[s]}$ $(p \neq 2; r \neq s)$, or $1 + 1^{[r]} + 1^{[s]}$. In the second and fourth cases, the projection of $X$ is contained in $\bar{A}_1 C_2$ and so $X$ is contained in $\bar{A}_1^2 C_2$, which is also a subgroup of $B_4$. Therefore we have already considered them. Now consider the first case, where the projection of $X$ to $C_3$ is a maximal subgroup $A_1$ $(p \geq 7)$ and so $X$ is a diagonal subgroup of $\bar{A}_1 A_1$. This gives the conjugacy classes $F_4(\#8)$ in Table \ref{F4tab}. From Table \ref{F4tabcomps}, we see that $X$ has no trivial composition factors on $L(F_4)$. Hence Corollary \ref{notrivs} applies and $X$ is $F_4$-irreducible.  

The final possibility is that the projection of $X$ to $C_3$ is contained in a maximal subgroup $A_1 A_1$ $(p \neq 2)$, which acts as $(2,1)$ on $V_{C_3}(100)$. In this case $X \hookrightarrow \bar{A}_1 A_1 A_1$ via $(1^{[r]},1^{[s]},1^{[t]})$ $(rst=0$). First, we suppose $s=t$. Then if $p \geq 5$ we have $X$ is contained in $\bar{A}_1^2 B_2$ and hence $B_4$ because $2^{[r]} \otimes 1^{[r]} = 3^{[r]} + 1^{[r]}$. If $p=3$ then $X$ is $M$-reducible by Lemma \ref{class} because $2^{[r]} \otimes 1^{[r]} = 1^{[r]} | 3^{[r]} | 1^{[r]}$. Therefore, we only need to consider the case $s \neq t$, yielding the conjugacy classes $F_4(\#9)$ in Table \ref{F4tab}. We now prove that they are all $F_4$-irreducible. From the composition factors of $L(F_4) \downarrow \bar{A}_1 C_3$ in Theorem \ref{maximalexcep}, we find that \[L(F_4) \downarrow \bar{A}_1 A_1 A_1 = (2,0,0) / (0,2,0) / (0,0,2) / (1,W(4),1) / (1,0,W(3)) / (0,W(4),2).\] If $p > 5$ then $X$ has no trivial composition factors on $L(F_4)$ and so Corollary \ref{notrivs} shows that $X$ is $F_4$-irreducible. If $p=5$ then Corollary \ref{notrivs} applies unless $X = F_4(\#9^{\{0,0,1\}})$, which is embedded via $(1,1,1^{[1]})$. In this case $L(F_4) \downarrow X = 14 /$ $\!\! 10^2 /$ $\!\! 8^3 /$ $\!\! 2^2 /$ $\!\! 0$. To show that $X$ is $F_4$-irreducible we use Lemma \ref{wrongcomps}. Suppose $Y$ is an $L$-irreducible subgroup $A_1$ of a Levi factor $L'$ having the same composition factors as $X$ on $L(F_4)$. Then using Table \ref{levif4}, we find that $L' = B_3$ since $X$, and hence $Y$, has a $15$-dimensional composition factor and only one trivial composition factor on $L(F_4)$. Further inspection of the dimensions of the composition factors of $X$ on $L(F_4)$ shows that there are three composition factors of dimension $8$ as well as the one of dimension 15. The dimensions of the composition factors of $L(F_4) \downarrow B_3$ are $21,8^2,7^2,1$. It follows that $Y$ is not contained in $B_3$. This contradiction shows that $Y$ does not exist and hence $X$ is $F_4$-irreducible by Lemma~\ref{wrongcomps}.   

Now let $p=3$. From the restriction of $V_{26} \downarrow \bar{A}_1 C_3$ in Theorem \ref{maximalexcep}, we have \[V_{26} \downarrow \bar{A}_1 A_1 A_1 = (1,2,1) / (0,4,0) / (0,2,2) / (0,0,0).\] Using Table \ref{levif4}, we see that $L' = C_3$ is the only Levi factor that can contain a subgroup $A_1$ with the same composition factors as $X$ on $V_{26}$ because $X$ has a 9-dimensional composition factor, namely $2^{[s]} \otimes 2^{[t]}$ (recalling $s \neq t$). We want to apply Lemma \ref{wrongcomps} to conclude that $X$ is $F_4$-irreducible. From Table \ref{levif4}, we have $V_{26} \downarrow C_3 = 100^2 / 010$. If the field twists for the embedding of $X$ in $\bar{A}_1 A_1 A_1$ are all distinct then $X$ has a 12-dimensional composition factor as well as a 9-dimensional one, and hence there is no subgroup $A_1$ of $C_3$ with the same composition factors as $X$ on $V_{26}$. The cases which remain are $X$ embedded via $(1^{[r]},1^{[r]},1^{[s]})$ $(rs=0)$ and $(1^{[r]},1^{[s]},1^{[r]})$ $(rs=0)$. The dimensions of the composition factors on $V_{26}$ are then $9,4^4,1$ (or $9,4^3,3,1^2$ if $s=r+1$) or $9^2,4,3,1$ respectively. None of these are compatible with a subgroup $A_1$ of $C_3$ and hence $X$ is $F_4$-irreducible. This completes the analysis of the $M$-irreducible $A_1$ subgroups contained in $M = \bar{A}_1 C_3$.

Now suppose $M = A_1 G_2$ $(p \neq 2)$. By Theorem \ref{thmG2}, the projection of $X$ to $G_2$ is either contained in $A_1 A_1$ or is maximal with $p \geq 7$. In the first case we claim that $X$ is contained in $\bar{A}_1 C_3$. Indeed, since the factor $G_2$ of $M$ is contained in $D_4$ by \cite[3.9]{se2}, it follows that the long root subgroup $A_1$ of $G_2$ is a long root subgroup $A_1$ of $D_4$ and hence $F_4$. Therefore $X$ is contained in $\bar{A}_1 C_{F_4}(\bar{A}_1)^\circ = \bar{A}_1 C_3$. Now consider the second case. Then $X \hookrightarrow A_1 A_1 < A_1 G_2$ $(p \geq 7)$ via $(1^{[r]},1^{[s]})$ $(rs=0)$. From Table \ref{F4tabcomps}, we see that if $r \neq s$ then $X$ has no trivial composition factors on $L(F_4)$. Hence $X$ is $F_4$-irreducible by Corollary \ref{wrongcomps}, yileding $F_4(\#11)$. Now consider $X \hookrightarrow A_1 A_1$ via $(1,1)$. Then \[ L(F_4) \downarrow X = W(10)^2 / W(8) / 6 / 4 / 2^3.\] From Table \ref{wrongcomps}, we have $Y = F_4(\#8^{\{0,0\}}) < \bar{A}_1 C_3$ has the same composition factors on $L(F_4)$. Since $p \geq 7 > 3 = N(A_1,F_4)$, Theorem \ref{A1samecomp} applies. Hence $X$ is conjugate to $Y$ and has already been considered.  

Now suppose $M = A_2 \tilde{A}_2$. Then $X$ is contained in $Y = Y_1 Y_2 = A_1 A_1 < M$ $(p \neq 2)$ (both factor $A_1$ subgroups are irreducibly embedded in $A_2$). We claim that $Y$ is contained in $\bar{A}_1 C_3$ and hence so is $X$, which has therefore already been considered. Indeed, by \cite[Table 8]{car}, we have that $F_4$ contains an involution which acts as a graph automorphism on both $A_2$ factors of $M$. Therefore, there exists an involution $t$ such that $Y < C_{F_4}(t)^\circ$. One calculates that $C_{F_4}(t)^\circ = \bar{A}_1 C_3$, as required. 

Now let $M = G_2$ $(p=7)$. By Theorem \ref{thmG2}, up to $M$-conjugacy, $X$ is contained in $A_1 A_1$ or is a maximal subgroup. Consider the first case. By \cite[Table 4.3.1]{gls3}, the subgroup $A_1 A_1 < G_2$ is the centraliser in $G_2$ of a semisimple element of order 2. By \cite[Proposition 1.2]{LS8} the connected centraliser in $F_4$ of $t$ is $B_4$ or $\bar{A}_1 C_3$ with the trace of $t$ on $V_{26}$ being $-6$ or $2$, respectively. We calculate that the trace of $t$ on $V_{26}$ is $2$ using $V_{26} \downarrow A_1 A_1 = (2,2) / (1,1) / (1,W(3)) / (0,W(4))$ and the fact the element $t$ can be seen as minus the identity in both $A_1$ factors. Therefore the subgroup $A_1 A_1$ is contained in $\bar{A}_1 C_3$, and in particular $X \hookrightarrow A_1 A_1$ via $(1^{[r]},1^{[s]})$ $(rs=0)$ is conjugate to $F_4(\#9^{\{r,r,s\}})$.

Now consider the second case, when $X$ is a maximal subgroup $A_1$ of $M$. By restricting from $L(F_4) \downarrow M$, it follows that $L(F_4) \downarrow X = 16 / 14 / 10^3 / 6 / 2^3$. Now let $Y = F_4(\#8^{\{1,0\}})$ from Table \ref{F4tab}. Then from Table \ref{F4tabcomps}, we have $L(F_4) \downarrow Y = 16 / 14 / 10^3 / 6 / 2^3$. As $p=7$, Theorem \ref{A1samecomp} applies and hence $X$ is conjugate to $Y$, and has already been considered.    

If $M = A_1$ $(p \geq 13)$ then $X = M$ and $X$ is $F_4$-irreducible and not conjugate to any other subgroup $A_1$ (this follows immediately from Theorem \ref{maximalexcep}).

Finally, we use the composition factors given in Table \ref{F4tabcomps} to show that there are no further conjugacies between the $A_1$ subgroups in Table \ref{F4tab}. 
\end{proof}

\section{Proof of Theorem \ref{thmE6}: $E_6$-irreducible $A_1$ subgroups} \label{secE6}

In this section we prove Theorem \ref{thmE6}, which classifies the $E_6$-irreducible $A_1$ subgroups of $E_6$. 

\begin{thm*} \label{thmE6}
Suppose $X$ is an irreducible subgroup $A_1$ of $E_6$. Then $X$ is conjugate to exactly one subgroup of Table \ref{E6tab} and each subgroup in Table \ref{E6tab} is irreducible. 
\end{thm*}

\begin{longtable}{>{\raggedright\arraybackslash}p{0.05\textwidth - 2\tabcolsep}>{\raggedright\arraybackslash}p{0.15\textwidth - 2\tabcolsep}>{\raggedright\arraybackslash}p{0.35\textwidth - 2\tabcolsep}>{\raggedright\arraybackslash}p{0.07\textwidth-\tabcolsep}@{}}

\caption{The $E_6$-irreducible $A_1$ subgroups of $E_6$ \label{E6tab}} \\

\hline

ID & $M$ & $V_M \downarrow X$ & $p$ \\

\hline

1 & $\bar{A}_1 A_5$ & $(1^{[r]},5^{[s]})$ $(rs=0)$ & $\geq 7$ \\

2 & &  $(1^{[r]},2^{[s]} \otimes 1^{[t]})$ ($rst = 0$; $s \neq t$) & $\geq 3$ \\

\hline

3 & $A_2^3$ & $(2,2^{[r]},2^{[s]})$ ($0 < r < s)$ & $\geq 3$ \\

\hline

4 & $A_2 G_2$ & $(2^{[r]},G_2(\#3)^{[s]})$ $(rs=0; r \neq s)$ & $\geq 7$ \\

\hline

5 & $F_4$ & $F_4(\#10)$ & $\geq 13$ \\

\hline

6 & $C_4$ $(p \neq 2)$ & 7 & $\geq 11$ \\

\hline

\end{longtable}

The composition factors of $V_{27}$ and $L(E_6)$ restricted to each irreducible subgroup $A_1$ in Table \ref{E6tab} are found in Table \ref{E6tabcomps}.

\begin{proof}
We use the same method as for $F_4$, taking each reductive, maximal connected subgroup $M$ of $E_6$ in turn (from Theorem \ref{maximalexcep}) and finding all $E_6$-irreducible $A_1$ subgroups contained in them, up to $E_6$-conjugacy. Let $X$ be an $M$-irreducible subgroup $A_1$ of $M$. 

First, consider $M = \bar{A}_1 A_5$. Then using Lemma \ref{class}, we see that the projection of $X$ to $A_5$ acts on $V_{A_5}(\lambda_1)$ either as $5$ $(p \geq 7)$ or $2^{[r]} \otimes 1^{[s]}$ $(p \neq 2$; $r \neq s)$. Suppose we are in the first case and so $X \hookrightarrow \bar{A}_1 A_1$ $(p \geq 7)$ via $(1^{[r]},1^{[s]})$, where the second factor $A_1$ acts as $5$ on $V_{A_5}(\lambda_1)$. From Table \ref{E6tabcomps}, we see that $X$ has no trivial composition factors on $L(E_6)$ and is thus $E_6$-irreducible by Corollary \ref{notrivs}, yielding $E_6(\#1)$. 

In the second case $X = E_6(\#2)$, a diagonal subgroup of $\bar{A}_1 A_1 A_1 < \bar{A}_1 A_5$, where $A_1 A_1 < A_5$ acts on $V_{A_5}(\lambda_1)$ as $(2,1)$. From the restriction of $L(E_6) \downarrow \bar{A}_1 A_5$ in Theorem \ref{maximalexcep}, it follows that \begin{align*} L(E_6) \downarrow \bar{A}_1 A_1 A_1 = & \hspace{0.15cm}  (2,0,0)/ (1,W(4),1)/(1,2,1)/(1,0,W(3)) / (0,W(4),2) / (0,W(4),0)/ \\ & \hspace{0.15cm} (0,2,0) / (0,2,2)/(0,0,2) . \end{align*} If $p > 5$ then $X$ is $E_6$-irreducible by Corollary \ref{notrivs}. If $p=5$ then Corollary \ref{notrivs} applies unless $X = E_6(\#2^{\{0,0,1\}})$. In this case $V_{27} \downarrow X = 12 / 8 / 6 / 4 / 0$ (from Table \ref{E6tabcomps}) so the dimensions of the composition factors are $9,8,5,4,1$. We use Lemma \ref{wrongcomps} to show that $X$ is $E_6$-irreducible. Suppose not, then there exists a subgroup $A_1$ with the same composition factors as $X$ on $V_{27}$ contained (not necessarily $L$-irreducibly) in $L'=D_5, A_1 A_4, A_1 A_2^2$ or $A_5$. But the dimensions of their composition factors on $V_{27}$ are $16, 10, 1$ for $D_5$, $10^2, 5, 2$ for $A_1 A_4$, $9, 6^2, 3^2$ for $A_1 A_2^2$ and $15, 6^2$ for $A_5$. Therefore, no subgroup $A_1$ of $L'$ has the same composition factors as $X$, a contradiction. When $p=3$, we use Lemma \ref{wrongcomps} again. There are four possibilities for the dimensions of the composition factors of $X$ on $V_{27}$: $12,9,4,1^2$ ($r,s,t$ distinct), $9^2,4,3,1^2$ ($r = t$), $9,4^4,1^2$ $(r = s \neq t-1)$ and $9,4^3,3,1^3$ $(r=s=t-1)$. It follows that only $L'=D_5$ can contain a subgroup $A_1$ with the same composition factors as $X$ on $V_{27}$. Further consideration of the dimensions (and recalling that $p=3$) leads to the only possibility being $Y < D_5$ with $V_{D_5}(\lambda_1) \downarrow Y = 2 \otimes 2^{[a]} + 0$ $(a \neq 0)$. Then $V_{27} \downarrow Y = 2 \otimes 2^{[a]} / 3 \otimes 1^{[a]} / 1 \otimes 3^{[a]} / (1 \otimes 1^{[a]})^2 / 0^2$. The composition factors of $X$ and $Y$ do not agree on $V_{27}$ regardless of the choice of $r,s,t$ and $a$. Hence $X$ is $E_6$-irreducible, completing the analysis of the $M$-irreducible $A_1$ subgroups of $M = \bar{A}_1 A_5$.        

Now let $M = A_2^3$. Then $p \neq 2$ and $X$ is a diagonal subgroup of $A_1^3 < A_2^3$, where each factor $A_1$ is irreducibly embedded in $A_2$. By Theorem \ref{maximalexcep}, we have $V_{27} \downarrow A_2^3 = (10,01,00) / (00,10,01) / (01,00,10)$ and hence $V_{27} \downarrow A_1^3 = (2,2,0) / (0,2,2) / (2,0,2)$.  First consider the case where all of the field twists in the embedding of $X$ are distinct. Then the action of $X$ on $V_{27}$ has three composition factors, all of dimension 9. Using Table \ref{levie6}, we see that no subgroup $A_1$ of a Levi subgroup can have the same composition factors as $X$ on $V_{27}$. Hence $X$ is $E_6$-irreducible by Lemma \ref{wrongcomps} and this yields $E_6(\#3)$.

If at least two of the field twists in the embedding of $X$ are equal then we claim that $X$ is contained in $\bar{A}_1 A_5$. To prove the claim, we first show that $A_1^3 < A_2^3$ is contained in $C_4$, acting as $(1,1,1)$ on $V_{C_4}(\lambda_1)$. Consider the standard graph automorphism of $E_6$, call it $\tau$. Then $w_0 = - \tau$ and so $t := \tau w_o$ acts as $-1$ on a maximal torus of $E_6$. Therefore $t$ induces a graph automorphism on each $A_2$ factor of $A_2^3$. It follows that $A_1^3 < C_{E_6}(t)$ because an irreducible $A_1$ in a subgroup $A_2$ is centralised by a graph automorphism of $A_2$. We check that dim($C_{L(E_6)}(t)$) = 36 and so dim($C_{E_6}(t)$) = 36 (by \cite[9.1]{borel}, since $t$ is semisimple). Therefore, $C_{E_6}(t)^\circ = C_4$ by \cite[Table 4.3.1]{gls3} and we have $A_1^3 < C_4$. Considering the composition factors of $A_1^3$ on $V_{27}$, it follows that it is conjugate to a subgroup $A_1^3$ acting as $(1,1,1)$ on $V_{C_4}(\lambda_1)$, as required. Therefore, $X$ is contained in $\bar{A}_1 C_3 < C_4$ since $1^{[r]} \otimes 1^{[r]} \otimes 1^{[s]} = 2^{[r]} \otimes 1^{[s]} + 1^{[s]}$ ($p \neq 2$). The factor $A_1$ of $\bar{A}_1 C_3$ is generated by root subgroups of $E_6$ and so $X < \bar{A}_1 C_{E_6}(\bar{A}_1)^\circ = \bar{A}_1 A_5$, proving the claim. If only two of the twists are equal then $X$ is $E_6$-irreducible and conjugate to  $E_6(\#2^{\{r,r,s\}})$. If all three twists are equal then $X$ is $C_4$-reducible and hence $E_6$-reducible. This completes the case $M = A_2^3$.   

Next, we let $M = A_2 G_2$. By Theorem \ref{thmG2}, up to $M$-conjugacy, the projection of $X$ to $G_2$ is either contained in $A_1 A_1$ or is maximal $(p \geq 7)$. Assume the former. Since the $G_2$ factor of $M$ is contained in $D_4$ by \cite[3.15]{se2}, the first $A_1$ factor of $A_1 A_1$ is generated by root subgroups of $E_6$. Therefore, $A_2 \bar{A}_1 A_1 < \bar{A}_1 C_{E_6}(\bar{A}_1)^\circ = \bar{A}_1 A_5$. Therefore, $X$ has already been considered in the $\bar{A}_1 A_5$ case. Now assume the projection to $G_2$ is maximal, so $p \geq 7$. Then $X \hookrightarrow A_1 A_1 < A_2 G_2$ via $(1^{[r]},1^{[s]})$ $(rs = 0)$ or $(1,1)$ where each factor $A_1$ is maximal. If $X$ is embedded via $(1^{[r]},1^{[s]})$ then $X$ is $E_6$-irreducible by Corollary \ref{notrivs}, yielding $E_6(\#4)$. If $X$ is embedded via $(1,1)$ then $X$ is conjugate to $Y = E_6(\#1^{\{0,0\}})$, by Theorem \ref{A1samecomp} since $p > 5 = N(A_1,E_6)$ and $X$ and $Y$ have the same composition factors on $L(E_6)$. 

Now suppose $M = F_4$. Theorem \ref{thmF4} gives all of the conjugacy classes of $F_4$-irreducible $A_1$ subgroups, showing they are all contained in $B_4$, $\bar{A}_1 C_3$ $(p \neq 2)$, $A_1 G_2$ $(p \neq 2)$ or $A_1$ $(p \geq 13)$. If $X$ is contained in $B_4$ then $X$ is $E_6$-reducible because $B_4$ is contained in a $D_5$ Levi subgroup. If $X$ is contained in $\bar{A}_1 C_3$ $(p \neq 2)$ then $X$ is also contained in $\bar{A}_1 A_5$ since  $C_{E_6}(\bar{A}_1)^\circ = A_5$, as above and has already been considered. If $X$ is contained in $A_1 G_2$ then $X$ is contained in the maximal subgroup $A_2 G_2$ (since $C_{E_6}(G_2)^\circ = A_2$) and has also been considered already. Finally, if $X$ is a maximal subgroup $A_1$ of $F_4$ then $X$ is $E_6$-irreducible by Corollary~\ref{notrivs}, yielding $E_6(\#5)$.      

Now let $M = C_4$ $(p \neq 2)$. By considering the action of $X$ on $V_{C_4}(\lambda_1)$ and using Lemma \ref{class}, it follows that $X$ is contained in $C_2^2$, $\bar{A}_1 C_3$, $A_1^3$ or $A_1$ $(p \geq 11)$. If $X$ is contained in $C_2^2$ then $X$ is $E_6$-reducible because $C_{E_6}(C_2)^\circ = C_2 T_1$, by \cite[p.333, Table 3]{LS6} and so $C_2^2$ is contained in a Levi subgroup of $E_6$. If $X$ is contained in $\bar{A}_1 C_3$ then $X$ is also contained in $\bar{A}_1 A_5$, hence considered in the $\bar{A}_1 A_5$ case above. If $X$ is contained in $A_1^3$, acting as $(1,1,1)$ on $V_{C_4}(\lambda_1$) then an argument in the $A_2^3$ case showed that $X$ is contained in $A_2^3$. The last possibility is $p \geq 11$ and $X$ is maximal in $C_4$ acting as $7$ on $V_{C_4}(\lambda_1)$. From Table \ref{E6tabcomps}, we see that $X$ is $E_6$-irreducible by Corollary \ref{notrivs}, yielding $E_6(\#6)$ in Table \ref{E6tab}. 

Now let $M$ be one of the two conjugacy classes of $G_2$ $(p \neq 7)$. By Theorem \ref{thmG2}, an $M$-irreducible subgroup $A_1$ is contained in $A_1 A_1$ or is maximal with $p > 7$. If $X$ is maximal then $X$ is conjugate to $E_6(\#6)$ by Theorem \ref{A1samecomp}, since both subgroups have the same composition factors on $L(E_6)$ and $p > 5 = N(A_1,E_6)$. Now suppose $X$ is contained in $A_1 A_1$. When $p \neq 2$, we claim that $A_1 A_1$ is contained in $\bar{A}_1 A_5$ and $X$ has already been considered. In $G_2$, we have that $A_1 A_1$ is the centraliser of a semisimple element of order 2, and by \cite[Table~4.3.1]{gls3}, the centraliser in $E_6$ of this involution is either $\bar{A}_1 A_5$ or $T_1 D_5$. An easy check shows it to be the former, proving the claim. 

When $p=2$, we claim that $X$ is $E_6$-reducible. To prove this we consider the action of $A_1 A_1$ on $L(E_6)$. By \cite[Table 10.1]{LS1}, we have $L(E_6) \downarrow G_2 = 11 + 01$. In Table \ref{G2tabcomps}, the composition factors of $V_{G_2}(01) \downarrow A_1 A_1$ are given and moreover, $V_{G_2}(01) \downarrow A_1 A_1 = ((0,0) | ((2,0) + (0,2)) | (0,0)) + (1,3)$. Therefore $A_1 A_1$, and hence $X$, fixes a non-zero vector of $L(E_6)$. By \cite[Lemma~1.3]{se2}, we have $X$ is contained in a parabolic subgroup, $\bar{A}_1 A_5$ or $A_2^3$. The $\bar{A}_1 A_5$ and $A_2^3$ cases show that neither $\bar{A}_1 A_5$ nor $A_2^3$ contain an $E_6$-irreducible subgroup $A_1$ when $p=2$. Therefore $X$ is $E_6$-reducible, as claimed. 

Finally, let $M$ be one of the two conjugacy classes of $A_2$ $(p \geq 5)$. There is just one $M$-irreducible subgroup $X$, acting as $2$ on $V_{A_2}(10)$. If $p \geq 7$ then Theorem \ref{A1samecomp} shows that $X$ is conjugate to $E_6(\#1^{\{0,0\}})$, which is contained in $\bar{A}_1 A_5$. When $p=5$, we show that $X$ is $E_6$-reducible. By \cite[Table 10.2]{LS1}, we have $V_{27} \downarrow A_2 = W(22) = 22 | 11$ or $W(22)^*$ and $V_{A_2}(20) \otimes V_{A_2}(02) = (11|22|11) + 00$.  Using \cite[1.2]{donkin}, which states that a tensor product of tilting modules is again tilting, we have $4 \otimes 4 = (0|8|0) + (2|6|2) + 4$. Since $V_{A_2}(20) \downarrow X = 4 + 0$ and $V_{A_2}(11) \downarrow X = 4 + 2$, it follows that  $V_{27} \downarrow X = (0 | 8 | 0) + (6 | 2) + 4^2$ or $ (0 | 8 | 0) + (2 | 6) + 4^2$ . This shows that $X$ fixes a $1$-space of $V_{27}$. The dimension of the centraliser in $E_6$ of this $1$-space is at least $51=78-27$ and hence $X$ is either contained in a parabolic subgroup or $F_4$. Assume the latter. By \cite[Table 10.2]{LS1}, we have $V_{27} \downarrow F_4 = 0001 + 0000$ and so any subgroup of $F_4$ has a trivial direct summand on $V_{27}$. Since $X$ does not have such a summand on $V_{27}$, it is not contained in $F_4$. Therefore $X$ is $E_6$-reducible.  

Using the composition factors listed in Table \ref{E6tabcomps}, we see there are no further conjugacies between the $A_1$ subgroups in Table \ref{E6tab}, which completes the proof. 
\end{proof}

\section{Proof of Theorem \ref{thmE7}: $E_7$-irreducible $A_1$ subgroups} \label{secE7}

In this section we find the $E_7$-irreducible $A_1$ subgroups of $E_7$, proving Theorem \ref{thmE7}. 

\begin{thm*} \label{thmE7}

Suppose $X$ is an irreducible subgroup $A_1$ of $E_7$. Then $X$ is conjugate to exactly one subgroup of Table \ref{E7tab} and each subgroup in Table \ref{E7tab} is irreducible. 

\end{thm*}

\begin{longtable}{>{\raggedright\arraybackslash}p{0.04\textwidth - 2\tabcolsep}>{\raggedright\arraybackslash}p{0.08\textwidth - 2\tabcolsep}>{\raggedright\arraybackslash}p{0.68\textwidth - 2\tabcolsep}>{\raggedright\arraybackslash}p{0.07\textwidth-\tabcolsep}@{}}

\caption{The $E_7$-irreducible $A_1$ subgroups of $E_7$ \label{E7tab}} \\

\hline

ID  & $M$ & $V_{M} \downarrow X$ & $p$ \\

\hline

1 & $A_1 D_6$ & $(1^{[r]},5^{[s]} \otimes 1^{[t]})$ $(rst = 0; s \neq t)$ & $\geq 7$ \\

2 & & $(1^{[r]},5^{[s]} \otimes 1^{[t]})$ $(rst = 0; s \neq t)$ & $\geq 7$  \\

3 & & $(1^{[r]},2^{[s]} \otimes 1^{[t]} \otimes 1^{[u]})$  ($rstu=0$; $s, t, u$ distinct) & $ \geq 3$  \\

4 & & $(1^{[r]}, 10^{[s]} + 0)$  $(rs=0)$ & $ \geq 11$ \\

5 & & $(1^{[r]},8^{[s]} + 2^{[t]})$  $(rst=0)$ & $\geq 11$ \\

6 & & $(1^{[r]},6^{[s]} + 4^{[t]})$  $(rst=0)$ & $\geq 7$ \\

7 & & $(1^{[r]}, 6^{[s]} +  1^{[t]} \otimes 1^{[u]} + 0)$  $(rs=0; r < t < u)$ & $\geq 7$ \\

8 & & $(1^{[r]},4^{[s]} + 2^{[t]} + 1^{[u]} \otimes 1^{[v]})$  ($rstu=0$; $s \neq t$; $u \leq v$; if $u=v$ then $t < u$)  & $\geq 5$  \\

9 & & $(1^{[r]},4^{[s]} + 2^{[s]} + 1^{[t]} \otimes 1^{[u]})$  ($rst =0$; $r < t < u$ or $t=u \neq s$) & $\geq 5$ \\

10 & & $(1^{[r]},2^{[s]} \otimes 2^{[t]} + 2^{[u]})$  $(rsu=0; s<t)$ & $\geq 3$ \\

11 & & $(1^{[r]},2^{[s]} + 2^{[t]} + 2^{[u]} + 2^{[v]})$  $(rs=0; s<t<u<v)$ & $\geq 3$ \\

12 & & $(1,1^{[r]} \otimes 1^{[s]} + 1^{[t]} \otimes 1^{[u]} + 1^{[v]} \otimes 1^{[w]})$  (see Table \ref{conditions1} for conditions \newline on $r, \ldots, w$) & all \\

13 & & $(1^{[r]},0|(2^{[s]} + 2^{[t]} + 2^{[u]})|0 + 1^{[v]} \otimes 1^{[w]})$  $(rsv=0; s < t < u; v < w)$ & $2$ \\

14 & & $(1^{[r]},0|(2^{[s]} + 2^{[t]} + 2^{[u]} + 2^{[v]} + 2^{[w]})|0)$  $(rs = 0; s < t < u < v < w)$ & $2$ \\

\hline

15 & $G_2 C_3$ & $(G_2(\#3)^{[r]},5^{[s]})$  $(rs=0; r \neq s)$ & $\geq 7$ \\

16 & & $(G_2(\#3)^{[r]},2^{[s]} \otimes 1^{[t]})$  $(rst=0; r \neq s; s \neq t)$ & $\geq 7$ \\

\hline

17 & $A_1 G_2$ & $(1^{[r]},G_2(\#3)^{[s]})$  $(rs=0; r \neq s)$ & $\geq 7$ \\

\hline

18 & $A_1 F_4$ & $(1^{[r]},F_4(\#10)^{[s]})$  $(rs=0)$ & $\geq 13$ \\

\hline

19 & $A_1 A_1$ & $(1^{[r]},1^{[s]})$  $(rs=0; r \neq s)$ & $\geq 5$ \\

\hline

20 & $A_1$ & $1$ & $\geq 17$ \\

\hline

21 & $A_1$ & $1$ & $\geq 19$ \\

\hline

\end{longtable}

The composition factors of the irreducible $A_1$ subgroups in Table \ref{E7tab} acting on $V_{56}$ and $L(E_7)$ are listed in Table \ref{E7tabcomps}. 

\begin{proof}
We consider each reductive, maximal connected subgroup $M$ of $E_7$ in turn. By Theorem \ref{maximalexcep}, they are $A_1 D_6$, $A_7$, $A_2 A_5$, $G_2 C_3$, $A_1 G_2$ $(p \neq 2)$, $A_1 F_4$, $A_2$ $(p \geq 5)$, $A_1 A_1$ $(p \geq 5)$, $A_1$ $(p \geq 17)$ and $A_1$ $(p \geq 19)$. Let $X$ be an $M$-irreducible subgroup $A_1$. 

First let $M = A_1 D_6$. First, we need to find the $E_7$-conjugacy classes of the $A_1 D_6$-irreducible $A_1$ subgroups contained in $A_1 D_6$. Using Lemma \ref{class}, we see that the $A_1 D_6$-irreducible $A_1$ subgroups are the subgroups listed in lines $1$ to $14$ of Table \ref{E7tab} without the constraints on the field twists, as well as $Y = A_1 < A_1 D_6$ acting as $(1^{[r]},3^{[s]} \otimes 1^{[t]} + 1^{[u]} \otimes 1^{[v]})$ $(p \geq 5; s \neq t)$ and $Z = A_1 < A_1 D_6$ acting as $(1^{[r]},1^{[s]} \otimes 1^{[t]} \otimes 1^{[u]} + 1^{[v]} \otimes 1^{[w]})$ ($p=2$; $s,t,u$ distinct). For most of the subgroups, the $E_7$-conjugacy classes are just the $A_1 D_6$-conjugacy classes. This is the case for subgroups $E_7(\#1)$--$E_7(\#6)$, $E_7(\#10)$, $E_7(\#11)$, $E_7(\#14)$ (we can check that $E_7$ does not fuse any of the $A_1 D_6$-conjugacy classes by considering the composition factors on $V_{56}$ given in Table \ref{E7tabcomps}). 

All of the remaining $A_1 D_6$-irreducible $A_1$ subgroups are contained in $A_1^3 D_4$, a maximal rank connected subgroup. By \cite[Table 10]{car}, we have $N_{E_7}(A_1^3 D_4) = (A_1^3 D_4).S_3$ where the $S_3$ acts simultaneously as the outer automorphism group of $A_1^3$ and $D_4$. First, suppose that the projection of $X$ to $D_4$ acts as $6^{[r]} + 0$ $(p \geq 7)$ or $4^{[r]} + 2^{[r]}$ $(p \geq 5)$ on $V_{D_4}(\lambda_1)$. Then the projection of $X$ is contained in the centraliser of both a triality automorphism (since $X$ is contained in a $G_2$ or $A_2$ respectively) and an involutory automorphism of $D_4$ (since $X$ is contained in a $B_3$ or $A_1 B_2$ respectively). The conjugacy classes of $E_7(\#7)$ and $E_7(\#9)$ follow. Next, assume the projection of $X$ to $D_4$ acts as $4^{[r]} + 2^{[s]}$ $(p \geq 3; r \neq s)$ or $3^{[r]} \otimes 1^{[s]}$ $(p\geq 5; r \neq s)$. Then, using the triality automorphism, we assume that $X$ acts as $4^{[r]} + 2^{[s]}$ on $V_{D_4}(\lambda_1)$, hence excluding $Y$ from Table~\ref{E7tab}. The projection of $X$ to $D_4$ is contained in $A_1 B_2$ and is therefore fixed by an involution in the outer automorphism group of $D_4$. By considering composition factors, we see that this involution swaps the two $A_1$ factors contained in $D_6$. The conjugacy classes of $X$ now follow, which are $E_7(\#8)$ in Table \ref{E7tab}. The same argument applies when the projection to $D_4$ acts as $0 | (2^{[r]} + 2^{[s]} + 2^{[t]}) | 0$ ($p = 2$; $r < s < t$) or $1^{[r]} \otimes 1^{[s]} \otimes 1^{[t]}$ ($p =2$; $r, s, t$ distinct). Therefore, we exclude $Z$ from Table \ref{E7tab} and obtain the conjugacy classes $E_7(\#13)$. 

The last possibility to consider is when the projection of $X$ to $D_4$ is contained in $\text{SO}_4  \text{SO}_4$ and so $X$ is $E_7(\#12)$. In this case $X$ is contained in $A_1^7$ and $N_{E_7}(A_1^7) = (A_1^7).\text{PSL}(2,7)$  by \cite[Table 10]{car}. The automorphism group of the Fano plane is $\text{PSL}(2,7)$ and this leads to an isomorphism between $\text{PSL}(2,7)$ and the subgroup of $S_7$ generated by $(1,2,3) (5,6,7)$ and $(2,4) (3,5)$. This subgroup is 2-transitive and has two orbits on sets of 3 points. One orbit is made up of all 3 point sets that form a line in the Fano plane and the other orbit is the 3 point sets that do not form a line. 

As $\text{PSL}(2,7)$ is transitive we may assume that the field twist in the first $A_1$ is zero and we let $X \hookrightarrow A_1^7$ via $(1,1^{[r]},1^{[s]},1^{[t]},1^{[u]},1^{[v]},1^{[w]})$. In particular, we have fixed that $(0,r,s)$, $(0,w,v)$ and $(0,t,u)$ form lines in the Fano plane. We claim that the $7$-tuples $0, r, \ldots, w$ satisfying the conditions in Table \ref{conditions1} yield a set of representatives of the $E_7$-conjugacy classes of $X$ without repetition. We will prove the first few lines of Table \ref{conditions1}. The others are similar and easier. 

Assume $r, \ldots, w$ are all non-zero, so we are in the first seven rows of Table~\ref{conditions1}. By Lemma \ref{class}, we have $X$ is $A_1 D_6$-irreducible if and only if the following conditions hold: the sets $\{r,s\}$, $\{t,u\}$ and $\{v,w\}$ are distinct and at most one of the sets has cardinality one. The stabiliser in $\text{PSL}(2,7)$ of a point is isomorphic to $S_4$.

First suppose $r, \ldots, w$ are all distinct. The action of $S_4$ on $r, \ldots, w$ is given by the natural action of $S_4$ on pairs of $\{1,2,3,4\}$. This action of $S_4$ on $r , \ldots, w$ is transitive and so we may assume that $r$ is the smallest integer in $r, \ldots, w$. Now consider the stabiliser in $S_4$ of $r$, a Klein four-group, $V_4$. Since $0,r,s$ form a line, it follows that $s$ is fixed and so no further conditions can be imposed on $s$. The action of $V_4$ on $t,u,v,w$ allows us to assume that $t$ is the smallest integer of $t,u,v,w$. The stabiliser of $t$ in $V_4$ is trivial and hence we have the conditions given in the first row of Table \ref{conditions1}.   

Next, suppose exactly two of $r, \ldots, w$ are the same. The action of $S_4$ on $r, \ldots, w$ has two orbits on pairs. Hence we assume either $r=s$ or $r=t$. If $r=s$, then the stabiliser of the pair $r, s$ is isomorphic to $\text{Dih}_8$, the dihedral group of order $8$. The action of $\text{Dih}_8$ on $t,u,v,w$ is transitive. Therefore we may assume that $t$ is the smallest integer of these. The stabiliser of $t$ in $\text{Dih}_8$ is isomorphic to $\mathbb{Z}_2$, swapping $v$ and $w$ and we therefore assume $v < w$. If $r=t$, then the stabiliser in $S_4$ of $\{r,t\}$ is $\mathbb{Z}_2$, swapping $s$ and $u$, and so we assume $s < u$. This yields the second and third rows of Table~\ref{conditions1}. 

For the final example, suppose exactly three of $r, \ldots, w$ are the same. The action of $S_4$ on $r, \ldots, w$ has two orbits on triples. Hence we assume either $r=s=t$ or $r=t=w$. The stabiliser of $r,s,t$ in $S_4$ is trivial and so no further conditions can be imposed. The stabiliser of $r,t,w$ is isomorphic to $\mathbb{Z}_3$, acting as a $3$-cycle on $t,u,v$. Therefore, we assume that $t$ is the smallest. 

We now need to show that the subgroups $E_7(\#1)$--$E_7(\#14)$ are $E_7$-irreducible. 

The subgroups with ID numbers $1$, $2$ $(p > 7)$, $3$ $(p > 5)$, $4$, $5$ $(p > 11)$, $6$ $(p  > 7)$, $7$ $(p  > 7)$, $8$ $(p >5)$, $9$ $(p >5)$,  $10$ $(p > 5)$, $11$ and $12$ $(p >3)$ are all $E_7$-irreducible by Corollary \ref{notrivs} (the composition factors of $L(E_7) \downarrow X$ are listed in Table \ref{E7tabcomps}). 

In many of the remaining cases, Corollary \ref{notrivs} still applies. We present the cases where we use Lemma \ref{wrongcomps} with Table \ref{levie7} to prove the remaining subgroups are $E_7$-irreducible. The arguments are all very similar and so we will omit the details for some of them. 

Firstly, consider $X = E_7(\#2)$ when $p=7$. Then Corollary \ref{notrivs} applies unless $r=s=t-1$ in which case $X$ has one trivial composition factor on $L(E_7)$ and $L(E_7) \downarrow X = 22 /$ $\!\!  18^2 /$ $\!\!  16 /$ $\!\!  14^2 /$ $\!\!  12^2 /$ $\!\!  10^3 /$ $\!\!  8 /$ $\!\!  6 /$ $\!\!  4 /$ $\!\!  2^5 /$ $\!\!  0$. We will use Lemma \ref{wrongcomps} by showing that the composition factors of any irreducible subgroup $A_1$ of a Levi subgroup on $L(E_7)$ are not the same as those of $X$. Suppose, for a contradiction, that $Y$ is an $L$-irreducible subgroup $A_1$ of a Levi subgroup $L$, having the same composition factors as $X$ on $L(E_7)$. Since $X$ has only one trivial composition factor on $L(E_7)$, we have $L'$ has only one trivial composition factor on $L(E_7)$. Therefore, using Table \ref{levie7}, we find the possibilities for $L'$ are $E_6$, $A_1 D_5$, $A_6$, $A_1 A_5$, $A_2 A_4$ and $A_1 A_2 A_3$. Since $p = 7$, Lemma \ref{class} shows there are no $A_6$-irreducible $A_1$ subgroups and so we immediately rule out $L' = A_6$. Suppose $Y$ is contained in $E_6$. Then by Theorem \ref{thmE6}, we see that $Y$ is conjugate to a subgroup in Table \ref{E6tab}. Using the composition factors in Table \ref{E6tabcomps}, we find that the composition factors of $L(E_7) \downarrow E_6(\#n)$ are not the same as $L(E_7) \downarrow X$ for any $n$ and hence $Y$ is not contained in $E_6$. Now suppose $Y$ is contained in $A_1 D_5$. From Table \ref{levie7}, we see that $(V_{A_1}(0),V_{D_5}(\lambda_1))$ occurs as a multiplicity two composition factor of $L(E_7) \downarrow A_1 D_5$. As there is no combination of composition factors of $L(E_7) \downarrow X$ that form two isomorphic 10-dimensional modules, it follows that $Y$ is not contained in $A_1 D_5$. Using Table \ref{levie7}, we see that $A_1 A_5$ has a $2$-dimensional composition factor and hence $Y$ is not contained in $A_1 A_5$. Now suppose $Y$ is contained in $A_2 A_4$. Since $Y$ is $A_2 A_4$-irreducible, it follows that $Y$ acts as $2^{[r]} \otimes 4^{[s]}$ on $(V_{A_2}(10),V_{A_4}(1000))$. Both $(V_{A_2}(00),V_{A_4}(1000))$ and $(V_{A_2}(00),V_{A_4}(0001))$ occur as composition factors of $L(E_7) \downarrow A_2 A_4$ and hence $Y$ has two $5$-dimensional composition factors on $L(E_7)$, a contradiction. Finally, suppose $Y$ is irreducibly contained in $A_1 A_2 A_3$. Consider the composition factors of $L(E_7) \downarrow A_1 A_2 A_3$ given in Table \ref{levie7}. Since all $A_1$-modules are self-dual it follows that the restriction to $Y$ of the composition factors $(V_{A_1}(2),V_{A_2}(00),V_{A_3}(000))$, $(V_{A_1}(0),V_{A_2}(11),V_{A_3}(000))$ and $(V_{A_1}(0),V_{A_2}(00),V_{A_3}(101))$ yield a copy of each non-trivial odd-multiplicity composition factor of $L(E_7) \downarrow X$. This is a contradiction, because the sum of the dimensions of one copy of each non-trivial odd-multiplicity composition factor is 46, which is greater than 26. Therefore, no such subgroup $Y$ exists and $X$ is indeed $E_7$-irreducible by Lemma \ref{wrongcomps}.     

Next we consider $X = E_7(\#3)$. If $p=5$, then Corollary \ref{notrivs} applies unless $r=s=u-1$, in which case $X$ has one trivial composition factor on $L(E_7)$. A similar argument to the previous one shows that $X$ is $E_7$-irreducible. If $p=3$, then there are more cases when Corollary \ref{notrivs} does not apply. If $r=u+1$, $r=s=u-1$, $r=s=t-1=u-2$ or $r=t=s-1=u-2$ then $X$ has one trivial composition factor on $L(E_7)$ and $X$ is $E_7$-irreducible by a similar argument to before. The only other case where Corollary \ref{notrivs} does not apply is $r=u$ (so $r,s,t$ are distinct), in which case $X$ has two trivial composition factors on $L(E_7)$. The composition factors of $X$ on $L(E_7)$ are $4^{[r]} /$ $\!\!  2^{[r]} \otimes 4^{[s]} /$ $\!\!  (2^{[r]} \otimes 2^{[s]} \otimes 2^{[t]})^2 /$ $\!\!  (2^{[r]})^5 /$ $\!\!  4^{[s]}  \otimes 2^{[t]} /$ $\!\!  4^{[s]} /$ $\!\!  2^{[s]} \otimes 2^{[t]} /$ $\!\!  2^{[s]} /$ $\!\!  (2^{[t]})^2 /$ $\!\!  0^2$. Suppose, for a contradiction, that $Y$ is an $L$-irreducible subgroup $A_1$ of a Levi subgroup $L$, having the same composition factors as $X$ on $L(E_7)$. Then using Table \ref{levie7}, we find the possibilities for $L'$ are $E_6$, $A_1 D_5$, $A_6$, $A_1 A_5$, $A_2 A_4$, $A_1 A_2 A_3$ and $A_1 A_4$. Since $p=3$, neither $A_6$ nor $A_4$ contain an irreducible subgroup $A_1$ and so we immediately rule out $A_6$, $A_2 A_4$ and $A_1 A_4$. The only $E_6$-irreducible $A_1$ subgroups when $p=3$ are $E_6(\#2)$ and $E_6(\#3)$. Using Tables \ref{E6tabcomps} and \ref{levie7}, we check that neither subgroup has the same composition factors as $X$ on $L(E_7)$ for any $r,s,t$. Now suppose $Y$ is contained in $A_1 D_5$. Since $p=3$, the projection of $Y$ to $D_5$ acts on $V_{D_5}(\lambda_1)$ as $2^{[a]} + 2^{[b]} + 1^{[c]} \otimes 1^{[d]}$ $(a \neq b)$ or $2^{[a]} \otimes 2^{[b]} + 0$ ($a \neq b$). In the first case when $c=d$ or in the second case, $L(E_7) \downarrow Y$ has at least three trivial composition factors since $(V_{A_1}(0),V_{D_5}(\lambda_1))$ is a multiplicity two composition factor of $L(E_7) \downarrow A_1 D_5$, which is a contradiction. Now suppose $c \neq d$. Then by considering the multiplicity of the 3-dimensional composition factors of $X$ on $L(E_7)$, it follows that $\{a,b\} = \{r,t\}$. But then the projection of $Y$ to $D_5$ will have a $2^{[r]} \otimes 2^{[t]}$ composition factor on  $V_{D_5}(\lambda_2)$, a contradiction. We can also rule out $Y$ being contained in $A_1 A_5$, since $X$ and hence $Y$, has no $2$-dimensional composition factors on $L(E_7)$. Finally, we rule out $A_1 A_2 A_3$ as it has no composition factors of dimension at least 27. Therefore, no such $Y$ exists and $X$ is $E_7$-irreducible by Lemma \ref{wrongcomps}.       

Now let $p=11$, and $X_1 = E_7(\#5^{\{1,0,0\}})$ and $X_2 = E_7(\#5^{\{0,0,1\}})$ (Corollary \ref{notrivs} applies for all of the other cases). Then from Table \ref{E7tabcomps}, we find that $V_{56} \downarrow X_1 = 19 / 13 / 11 / 9^2 / 5 / 3$ and $V_{56} \downarrow X_2 = 23 / 21 / 15 / 9 / 7$. In particular, neither $X_1$ nor $X_2$ have a trivial composition factor on $V_{56}$. Then by Table \ref{levie7}, if there exists a subgroup $A_1$ having the same composition factors as $X$ on $V_{56}$ contained in a Levi subgroup, it will be contained in one of the following Levi subgroups: $D_6$, $A_1 D_5$, $A_6$, $A_1 A_5$, $A_2 A_4$ or $A_1 A_2 A_3$. The dimensions of composition factors of $X_1$ and $X_2$ on $V_{56}$ are $18, 10^2, 6^2, 4, 2$ and $22, 10^2, 8, 6$, respectively. Using Table \ref{levie7}, we see that this is incompatible with any subgroup of such a Levi subgroup. Hence Lemma \ref{wrongcomps} shows that both $X_1$ and $X_2$ are $E_7$-irreducible. 

Similarly, let $p=7$ and $X = E_7(\#6^{\{1,0,0\}})$. Then $V_{56} \downarrow X = 13 / 11 / 9 / 7 / 5^2 / 3^2$ with dimensions $14,10,6^3,4^2,2$. These dimensions are incompatible with any subgroup of a Levi factor, using Table \ref{levie7}. Hence $X$ is $E_7$-irreducible by Lemma \ref{wrongcomps}. Similar arguments show that $E_7(\#7)$ $(p = 7)$, $E_7(\#8)$ and $E_7(\#9)$ (both with $p=5)$ are $E_7$-irreducible. 

Now consider $E_7(\#10)$. Firstly, if $p = 5$ then the only case for which Corollary \ref{notrivs} does not apply is $X = E_7(\#10^{\{0,0,1,0\}})$. From Table \ref{E7tabcomps}, we have $V_{56} \downarrow X = 17 /$ $\!\!  15 /$ $\!\!  13 /$ $\!\!  11 /$ $\!\!  9 /$ $\!\!  7 /$ $\!\!  3 /$ $\!\!  1$. The dimensions of these composition factors are incompatible with any subgroup of a Levi factor and hence $X$ is $E_7$-irreducible by Lemma \ref{wrongcomps}. 

Now suppose $p=3$. There are many cases where Corollary \ref{notrivs} does not apply. Let $X_1 = E_7(\#10^{\{0,0,1,0\}})$ and $X_2 = E_7(\#10^{\{1,0,1,1\}})$. Then both $X_1$ and $X_2$ have three trivial composition factors on $L(E_7)$. Suppose $Y$ is a subgroup of a Levi factor $L$ having the same composition factors as $X_1$ on $V_{56}$ and $L(E_7)$. Using Table \ref{levie7} and the number of trivial composition factors on $L(E_7)$, it follows that $Y$ is an $L'$-irreducible subgroup of $L' = E_6, D_6, A_6, A_1 D_5, A_1 A_5, A_2 A_4, A_1 A_4$ or $A_1 A_2 A_3$. From Table \ref{E7tabcomps}, we have $V_{56} \downarrow X_1 = 11 / 9^2 / 7^3 / 5^2 / 3^5 / 1^3$. Therefore $Y$ is not a subgroup of $E_6$, $A_6$, $A_2 A_4$ or $A_1 A_4$ by considering the dimensions of the composition factors on $V_{56}$. Suppose $Y$ is contained in $D_6$. Then by Lemma \ref{class}, we have $V_{D_6}(\lambda_1) \downarrow Y = 2^{[r]} \otimes 2^{[s]} + 2^{[t]}$ $(r \neq s)$, $2^{[r]} + 2^{[s]} + 2^{[t]} + 2^{[u]}$ ($r,s,t,u$ distinct) or $1^{[r]} \otimes 1^{[s]} + 1^{[t]} \otimes 1^{[u]} + 1^{[v]} \otimes 1^{[w]}$ (the sets $\{r,s\}$, $\{t,u\}$, $\{v,w\}$ are distinct and at least two of them have cardinality two). But restricting from $D_6$, we find that $Y$ has a 9-dimensional, 3-dimensional or 4-dimensional composition factor on $V_{56}$, respectively, which is a contradiction. Now suppose $Y$ is contained in $A_1 D_5$. Then by Lemma \ref{class}, we have $(V_{A_1}(1),V_{D_5}(\lambda_1)) \downarrow Y = 1^{[r]} \otimes 2^{[s]} \otimes 2^{[t]} + (1^{[r]})^2$ ($s \neq t$) or $1^{[r]} \otimes 1^{[s]} \otimes 1^{[t]} + 1^{[r]} \otimes 2^{[u]} + 1^{[r]} \otimes 2^{[v]}$ ($u \neq v$ and if $s = t$ then $s,u,v$ distinct). This leads to a contradiction as the composition factors of $Y$ do not match those of $X_1$. Similarly, if $Y$ is contained in $A_1 A_2 A_3$ then $V_{56} \downarrow Y$ has a 4-dimensional composition factor, a contradiction. Finally, suppose $Y$ is contained in $A_1 A_5$. Then $(V_{A_1}(1),V_{A_5}(\lambda_1)) \downarrow Y = 1^{[r]} \otimes 2^{[s]} \otimes 1^{[t]}$ ($s \neq t$) and $V_{56} \downarrow Y = (1^{[r]} \otimes 2^{[s]} \otimes 1^{[t]})^2 / (2^{[s]} \otimes 1^{[t]})^3 / (1^{[s+1]} \otimes 1) / 1^{[t+1]} / (1^{[t]})^2$. Hence $Y$ has a 4-dimensional composition factor again, a contradiction. Therefore $Y$ does not exist and $X_1$ is $E_7$-irreducible. The proof is almost identical for $X_2$ and is similar and easier for the other cases as they all have fewer trivial composition factors on $L(E_7)$. 

Similar arguments show that $E_7(\#11)$ is $E_7$-irreducible when $p=3,5$ and $E_7(\#12)$ is $E_7$-irreducible when $p=3$. 

Now suppose $p=2$. First consider $X = E_7(\#14)$. Suppose $Y$ is an $L'$-irreducible subgroup $A_1$ of a Levi subgroup $L$ with the same composition factors as $X$ on $V_{56}$ .  From Table \ref{E7tabcomps}, we see that $V_{56} \downarrow X$ has a 32-dimensional composition factor. Therefore, $L'=D_6$ and from Table \ref{levie7}, we have $V_{56} \downarrow D_6 = \lambda_1^2 / \lambda_5$. It follows that the remaining composition factors of $V_{56} \downarrow Y$ have even multiplicity, a contradiction. Therefore $X$ is $E_7$-irreducible by Lemma \ref{wrongcomps}.  

Finally, let $X$ be $E_7(\#12)$ or $E_7(\#13)$. Using Lemma \ref{bada1p2} and Table \ref{levie7}, we find the composition factors on $V_{56}$ of each $L$-irreducible subgroup $A_1$ of a Levi factor $L'$. We carefully check that they do not match the composition factors of $X$ on $V_{56}$. Thus $X$ is $E_7$-irreducible by Lemma \ref{wrongcomps}. This completes the analysis of the case $M = A_1 D_6$. 

The next case to consider is $M = A_7$. By Lemma \ref{class}, it follows that $X$ acts on $V_{A_7}(\lambda_1)$ as $7$ $(p \geq 11)$, $1 \otimes 1^{[r]} \otimes 1^{[s]}$ $(0 < r < s)$ or $3^{[r]} \otimes 1^{[s]}$ ($p \geq 5$; $r \neq s$). In the first two cases $X$ preserves a symplectic form on $V_{A_7}(\lambda_1)$ and hence $X$ is contained in $C_4$. By \cite[Lemma~6.1]{tho1}, this subgroup $C_4$ is $E_7$-reducible and hence so is $X$. In the final case, $X$ acts as $3^{[r]} \otimes 1^{[s]}$ ($p \geq 5$; $r \neq s$) and is hence contained in $D_4$. The normaliser of $D_4$ in $E_7$ contains a triality automorphism of $D_4$, by \cite[Lemma 2.15]{clss}. Hence $X$ is $E_7$-conjugate to an $A_7$-reducible subgroup $A_1$ acting as $4^{[r]} + 2^{[s]}$ and there are no $E_7$-irreducible $A_1$ subgroups contained in $M$.  

Now let $M = A_2 A_5$. Then using Lemma \ref{class}, we see that the projection of $X$ to $A_5$ is contained in $C_3$ and $p \geq 3$. By \cite[Table 8.2]{LS3}, the connected centraliser of this subgroup $C_3$ is $G_2$ and hence the factor $A_2$ of $M$ is contained in $G_2$. Moreover, by Theorem \ref{thmG2} the $A_2$-irreducible subgroup $A_1$ is contained in $\bar{A}_1 A_1$ and hence $X$ is contained in $\bar{A}_1 A_1 C_3 < \bar{A}_1 D_6$. Therefore $X$ has already been considered.  

Now suppose $M = G_2 C_3$. We note that the factor $G_2$ is contained in $D_4$ and hence subgroups generated by long root subgroups of $G_2$ are generated by long root subgroups of $E_7$. By Theorem \ref{thmG2}, the projection of $X$ to $G_2$ is contained in $\bar{A}_1 A_1$ or is maximal with $p \geq 7$. If the projection of $X$ to $G_2$ is contained in $\bar{A}_1 A_1$ then $X$ is contained in $\bar{A}_1 D_6$ and has already been considered. We therefore assume the projection of $X$ to $G_2$ is maximal and so $p \geq 7$. Using Lemma \ref{class}, we find that the projection of $X$ to $C_3$ is contained in $\bar{A}_1 C_2$, $A_1 A_1$ or is maximal. If the projection of $X$ is contained in $\bar{A}_1 C_2$ then $X$ is contained in $\bar{A}_1 D_6$ and has already been considered. Now suppose the projection is contained in $A_1 A_1$ acting as $(2,1)$ on $V_{C_3}(100)$. Then $X \hookrightarrow A_1 A_1 A_1 < G_2 C_3$ via $(1^{[r]},1^{[s]},1^{[t]})$ $(rst=0; s \neq t)$. If $r=s$ we claim that $X$ is contained in $\bar{A}_1 D_6$. To show this we first note that $X$ is also contained in $A_1 A_1 G_2 < A_1 F_4$, since the factor $G_2$ of $A_1 A_1 G_2$ is contained in a $D_4$ Levi subgroup and is hence conjugate to the factor $G_2$ of $M$. It follows that $X$ is conjugate to $Y \hookrightarrow A_1 A_1 A_1 < A_1 A_1 G_2 < A_1 F_4$ via $(1^{[t]},1^{[r]},1^{[r]})$. Moreover, by Theorem \ref{thmF4}, we have $Y$ is conjugate to a subgroup of $A_1 \bar{A}_1 C_3 < A_1 F_4$ and hence to a subgroup of $\bar{A}_1 D_6$. Specifically, $X$ is conjugate to $E_7(\#1^{\{r,r,t\}})$ and is $E_7$-irreducible. If $r \neq s$ then $X$ is $E_7$-irreducible by Corollary \ref{notrivs}, yielding $E_7(\#16)$. Finally, suppose the projection of $X$ to $C_3$ is maximal, so $X \hookrightarrow A_1 A_1 < G_2 C_3$ via $(1^{[r]},1^{[s]})$ $(rs=0)$. If $r \neq s$ then $X$ is $E_7$-irreducible by Corollary \ref{notrivs}, giving $E_7(\#15)$. If $p > 7$ and $r=s=0$ then Theorem \ref{A1samecomp} shows that $X$ is conjugate to $E_7(\#5^{\{0,0,0\}})$ in $\bar{A}_1 D_6$. 

When $p=7$ and $r=s=0$ then we note a correction to \cite[Theorem 8.13]{bon} and show that $X$ is $E_7$-reducible. This is almost shown in the proof of \cite[Lemma 4.6]{LS4} but we provide the full argument here. Assume $X$ is $E_7$-irreducible. This is almost shown in the proof of \cite[Lemma 4.6]{LS4} but we provide the full argument here.  From Theorem \ref{maximalexcep}, we have $V_{56} \downarrow G_2 C_3 = (10,100) / (0,001)$ and hence $V_{56} \downarrow X = 11 / 9^2 / 7 / 5^2 / 3^4 / 1^2$. It is easy to check that the following Weyl modules have the indicated structure: $W(11) = 11 | 1$, $W(9) = 9 | 3$, $W(7) = 7 | 5$, $W(5) = 5$, $W(3) = 3$. By \cite[II 4.14]{Jan}, only $11$ extends $1$ and $\text{Ext}_{A_1}^1(11,1) \cong K$, so $X$ stabilises a module $W \cong 1$. We wish to investigate $N := N_{E_7}(W)^\circ$. The variety of all 2-spaces in $V_{56}$ has dimension 108 and so $N$ has dimension at least 25 ($=\text{dim}(E_7) - 108$). Consider a maximal connected subgroup $M_1$ containing $N$ and hence $X$. This subgroup is reductive (otherwise $X$ is $E_7$-reducible, a contradiction) and hence listed in Theorem \ref{maximalexcep}. The possibilities for $M_1$ are $A_7$, $A_1 D_6$, $A_2 A_5$, $A_1 F_4$ and $G_2 C_3$.  Since $X$ is $E_7$-irreducible and contained in $N$, it follows that $M_1$ contains an $E_7$-irreducible subgroup $A_1$ with the same composition factors as $X$ on $V_{56}$. By the previous cases, $A_7$ does not contain any $E_7$-irreducible $A_1$ subgroups and so $M_1$ is not $A_7$. Now suppose $M_1 = A_1 D_6$. Then $X$ is conjugate to $E_7(\#n)$ where $n$ is one of $1,2,3,6,7, \dots,12$. Using the composition factors given in Table \ref{E7tabcomps} we see this is not possible. Next, suppose $M_1 = A_2 A_5$. Then since all $A_2 A_5$-irreducible $A_1$ subgroups are contained in $A_1 D_6$, this is also impossible. Suppose $M_1 = A_1 F_4$. Since $p=7$, the subgroup $A_1 F_4$ does not fix a $2$-space on $V_{56}$ and therefore $N$ is properly contained in $A_1 F_4$. Since $N$ has dimension at least 25, it follows from Theorem \ref{maximalexcep} that $N$ is contained in $A_1 B_4$ or $A_1 \bar{A}_1 C_3$. In both cases, it follows that $N$ is contained in $A_1 D_6$ (see the $M = A_1 F_4$ case below). Hence $X$ is contained in $A_1 D_6$, which is again impossible. Finally, suppose $M_1 = G_2 C_3$. Since $G_2 C_3$ does not fix a $2$-space of $V_{56}$, we have $N$ is contained in a proper reductive, maximal connected subgroup of $G_2 C_3$, which contains $X$. The only $A_1$ subgroups of $G_2 C_3$ with the same composition factors as $X$ on $V_{56}$ are all $G_2 C_3$-conjugate to $X$. It follows that the only reductive, connected proper subgroups of $G_2 C_3$ containing $X$ are $A_1 C_3$, $G_2 A_1$ and $A_1 A_1$, where the factor $A_1$ subgroups are maximal in their respective factors of $G_2 C_3$. All three subgroups have dimension less than $25$. This is a contradiction, proving $X$ is $E_7$-reducible.  

Next, we suppose $M = A_1 G_2$ ($p \neq 2)$. By Theorem \ref{thmG2}, the projection of $X$ to $G_2$ is contained in $A_1 A_1$ or is maximal with $p \geq 7$. Consider the first case. We claim that $X$ is also contained in $A_1 D_6$ and has already been considered. This follows by calculating the centraliser in $E_7$ of the involution that the $A_1 A_1$ centralises in $G_2$, and finding it to be $A_1 D_6$. Now suppose the projection of $X$ to $G_2$ is maximal, so $p \geq 7$ and $X \hookrightarrow A_1 A_1$ via $(1^{[r]},1^{[s]})$ $(rs=0)$. If $r \neq s$ then $X$ is $E_7$-irreducible by Corollary \ref{notrivs}, yielding $E_7(\#17)$ in Table \ref{E7tabcomps}. If $p \geq 11$ and $r=s=0$ then Theorem \ref{A1samecomp} shows that $X$ is conjugate to $E_7(\#5^{\{0,0,0\}})$, a subgroup of $A_1 D_6$. Another correction to \cite[Theorem 8.13]{bon} is that if $p=7$ and $r=s=0$ then $X$ is $E_7$-reducible. This follows immediately from the argument given in the case $M = G_2 C_3$ because here again we have a subgroup $A_1$, $X$, such that $V_{56} \downarrow X = 11 / 9^2 / 7 / 5^2 / 3^4 / 1^2$ and the argument only relied upon the composition factors on $V_{56}$.

Now let $M = A_1 F_4$. Theorem \ref{thmF4} shows that the projection of $X$ to $F_4$ is contained in $B_4$, $\bar{A}_1 C_3$ ($p \neq 2$), $A_1 G_2$ $(p \neq 2)$ or $A_1$ $(p \geq 13)$. Any subgroup of $A_1 B_4$ is contained in $\bar{A}_1 D_6$. Indeed, $B_4$ (or its Lie algebra if $p=2$) has a non-trivial centre and the full centraliser of this centre is $\bar{A}_1 D_6$. Similarly, if $X$ is contained in $A_1 \bar{A}_1 C_3$ then it is contained in $\bar{A}_1 D_6$ because the connected centraliser of $\bar{A}_1$ in $E_7$ is $D_6$. We saw in the $M = G_2 C_3$ case that $A_1 A_1 G_2$ is contained in $G_2 C_3$ and so $X$ has already been considered when its projection to $F_4$ is contained in $A_1 G_2$. That leaves us to consider $X \hookrightarrow A_1 A_1 < A_1 F_4$ via $(1^{[r]},1^{[s]})$ ($p \geq 13$; $rs=0$), where the second factor $A_1$ is maximal in $F_4$. In this case Corollary \ref{notrivs} shows that $X$ is $E_7$-irreducible, yielding $E_7(\#18)$.

Now suppose $M = A_2$ $(p \geq 5)$. Then $X$ acts on $V_{A_2}(10)$ as 2. First, let $p \geq 11$. By Theorem \ref{maximalexcep}, we have $L(E_7) \downarrow M = 44 / 11$. From this, it follows that $L(E_7) \downarrow X = 16 / 14 / 12^2 / 10^2 / 8^3 / 6 / 4^3 / 2 / 0$. By Theorem \ref{maximalexcep}, we have $L(E_7) \downarrow A_7 = (\lambda_1 + \lambda_7) / \lambda_4$. Letting $Y = A_1 < A_7$ with $V_{A_7}(\lambda_1) \downarrow Y = 7$, it follows that $Y$ has the same composition factors as $X$ on $L(E_7)$. Since $p \geq 11 > 7 = N(A_1,E_7)$, Theorem \ref{A1samecomp} applies. Hence $X$ is conjugate to $Y$, which is contained in a parabolic subgroup of $E_7$. Therefore $X$ is $E_7$-reducible. 

For $p = 5,7$ we show that $X$ fixes a $1$-space of $V_{56}$. It then follows that $X$ is contained in a parabolic subgroup of $E_7$ since the dimension of the centraliser of this $1$-space is at least 77. From \cite[Table 10.2]{LS1}, we see that $V_{56} \downarrow M = 60 + 06$ $(p = 7)$ and $V_{56} \downarrow M = 22 | (60 + 06) | 22$ $(p=5)$. When $p=7$, we have $V_{A_2}(60) = S^6(V_{A_2}(10))$ and restricting to $X$ yields $S^6(2) = (0|12|0) + (4|8|4)$ (this final calculation follows since $p > 6$ and thus if $W$ is tilting then so is $S^6(W)$). Therefore $X$ fixes a $1$-space of $V_{56}$. When $p=5$, we have $V_{A_2}(20) \otimes V_{A_2}(02) = (11|22|11) + 00$ and restricting to $X$ gives $(4+0) \otimes (4+0) = (0|8|0) + (2|6|2)  + 4^3 + 0$. Since $V_{A_2}(11) \downarrow X = 4 + 2$, it follows that $V_{A_2}(22) \downarrow X =  (0|8|0) + 6 + 4$ and $X$ fixes a $1$-space of $V_{56}$. In both cases $X$ fixes a $1$-space and is hence $E_7$-reducible.

Now let $M = A_1 A_1$ $(p \geq 5)$. Then $X$ is a diagonal subgroup of $M$ embedded via $(1^{[r]},1^{[s]})$ $(rs=0)$. If $r \neq s$ then $X$ is $E_7$-irreducible by Corollary \ref{notrivs}, yielding $E_7(\#19)$ in Table \ref{E7tab}. Now suppose $r=s=0$. If $p > 7$ then Theorem \ref{A1samecomp} shows that $X$ is conjugate to $E_7(\#6^{\{0,0,0\}})$ and is hence $E_7$-irreducible. If $p=7$, we claim that $X$ is also conjugate to $Y = E_7(\#6^{\{0,0,0\}})$. Restricting from $M$ and $A_1 D_6$, we see that $X$ and $Y$ have the same composition factors on $L(E_7)$ and on $V_{56}$. We note that $Y$ was already shown to be $E_7$-irreducible when we considered $A_1 D_6$ above, using only the composition factors of $Y$ on $L(E_7)$ as Corollary \ref{notrivs} applies. Therefore, $X$ is also $E_7$-irreducible. To prove $X$ is conjugate to $Y$ we follow the proof of \cite[Lemma~6.7]{LS3}. From Table \ref{E7tabcomps}, we see that $L(E_7) \downarrow X$ has no composition factors of the form $5 \otimes c^{[1]}$ where $c > 0$. Since $X$ and $Y$ have the same composition factors on $L(E_7)$ they have the same labelled diagram. Hence the hypothesis of \cite[Lemma~6.7]{LS3} holds and the proof of it shows that $L(X) = L(Y_1)$, where $Y_1$ is a suitable $E_7$-conjugate of $Y$. Since $X$ is $E_7$-irreducible, we claim that $C:= C_{E_7}(L(X))^\circ = 1$. Indeed, $X$ normalises $C$ and so $C$ is reductive, otherwise $X$ would be contained in a parabolic subgroup of $E_7$. Furthermore, since $C$ is a connected reductive group, the connected group $X$ centralises $C$ and hence by Lemma \ref{semirr}, we have $C = 1$. Thus $C_{E_7}(L(X))^\circ = 1$ and $N_{E_7}(L(X))^\circ = X$, showing that $Y_1 = X$ and $X$ is $E_7$-conjugate to $Y$.    

If $p=5$ we note a final correction to \cite[Theorem 8.13]{bon}. We claim that $X \hookrightarrow M$ via $(1,1)$ is $E_7$-reducible. Suppose, for a contradiction, that $X$ is $E_7$-irreducible. First, we see that $V_{56} \downarrow X = 9 / 7^2 / 5^3 / 3^6 / 1^2$ (this follows from $V_{56} \downarrow M$ which is given in Theorem \ref{maximalexcep}). We claim the only composition factor that extends $1$ is $7 = 2 \otimes 1^{[1]}$ and that $\text{Ext}^1_{A_1}(7,1) \cong K$. This follows from \cite[II 4.14]{Jan} and the structure of the following Weyl modules: $W(9) = 9 = 4 \otimes 1^{[1]}$, $W(7) = 7 | 1$, $W(5) = 5 | 3$ and $W(3) = 3$. Since $V_{56}$ is self-dual, $V_{56} \downarrow X$ has a submodule $W \cong 1$. By a previous argument, $N := N_{E_7}(W)^\circ$ is of dimension at least 25 and we may assume that it is contained in a reductive, maximal connected subgroup $M_1$ of $E_7$. The possibilities for $M_1$ are $A_7$, $A_2 A_5$, $A_1 D_6$, $A_1 F_4$ and $G_2 C_3$. Since $X$ is $E_7$-irreducible and contained in $N$, it follows that $M_1$ contains an $E_7$-irreducible subgroup $A_1$ with the same composition factors as $X$ on $V_{56}$. By the previous cases, $A_7$ does not contain any $E_7$-irreducible $A_1$ subgroups and so $M_1$ is not $A_7$. Since $p=5$, it also follows that every $E_7$-irreducible subgroup $A_1$ of $A_2 A_5$, $G_2 C_3$ and $A_1 F_4$ is conjugate to a subgroup of $A_1 D_6$. Therefore, $A_1 D_6$ contains an $E_7$-irreducible subgroup $A_1$ with the same composition factors as $X$ on $V_{56}$. By the $M = A_1 D_6$ case, it follows that $E_7(\#n)$, where $n$ is one of $3, 8, 9, 10, 11$ or $12$, has the same composition factors as $X$ on $V_{56}$. Using Table \ref{E7tabcomps}, we see that this is a contradiction. Therefore $X$ is $E_7$-reducible, as claimed.    

Now suppose $M$ is one of the two conjugacy classes of maximal $A_1$ subgroups in $E_7$. Then $M = X$ and $X$ is $E_7$-irreducible. This accounts for the subgroups $E_7(\#20)$ and $E_7(\#21)$. 

Finally, we check there are no more $E_7$-conjugacies between any of the irreducible $A_1$ subgroups by comparing the composition factors in Table \ref{E7tabcomps}.  
\end{proof}
\vspace{-0.6cm}

\section{Proof of Theorem \ref{thmE8}: $E_8$-irreducible $A_1$ subgroups} \label{secE8}

In this section we classify the $E_8$-irreducible $A_1$ subgroups of $E_8$.

\begin{thm*} \label{thmE8}

Suppose $X$ is an irreducible subgroup $A_1$ of $E_8$. Then $X$ is conjugate to exactly one subgroup of Table \ref{E8tab} and each subgroup in Table~\ref{E8tab} is irreducible. 

\end{thm*}

\begin{longtable}{>{\raggedright\arraybackslash}p{0.05\textwidth - 2\tabcolsep}>{\raggedright\arraybackslash}p{0.10\textwidth - 2\tabcolsep}>{\raggedright\arraybackslash}p{0.75\textwidth - 2\tabcolsep}>{\raggedright\arraybackslash}p{0.10\textwidth-\tabcolsep}@{}}

\caption{The $E_8$-irreducible $A_1$ subgroups of $E_8$ \label{E8tab}} \\

\hline

ID  & $M$ & $V_{M} \downarrow X$ & $p$ \\

\hline

1 & $D_8$ & $7^{[r]} \otimes 1^{[s]}$  $(rs=0; r \neq s)$ & $ \geq 11$ \\

2 & & $7^{[r]} \otimes 1^{[s]}$  $(rs=0; r \neq s)$ & $ \geq 11$  \\

3 & & $3 \otimes 3^{[r]}$  $(r \neq 0)$ & $\geq 5$  \\

4 & & $3 \otimes 3^{[r]}$  $(r \neq 0)$ & $\geq 5$  \\

5 & & $5^{[r]} \otimes 1^{[s]} + 1^{[t]} \otimes 1^{[u]}$ $(rstu=0; r \neq s)$ & $\geq 7$ \\

6 & & $2^{[r]} \otimes 1^{[s]} \otimes 1^{[t]} + 1^{[u]} \otimes 1^{[v]}$ ($r,s,t$ distinct; $rstu=0$; $u \leq v$; if $u = v$ then $s < t$) &  $\geq 3$ \\

7 & & $4^{[r]} \otimes 2^{[s]} + 0$  $(rs=0; r \neq s)$ & $\geq 5$  \\

8 & & $1 \otimes 1^{[r]} \otimes 1^{[s]} \otimes 1^{[t]}$ ($0 < r < s < t$) & all  \\

9 & & $1 \otimes 1^{[r]} \otimes 1^{[s]} \otimes 1^{[t]}$ ($0 < r < s < t$) & $ \geq 3$  \\

10 & & $14 + 0$ & $\geq 17$  \\

11 & & $12^{[r]} + 2^{[s]}$  $(rs=0)$ & $\geq 13$  \\

12 & & $10^{[r]} + 4^{[s]}$  $(rs=0)$  & $\geq 11$  \\     

13 & & $10^{[r]} + 1^{[s]} \otimes 1^{[t]} + 0$ $(rs=0; s < t)$ & $\geq 11$  \\

14 & & $8^{[r]} + 6^{[s]}$  $(rs=0)$ & $\geq 11$  \\

15 & & $8^{[r]} + 2^{[s]} + 1^{[t]} \otimes 1^{[u]}$ ($rst=0$; $t \leq u$; if $t=u$ then $s < t$) & $\geq 11$ \\

16 & & $6^{[r]} + 2^{[s]} \otimes 2^{[t]}$ ($rs=0$; $s < t$) & $\geq 7$ \\

17 & & $6^{[r]} + 4^{[s]} + 1^{[t]} \otimes 1^{[u]}$  ($rst = 0$; $t \leq u$) & $\geq 7$  \\

18 & & $6^{[r]} + 1^{[s]} \otimes 1^{[t]} + 1^{[u]} \otimes 1^{[v]} + 0$  ($rs = 0$; $s < t < u < v$) & $\geq 7$ \\

19 & & $6^{[r]} + 2^{[s]} + 2^{[t]} + 2^{[u]}$  ($rs=0$; $s < t < u$) & $\geq 7$ \\ 

20 & & $4 + 4^{[r]} + 4^{[s]}  + 0$  ($0 < r < s$) & $\geq 5$ \\

21 & & $4^{[r]} + 4^{[s]} + 2^{[t]} + 2^{[u]}$  ($rt = 0$; $r < s$; $t < u$) & $\geq 5$ \\ 

22 & & $4^{[r]} + 2^{[s]} + 3^{[t]} \otimes 1^{[u]}$  ($rsu = 0$; $r \neq s$; $r \leq t$; $t \neq u$; if $r=t$ then $s \leq u$)  & $\geq 5$ \\ 

23 & & $4^{[r]} + 2^{[s]} + 1^{[t]} \otimes 1^{[u]} + 1^{[v]} \otimes 1^{[w]}$ ($rstuvw=0$; see Table \ref{23conditions} for the further conditions on $r, \ldots, w$) & $\geq 5$ \\

24 & & $2^{[r]} \otimes 2^{[s]} + 2^{[t]} + 1^{[u]} \otimes 1^{[v]}$ ($rtu =0$; $r < s$; $u \leq v$; if $u=v$ then $u < t$) & $\geq 3$ \\

25 & & $2^{[r]} + 2^{[s]} + 2^{[t]} + 2^{[u]} + 1^{[v]} \otimes 1^{[w]}$ ($rv=0$; $r < s < t < u $; $v \leq w$; if $v=w$ then $u < v$) & $\geq 3$ \\ 

26 & & $1 \otimes 1^{[r]} + 1^{[s]} \otimes 1^{[t]} + 1^{[u]} \otimes 1^{[v]} + 1^{[w]} \otimes 1^{[x]}$ (see Table \ref{8conditions1} for conditions \newline on $r, \ldots, x$) & any \\

27 & & $0|(2^{[r]} + 2^{[s]} + 2^{[t]})|0 + 1^{[u]} \otimes 1^{[v]} + 1^{[w]}  \otimes 1^{[x]}$ ($ruv=0$; see Table \ref{27conditions} for the further conditions on $r, \ldots, x$)  & $=2$ \\

28 & & $0|(2 + 2^{[r]} + 2^{[s]})|0 + 1^{[t]} \otimes 1^{[u]} \otimes 1^{[v]}$ ($0 < r < s$; $t<u<v$; if $t=0$ then $r \leq u$; if $t=0$ and $r=u$ then $s \leq v$) & $=2$ \\

29 & & $0|(2^{[r]} + 2^{[s]} + 2^{[t]} + 2^{[u]} + 2^{[v]})|0 + 1^{[w]}  \otimes 1^{[x]}$ ($rw=0$; $r < s < t < u < v$; $w < x$)  & $=2$ \\

30 & & $0|(2 + 2^{[r]} + 2^{[s]} + 2^{[t]} + 2^{[u]} + 2^{[v]}+ 2^{[w]})|0$ ($0 < r < s < t < u < v < w$) & $=2$ \\

\hline

31 & $A_1 E_7$ & $(1^{[r]},E_7(\#15^{\{s,t\}}))$ $(rst=0; s \neq t)$ & $\geq 7$ \\

32 & & $(1^{[r]},E_7(\#16^{\{s,t,u\}}))$ $(rstu=0; s \neq t; t \neq u)$ & $\geq 7$ \\

33 & & $(1^{[r]},E_7(\#17^{\{s,t\}}))$ ($rst=0; s \neq t$) & $\geq 7$  \\

34 & & $(1^{[r]},E_7(\#18^{\{s,t\}}))$ ($rst=0$)& $\geq 13$  \\

35 & & $(1^{[r]},E_7(\#19^{\{s,t\}}))$ ($rst=0$; $s \neq t$)  & $\geq 5$  \\

36 & & $(1^{[r]},E_7(\#20)^{[s]})$ $(rs=0)$ & $\geq 17$  \\

37 & & $(1^{[r]},E_7(\#21)^{[s]})$ $(rs=0)$ & $\geq 19$  \\

\hline

38 & $G_2 F_4$ & $(G_2(\#3)^{[r]},F_4(\#10)^{[s]})$ $(rs=0; r \neq s)$ & $\geq 13$ \\

39 & & $(G_2(\#3)^{[r]},F_4(\#11^{\{s,t\}}))$ ($rs=0$; $r < t$; $r \neq s$; $s \neq t$) & $\geq 7$ \\

\hline

40 & $A_1$ & 1 & $\geq 23$  \\

\hline

41 & $A_1$ & 1 & $\geq 29$ \\

\hline

42 & $A_1$ & 1 & $\geq  31$ \\

\hline

\end{longtable}

The composition factors of $L(E_8)$ restricted to each irreducible subgroup $A_1$ are given in Table \ref{E8tabcomps}.

\begin{proof}

We consider each reductive, maximal connected subgroup $M$ of $E_8$ in turn. By Theorem \ref{maximalexcep}, they are $D_8$, $A_8$, $\bar{A}_1 E_7$, $A_2 E_6$, $A_4^2$, $G_2 F_4$, $B_2$ $(p \geq 5)$, $A_1 A_2$ $(p \geq 5)$, $A_1$ $(p \geq 23)$, $A_1$ $(p \geq 29)$ and $A_1$ $(p \geq 31)$. Let $X$ be an $M$-irreducible subgroup $A_1$. 

Firstly, suppose $M = D_8$. We start by finding the $E_8$-conjugacy classes of $D_8$-irreducible subgroups of $D_8$; we claim that these are $E_8(\#1)$--$E_8(\#30)$ in Table \ref{E8tab} as well as the class of $A_1$ subgroups acting as $1 \otimes 1^{[r]} \otimes 1^{[s]} \otimes 1^{[t]}$ that are excluded from $E_8(\#9)$ when $p=2$. This is entirely similar to the case $A_1 D_6 < E_7$ and is a mainly routine task of using Lemma \ref{class} to find all of the $D_8$-conjugacy classes of $D_8$-irreducible subgroups and then considering which classes are fused in $E_8$. We will just give some details on the $D_8$-classes which are fused in $E_8$. 

First, we note that the excluded class of $A_1$ subgroups acting as $1 \otimes 1^{[r]} \otimes 1^{[s]} \otimes 1^{[t]}$  on $V_{D_8}(\lambda_1)$ when $p=2$ are contained in $B_4 (\ddagger)$, with notation from \cite[Lemma 7.1]{tho1}. By \cite[Lemma 7.4]{tho1}, this subgroup $B_4$ is contained in a parabolic subgroup of $E_8$ and hence so is the class of $A_1$ subgroups. 

The $D_8$-classes of $A_1$ subgroups which are fused in $E_8$ are all contained in the maximal rank subsystem subgroups $A_1^2 D_6$ or $D_4^2$. By \cite[Table 11]{car}, we have $N_{E_8}(A_1^2 D_6) = (A_1^2 D_6).2$ where the involution simultaneously acts a graph automorphism of $D_6$ and swaps the two $A_1$ factors. Consider a subgroup $A_1$ acting on $V_{D_8}(\lambda_1)$ as $5^{[r]} \otimes 1^{[s]} + 1^{[t]} \otimes 1^{[u]}$ $(r \neq s; rstu=0)$ when $p \geq 7$. There are two $D_8$-classes of such $D_8$-irreducible $A_1$ subgroups, since there are two $D_6$-classes of $A_1$ subgroups acting as $5^{[r]} \otimes 1^{[s]}$ $(r \neq s)$ on $V_{D_6}(\lambda_1)$. These classes are fused in $E_8$ by an involution in $N_{E_8}(A_1^2 D_6)$, yielding $E_8(\#5)$. Similarly, consider a subgroup $A_1$ acting on $V_{D_8}(\lambda_1)$ as $2^{[r]} \otimes 1^{[s]} \otimes 1^{[t]} + 1^{[u]} \otimes 1^{[v]}$ ($r,s,t$ distinct; $rstuv=0$) when $ p \neq 2$. There is just one $D_8$-class of such $D_8$-irreducible $A_1$ subgroups, since the graph automorphism of $D_6$ just swaps $s$ and $t$. In $E_8$ we may swap $u$ and $v$ or if $u=v$ then we may swap $s$ and $t$. Therefore, to have a complete set of representatives without repeats we need $u < v$ or $u = v$ and $s < t$, as in $E_8(\#6)$. Similar arguments apply to yield $E_8(\#13)$, $E_8(\#15)$, $E_8(\#17)$ and $E_8(\#25)$. 

We now consider subgroups of $D_4^2$. We have $N_{E_8}(D_4^2) = (D_4^2).(2 \times S_3)$ by \cite[Table 11]{car}, where the $S_3$ acts simultaneously on both $D_4$ factors and the involution commuting with $S_3$ swaps the $D_4$ factors. Firstly, let $Y$ be a subgroup $A_1$ of $D_4^2$ acting as $4^{[r]} \otimes 2^{[s]}$ or $3^{[r]} \otimes 1^{[s]}$ (2 classes) on the first $D_4$ factor and as $4^{[t]} \otimes 2^{[u]}$ or $3^{[t]} \otimes 1^{[u]}$ (2 classes) on the second $D_4$ factor. By using a triality automorphism, we may assume $Y$ acts as $4^{[r]} + 2^{[s]}$ on the first factor $D_4$. The projection of $Y$ to the first factor $D_4$ lies in $\text{SO}_5 \text{SO}_3$ and is hence contained in the centraliser of an involution in the $S_3$. Therefore, we may act by this involution on the second factor $D_4$ reducing the possibilities, up to $E_8$-conjugacy, to $4^{[t]} + 2^{[u]}$ or $3^{[t]} \otimes 1^{[u]}$ (1 class). Furthermore, if $r=s$ then the projection of $Y$ lies in $A_2$, which is the centraliser of a triality automorphism. We may therefore assume that the projection of $Y$ to the second factor $D_4$ acts as $4^{[t]} + 2^{[u]}$. This analysis leads to the classes $E_8(\#21)$ and $E_8(\#22)$ in Table \ref{E8tab}. We note that in $E_8(\#21)$ we may assume $r \leq s$ (and $t \leq u$) since the Weyl group of $D_8$ contains an involution swapping the stabilisers of the two $5$-spaces, $4^{[r]}$ and $4^{[s]}$ (the stabilisers of the two $3$-spaces, $2^{[t]}$ and $2^{[u]}$). Also, the involution swapping the two $D_4$ factors allows us to assume $r \leq t$ in $E_8(\#22)$.     

A similar analysis when $p=2$ and $Y$ acts as either $0|(2^{[r]} + 2^{[s]} + 2^{[t]})|0$ or $1^{[r]} \otimes 1^{[s]} \otimes 1^{[t]}$ on each of the $D_4$ factors, leads to two collections of conjugacy classes, namely an $E_8$-reducible one and $E_8(\#28)$.   

Next we consider subgroups contained in $A_1^4 D_4$. By \cite[Table 11]{car}, we see that $N_{E_8}(A_1^4 D_4) = (A_1^4 D_4).S_4$, where the $S_4$ acts naturally on the four $A_1$ factors and induces an action of $S_3$ on the $D_4$ (with the normal Klein four-subgroup acting trivially). Let $Y$ be a $D_8$-irreducible subgroup $A_1$ of $A_1^4 D_4$ that is not contained in $A_1^8$ (we will consider this in the next paragraph). Then the projection of $Y$ to $D_4$ acts as $6^{[r]} + 0$ $(p \geq 7)$,  $4^{[r]} + 2^{[s]}$ $(p \geq 3)$, $3^{[r]} \otimes 1^{[s]}$ $(p \geq 3; r \neq s)$, $0 | (2 + 2^{[r]} + 2^{[s]}) | 0$ $(p=2; 0 < r < s)$ or $1 \otimes 1^{[r]} \otimes 1^{[s]}$ $(p=2; 0 < r < s)$ on $V_{D_4}(\lambda_1)$.  In the first case the projection of $Y$ is contained in $G_2$ and hence centralised by the action of $S_3$. In this case the action of $S_4$ on $A_1^4$ allows us to assume $s < t < u < v$, yielding $E_8(\#18)$. In the second and third cases, we may use the triality automorphism to assume $Y$ acts as $4^{[r]} + 2^{[s]}$. If $r=s$ then the projection of $Y$ is centralised by the action of $S_3$; whereas when $r \neq s$ the projection of $Y$ is only centralised by an involution in $S_3$. This yields the constraints on the field twists in $E_8(\#23)$. Similarly, the fourth and fifth cases yield $E_8(\#27)$. 

Finally, we consider the classes of irreducible $A_1$ subgroups contained in $A_1^8$. By \cite[Table 11]{car}, we have $N_{E_8}(A_1^8) = (A_1^8).\text{AGL}(3,2)$, where $\text{AGL}(3,2) < S_8$ acts on the eight $A_1$ factors. The subgroup $\text{AGL}(3,2) < S_8$ is generated by $(2,4) (6,8)$, $(2,5,3)(4,6,7)$ and $(1,2)(3,4)(5,6)(7,8)$. To find the $E_8$-classes of the $D_8$-irreducible $A_1$ subgroups, we follow the same method as for $A_1^7 < E_7$, systematically formulating constraints on the field twists, ensuring that each ordered set $0,r,s,t,u,v,w,x$ gives a $D_8$-irreducible and there are no repeated classes. We note that $\text{AGL}(3,2)$ is $3$-transitive, with two orbits on $4$-sets, with representatives (in terms of the eight field twists) given by $0,r,s,t$ and $0,r,s,u$. Moreover the stabiliser of a singleton is isomorphic to $\text{PSL}(2,7)$, the stabiliser of a pair is isomorphic to $\mathbb{Z}_2 \times S_4$, the stabiliser of a triple is isomorphic to $S_4$, as is the stabiliser of either class of quadruples. From this, it is straightforward to prove the ordered sets $0, r, \ldots, x$ satisfying the conditions of Table \ref{8conditions1} yield a complete set of $E_8$-conjugacy classes of $D_8$-irreducible $A_1$ subgroups contained in $A_1^8$, without repeat. This gives $E_8(\#26)$ in Table \ref{E8tab}. 

For the case $M = D_8$, it remains to prove that $E_8(\#1)$--$E_8(\#30)$ are $E_8$-irreducible. Firstly, by considering the composition factors from Table \ref{E8tabcomps}, Corollary \ref{notrivs} shows that $E_8(\#n)$ is $E_8$-irreducible for the following ID numbers $n$: $1$, $2$ ($p \geq 13$), $3$ ($p \geq 7$), $4$, $5$ ($p \geq 11$), $6$ ($p \geq 7$), $7$, $8$ ($p \neq 2$), $9$ ($p \geq 5$), $10$, $11$, $12$, $13$ ($p \geq 13$), $14$, $15$ $(p \geq 13)$, $16$ $(p \geq 11)$, $17$ $(p \geq 11)$,  $18$ $(p \geq 11)$, $19$, $20$, $21$, $22$ $(p \geq 7)$, $23$ $(p \geq 7)$, $24$ $(p \geq 7)$, $25$ $(p \geq 5)$ and $26$ $(p \geq 5)$. We now prove the remaining subgroups are $E_8$-irreducible using Lemma \ref{wrongcomps}.  

First, let $X = E_8(\#2)$ when $p=11$. From Table \ref{E8tabcomps}, we have $L(E_8) \downarrow X$ has a trivial composition factor only when $r = 0, s=1$ and Corollary \ref{notrivs} applies otherwise. In the case $r=0, s=1$, we have $L(E_8) \downarrow X = 40 /$ $\!\! 34 /$ $\!\! 30^2 /$ $\!\! 26^2 /$ $\!\! 22^2 /$ $\!\! 20^2 /$ $\!\! 16^2 /$ $\!\! 14^2 /$ $\!\! 12 /$ $\!\! 10 /$ $\!\! 6^2 /$ $\!\! 4 /$ $\!\! 2 /$ $\!\! 0 $. Assume there exists an $L$-irreducible subgroup $Y$ of a Levi factor $L$ with the same composition factors as $X$ on $L(E_8)$. Using Table \ref{levie8}, we find that the possibilities for $L'$ are $A_1 E_6$, $D_7$, $A_2 D_5$, $A_7$, $A_3 A_4$, $A_1 A_6$ and $A_1 A_2 A_4$. We rule out $A_1 E_6$ since it has a $2$-dimensional composition factor on $L(E_8)$. Suppose $Y$ is contained in $D_7$. Using Table \ref{levie8}, we see that $V_{D_7}(\lambda_1)$ occurs as a multiplicity two composition factor of $L(E_8) \downarrow D_7$. Therefore, $Y$ does not have the same composition factors as $X$ on $L(E_8)$, since there are no combination of composition factors of $L(E_8) \downarrow X$ that form two isomorphic $14$-dimensional modules. Now suppose $Y$ is contained in $A_2 D_5$. It follows from the composition factors of $X$, and by assumption $Y$, on $L(E_8)$ that $(V_{A_2}(00),V_{D_5}(\lambda_1)) \downarrow Y = 6 + 2^{[1]}$ and thus $(V_{A_2}(00),V_{D_5}(\lambda_4)) \downarrow Y = 6 \otimes 1^{[1]} + 1^{[1]}$. Since $(V_{A_2}(00),V_{D_5}(\lambda_4))$ occurs as a composition factor of $A_2 D_5$ on $L(E_8)$, we see that $Y$ has a $2$-dimensional composition factor on $L(E_8)$, a contradiction. Similarly, if $Y$ is contained in $A_7$ then $V_{A_7}(\lambda_1) \downarrow Y = 14 = 3 \otimes 1^{[1]}$. Therefore, $V_{A_7}(\lambda_3) \downarrow Y$, which occurs as a composition factor of $L(E_8) \downarrow A_7$, has a composition factor of high weight 36, a contradiction. Now suppose $Y$ is contained in $A_3 A_4$. From Table \ref{levie8}, we see that $(V_{A_3}(100),V_{A_4}(0000))$ and $(V_{A_3}(001),V_{A_4}(0000))$ both occur as composition factors of $L(E_8) \downarrow A_3 A_4$. But $L(E_8) \downarrow X$ has only one composition factor of dimension four and so the projection of $Y$ to $A_3$ is not $A_3$-irreducible, a contradiction. Now suppose $Y$ is contained in $A_1 A_6$. Then we find that $V_{A_6}(\lambda_1) \downarrow Y = 6$ and so both $V_{A_6}(\lambda_3) \downarrow Y$ and $V_{A_6}(\lambda_4) \downarrow Y$ have a composition factor of high weight 12. Therefore, $L(E_8) \downarrow Y$ has at least two composition factors of high weight 12, a contradiction. Finally, suppose $Y$ is contained in $A_1 A_2 A_4$. The largest dimension of a composition factor of $L(E_8) \downarrow A_1 A_2 A_4$ is 30 and hence $L(E_8) \downarrow Y$ does not have a composition factor of dimension 32, a contradiction. We have hence shown that no such subgroup $Y$ exists and so Lemma \ref{wrongcomps} shows that $X$ is $E_8$-irreducible.           

Similar arguments show that $E_8(\#3)$ $(p=5)$, $E_8(\#6)$ $(p=5)$, $E_8(\#13)$ $(p=11)$, $E_8(\#16)$ $(p=7)$, $E_8(\#18)$ $(p=7)$, $E_8(\#24)$ $(p=5)$ , $E_8(\#25)$ $(p=3)$ and $E_8(\#26)$ $(p=3)$ are $E_8$-irreducible, as they have at most one trivial composition factor on $L(E_8)$. 

We next consider the remaining cases when $p \neq 2$. First let $X = E_8(\#5)$ when $p=7$. Then $X$ has a trivial composition factor on $L(E_8)$ when $r=s-1=u$ and $X$ has two trivial composition factors when $r=s-1=t=u$. Using Lemma \ref{wrongcomps}, we will show that $X$ is $E_8$-irreducible when $r=s-1=t=u$. The case $r=s-1=u \neq t$ is similar and in all other cases Corollary \ref{notrivs} applies. Since $rstu=0$, we have $r=t=u=0$, $s=1$ and from Table \ref{E8tabcomps}, we see that $L(E_8) \downarrow X = 22^2 /$ $\!\! 20 /$ $\!\! 18^3 /$ $\!\! 16 /$ $\!\! 14^3 /$ $\!\! 12^5 /$ $\!\! 10^5 /$ $\!\! 8^2 /$ $\!\! 6 /$ $\!\! 4^2 /$ $\!\! 2^7 /$ $\!\! 0^2 $. Suppose there exists an $L$-irreducible subgroup $Y$ of a Levi factor $L$ with the same composition factors as $X$ on $L(E_8)$. Using Table \ref{levie8} and considering the number of trivial composition factors of $L(E_8) \downarrow X$, we find that the possibilities for $L'$ are $A_1 E_6$, $D_7$, $A_2 D_5$, $A_2 D_4$, $A_7$, $A_3 A_4$, $A_1 A_6$, $A_1 A_2 A_4$, $A_1^2 A_4$, $A_3^2$ and $A_1^2 A_2^2$. We rule out $L'$ being $A_1 E_6$, $A_1^2 A_4$ or $A_1^2 A_2^2$ since they have $2$-dimensional composition factors on $L(E_8)$. We also rule out $A_1 A_6$ since $A_6$ does not contain an $A_6$-irreducible subgroup $A_1$ when $p=7$, by Lemma \ref{class}. Now suppose $Y$ is contained in $D_7$. From Table \ref{levie8}, we see $V_{D_7}(\lambda_1)$ occurs as a multiplicity two composition factor of $L(E_8) \downarrow D_7$. As $L(E_8) \downarrow X$ has only one $7$-dimensional composition factor, two $5$-dimensional composition factors and two trivial composition factors, it follows that $V_{D_7}(\lambda_1) \downarrow Y = 2^{[a]} + 2^{[b]} + 1^{[c]} \otimes 1^{[d]} + 1^{[e]} \otimes 1^{[f]}$ with $c \neq d$ and $e \neq f$. Therefore, $L(E_8) \downarrow Y$ has at least four $4$-dimensional composition factors, a contradiction. Now suppose $Y$ is contained in $A_2 D_5$. Then \begin{align*} L(E_8) \downarrow A_2 D_5 = & (W(11),0) / (10,\lambda_1) / (10,\lambda_4) / (10,0) /   (01,\lambda_1) / (01,\lambda_5) /  (01,0) / (00,W(\lambda_2)) / \\ &  (00,\lambda_4) / (00,\lambda_5) / (00,0). \end{align*} The projection of $Y$ to $A_2$ acts as $2^{[a]}$ on $10$ and hence has composition factors $4^{[a]} / 2^{[a]}$ on $W(11)$. We also have $V_{D_5}(\lambda_4) = V_{D_5}(\lambda_5)^*$. Therefore, $(10,\lambda_1) \downarrow Y = (01,\lambda_1) \downarrow Y$, $(10,\lambda_4) \downarrow Y = (01,\lambda_5) \downarrow Y$ and $(00,\lambda_4) \downarrow Y = (00,\lambda_5) \downarrow Y$. It follows that $(00,\lambda_2) \downarrow Y$ contains at least one copy each $Y$-composition factor of $L(E_8)$ occurring with odd multiplicity (except for possibly a composition factor of high weight $2$). The sum of the dimensions of such composition factors is 75 which is greater than $45 = \text{dim} (V_{D_5}(\lambda_2))$, a contradiction. The previous argument does not use the $D_5$-irreducibility of $Y$ and hence also shows that $Y$ is not contained in $A_2 D_4$. 

Now suppose $Y$ is contained in $A_7$. Since $p=7$ and $Y$ is $A_7$-irreducible, $Y$ acts as $3^{[a]} \otimes 1^{[b]}$ $(a \neq b)$ or $1^{[a]} \otimes 1^{[b]} \otimes 1^{[c]}$ ($a,b,c$ distinct) on $V_{A_7}(\lambda_1)$. In the latter case, $Y$ is contained in $C_4$, which has three trivial composition factors on $L(E_8)$ (by \cite[Table 8.1]{LS3}), hence $Y$ has at least three trivial composition factors, a contradiction. So $Y$ acts as $3^{[a]} \otimes 1^{[b]}$ on $V_{A_7}(\lambda_1)$. From Table \ref{levie8}, we have $L(E_8) \downarrow A_7 = (\lambda_1 + \lambda_7) /$ $\!\!  \lambda_1 /$ $\!\!  \lambda_2 /$ $\!\!  \lambda_3 /$ $\!\!  \lambda_5 /$ $\!\!  \lambda_6 /$ $\!\!  \lambda_7 /$ $\!\! 0$. Since $\lambda_i = \lambda_{8-i}^*$ for $i=1, 2, 3$ it follows that $(\lambda_1 + \lambda_7) \downarrow Y$ has at least one copy of each odd multiplicity composition factor of $L(E_8) \downarrow Y$. The sum of the dimensions of such composition factors is at least 78, which is greater than $63$, a contradiction. Now suppose $Y$ is contained in $A_3 A_4$. Then $L(E_8) \downarrow A_3 A_4$ has one trivial composition factor. All of the other composition factors occur in pairs with their duals, except for $(V_{A_3}(101),V_{A_4}(0000))$ and $(V_{A_3}(000),V_{A_4}(1001))$. Since $Y$ has exactly two trivial composition factors, it follows that $V_{A_3}(101)$ restricted to the projection of $Y$ to $A_3$ or $V_{A_4}(1001)$ restricted to the projection of $Y$ to $A_4$ has exactly one trivial composition factor (and not both). However, the projection of $Y$ to $A_3$ and the projection to $A_4$ are irreducible and so act as $1 \otimes 1^{[a]}$ $(a \neq 0)$ or $3^{[a]}$ and $4^{[a]}$ on the natural module, respectively. Neither action on $V_{A_3}(100)$ yields a trivial composition factor on $V_{A_3}(101)$ and the action on $V_{A_4}(1000)$ does not yield a trivial composition factor on $V_{A_4}(1001)$ either. Hence $Y$ is not contained in $A_3 A_4$. A similar argument also rules out $A_1 A_2 A_4$. Finally, suppose $Y$ is contained in $A_3^2$. Then $(V_{A_3}(101),V_{A_3}(000)) \downarrow Y$ and $(V_{A_3}(101),V_{A_3}(000)) \downarrow Y$ have a least one copy of each odd multiplicity composition factor of $L(E_8) \downarrow Y$. As before, the sum of the dimensions of one copy of each odd multiplicity isomorphism class of composition factors is 78, which is greater than 30, a contradiction. We have shown that no such subgroup $Y$ exists and hence $X$ is $E_8$-irreducible by Lemma \ref{wrongcomps}.

Similar arguments show that $E_8(\#15)$ ($p=11$), $E_8(\#17)$ ($p=7$), $E_8(\#22)$ ($p=7$) and $E_8(\#24)$ ($p=3$) are $E_8$-irreducible, as they have at most two trivial composition factors on $L(E_8)$. The remaining cases when $p \neq 2$ are $E_8(\#6)$ $(p=3)$, $E_8(\#9)$ $(p=3)$, $E_8(\#22)$ $(p=5)$ and $E_8(\#23)$ $(p=3)$. They all have at most four trivial composition factors on $L(E_8)$ (in fact, $E_8(\#12)$ has at most three). We will consider one of the cases in which $E_8(\#6)$ has four trivial composition factors and prove it is $E_8$-irreducible. The other cases are all similar. 

Let $X = E_8(\#6)$ and $p=3$. When $s=u=v=r-1=t-2$, we see from Table \ref{E8tabcomps} that $X$ has four trivial composition factors on $L(E_8)$. Since $rst=0$, we have $s=u=v=0$, $r=1$, $t=2$ and $L(E_8) \downarrow X = 30 /$ $\!\!  28 /$ $\!\!  26^2 /$ $\!\!  24 /$ $\!\!  22 /$ $\!\!  18^4 /$ $\!\! 16^5 /$ $\!\!  14^2 /$ $\!\!  12^7 /$ $\!\!  10^2 /$ $\!\!  6 /$ $\!\!  4^2 /$ $\!\!  2^6 /$ $\!\! 0^4$. As usual, we suppose there exists an $L$-irreducible subgroup $Y$ of a Levi factor $L$ with the same composition factors as $X$ on $L(E_8)$. Since $p=3$ there are a few Levi subgroups $L$, that although they have four or fewer trivial composition factors on $L(E_8)$, do not have an $L'$-irreducible subgroup $A_1$. Using Table \ref{levie8}, we find that the possibilities for $L'$ are $E_7$, $A_1 E_6$, $D_7$, $A_2 D_5$, $A_1 D_5$, $A_2 D_4$, $A_7$, $A_1 A_5$,  $A_3^2$ and $A_1^2 A_2^2$. We immediately rule out $L'$ being $A_1 E_6$, $A_1 D_5$ or $A_1 A_5$ since they have $2$-dimensional composition factors on $L(E_8)$. We also rule out $L'$ being $A_2 D_4$, $A_3^2$ or $A_1^2 A_2^2$ since they do not have at least two composition factors of dimension at least 27. 

Now suppose that $Y$ is contained in $A_7$. Then $Y$ acts as $1^{[a]} \otimes 1^{[b]} \otimes 1^{[c]}$ ($a,b,c$ distinct). Therefore, $V_{A_7}(\lambda_1 + \lambda_7) \downarrow Y = 2^{[a]} \otimes 2^{[b]} \otimes 2^{[c]} /$ $\!\!  2^{[a]} \otimes 2^{[b]} /$ $\!\!  2^{[a]} \otimes 2^{[c]} /$ $\!\!  2^{[b]} \otimes 2^{[c]} /$ $\!\!  2^{[a]} /$ $\!\!  2^{[b]} /$ $\!\!  2^{[c]} /$ $\!\!  0$. In particular, $L(E_8) \downarrow Y$ has at least three 9-dimensional composition factors, a contradiction. 

Now suppose $Y$ is contained in $E_7$. Then $Y$ is conjugate to $E_7(\#3)$, $E_7(\#10)$, $E_7(\#11)$ or $E_7(\#12)$ by Theorem \ref{thmE7}. From Table \ref{levie8}, we have $L(E_8) \downarrow E_7 = \lambda_1 /$ $\!\! \lambda_7^2 /$ $\!\! 0^3$. Since $L(E_8) \downarrow Y$ has exactly two 27-dimensional composition factors, which are isomorphic to each other (both are $26 = 2 \otimes 2^{[1]} \otimes 2^{[2]}$) it follows that $V_{E_7}(\lambda_7) \downarrow Y$ has exactly one 27-dimensional composition factor or $V_{E_7}(\lambda_1) \downarrow Y$ has two isomorphic 27-dimensional composition factors. Using Table \ref{E7tabcomps}, we see that this is not true for $E_7(\#10)$, $E_7(\#11)$ and $E_7(\#12)$. Therefore $Y$ is conjugate to $E_7(\#3)$. But then $L(E_8) \downarrow Y$ has a 2-dimensional composition factor coming from $V_{E_7}(\lambda_1) \downarrow Y$, a contradiction. 

Suppose $Y$ is contained in $D_7$. Then since $Y$ is $D_7$-irreducible and $p=3$, it acts on $V_{D_7}(\lambda_1)$ as $2^{[a]} + 2^{[b]} +2^{[c]} + 1^{[d]} \otimes 1^{[e]} + 0$ ($a,b,c$ distinct; $d \neq e$), $2^{[a]} + 2^{[b]} + 1^{[c]} \otimes 1^{[d]} + 1^{[e]} \otimes 1^{[f]}$ ($a \neq b$; $c \neq d$; $e \neq f$; $\{c,d\} \neq \{e,f\}$) or $2^{[a]} \otimes 2^{[b]} + 1^{[c]} \otimes 1^{[d]} + 0$ ($a \neq b$; $c \neq d$). The first action is impossible, since $V_{D_7}(\lambda_1)$ occurs as a multiplicity two composition factor in $L(E_8) \downarrow D_7$ and $L(E_8) \downarrow Y$ only has two non-isomorphic $3$-dimensional composition factors. Similarly, the latter action is also impossible, since $L(E_8) \downarrow Y$ only has one $9$-dimensional composition factor. So $Y$ acts as $2^{[a]} + 2^{[b]} + 1^{[c]} \otimes 1^{[d]} + 1^{[e]} \otimes 1^{[f]}$ ($a \neq b$, $c \neq d$, $e \neq f$ and $\{c,d\} \neq \{e,f\}$) on $V_{D_7}(\lambda_1)$. It follows that \begin{align*} V_{D_7}(\lambda_2) \downarrow Y = & 2^{[r]} \otimes 2^{[s]} / 2^{[r]} \otimes 1^{[t]} \otimes 1^{[u]} / 2^{[r]} \otimes 1^{[v]} \otimes 1^{[w]} / 2^{[r]} / 2^{[s]} \otimes 1^{[t]} \otimes 1^{[u]} / \\ & 2^{[s]} \otimes 1^{[v]} \otimes 1^{[w]} / 2^{[s]} / 2^{[t]} / 1^{[t]} \otimes 1^{[u]} \otimes 1^{[v]} \otimes 1^{[w]} / 2^{[u]} / 2^{[v]} / 2^{[w]} \end{align*} and \begin{align*} V_{D_7}(\lambda_6) \downarrow Y = V_{D_7}(\lambda_7) \downarrow Y = & 1^{[r]} \otimes 1^{[s]} \otimes 1^{[t]} \otimes 1^{[v]} / 1^{[r]} \otimes 1^{[s]} \otimes 1^{[t]} \otimes 1^{[w]} / \\ & 1^{[r]} \otimes 1^{[s]} \otimes 1^{[u]} \otimes 1^{[v]} / 1^{[r]} \otimes 1^{[s]} \otimes 1^{[u]} \otimes 1^{[w]}. \end{align*} Hence $L(E_8) \downarrow Y$ has no $27$-dimensional composition factors, a contradiction. Therefore $Y$ does not exist and $X$ is $E_8$-irreducible by Lemma \ref{wrongcomps}. 

The final step for $M=D_8$ is to consider the case $p=2$, where $X$ is one of $E_8(\#n)$ where $n=8, 26, 27, 28, 29$ or $30$. As with the previous cases, we use Lemma \ref{wrongcomps} to prove $X$ is $E_8$-irreducible. Lemma \ref{bada1p2} contains all $L$-irreducible $A_1$ subgroups of Levi factors $L$ when $p=2$. For such a subgroup $Y$, one can write down the composition factors of $L(E_8) \downarrow Y$ and compare them to those of $L(E_8) \downarrow X$ from Table \ref{E8tabcomps}. In all cases it is straightforward to show they never match and we leave the details to the reader. This completes the case $M = D_8$.       

We next consider the case $M = A_8$. Using Lemma \ref{class}, we see that $X$ acts on $V_{A_8}(\lambda_1)$ as $8$ $(p \geq 11)$ or $2 \otimes 2^{[r]}$ $(p \geq 3; r \neq 0)$. In both cases $X$ preserves an orthogonal form on $V_{A_8}(\lambda_1)$ and is hence contained in $B_4$. By \cite[Table~8.1]{LS3}, we have that this subgroup $B_4$ is also contained in $D_8$ acting irreducibly on the natural module. So in the first case, $X$ is contained in $D_8$ and is conjugate to $E_8(\#9)$. In the second case, $X$ acts as $W(3) \otimes 1^{[r]} + 1^{[r]} \otimes W(3)$ on $V_{D_8}(\lambda_1)$ (by \cite[Prop 2.13]{LS3}). If $p =3$ then $X$ is $D_8$-reducible by Lemma \ref{class} and hence $E_8$-reducible. If $p > 3$, then $X$ is conjugate to $E_8(\#21^{\{0,r,0,r\}})$ and hence $E_8$-irreducible. 

Now let $M = \bar{A}_1 E_7$. The projection of $X$ to $E_7$ is $E_7$-irreducible and so by Theorem \ref{thmE7}, it is $E_7$-conjugate to a subgroup in Table \ref{E7tab}. Let $Y$ be the projection of $X$ to $E_7$ so $X$ is a diagonal subgroup of $\bar{A}_1 Y$. We now analyse the different possibilities for $Y$ from Theorem \ref{thmE7}. Suppose $Y$ is contained in $\bar{A}_1 D_6$. Then $X$ is contained in $\bar{A}_1^2 D_6$ which is a subgroup of $D_8$, and has hence already been considered. 

Next, suppose $Y$ is $E_7(\#15)$ or $E_7(\#16)$ and so $Y$ is contained in $G_2 C_3$. Consider the first case, $Y = E_7(\#15)$. Then $X \hookrightarrow \bar{A}_1 A_1 A_1$ via $(1^{[r]},1^{[s]},1^{[t]})$ $(rst=0; s \neq t)$, where the second factor $A_1$ is maximal in $G_2$ and the third factor $A_1$ is maximal in $C_3$. Corollary \ref{notrivs} shows that $X$ is $E_8$-irreducible, yielding $E_8(\#31)$. Now let $Y = E_7(\#16)$ and so $X \hookrightarrow \bar{A}_1 A_1 A_1 A_1 < \bar{A}_1 G_2 C_3$ via $(1^{[r]},1^{[s]},1^{[t]},1^{[u]})$ $(rstu=0; s \neq t; t \neq u)$, where the second factor is maximal in $G_2$ and $A_1 A_1 < C_3$ acts as $(2,1)$ on $V_{C_3}(100)$. If $p > 7$ then Corollary \ref{notrivs} shows that $X$ is $E_8$-irreducible. When $p=7$, Corollary \ref{notrivs} applies except when $r=t=1$, $s=u=0$. Using the restriction of $L(E_8)$ to $E_8(\#32)$ in Table \ref{E8tabcomps}, we calculate that $L(E_8) \downarrow X = 58 /$ $\!\! 44 /$ $\!\! 36 /$ $\!\! 34 /$ $\!\! 30^3 /$ $\!\! 28 /$ $\!\! 26^2 /$ $\!\! 22 /$ $\!\! 14^4 /$ $\!\! 12^2 /$ $\!\! 10^2 /$ $\!\! 2^3 /$ $\!\! 0$.  Suppose $Z$ is an $L$-irreducible subgroup of a Levi factor $L$ having the same composition factors as $X$ on $L(E_8)$. Since $L(E_8) \downarrow X$ has only one trivial composition factor, the possibilities for $L'$ are $D_7$, $A_7$, $A_1 E_6$, $A_1 A_6$, $A_2 D_5$, $A_1 E_6$ and $A_3 A_4$. Suppose $L' = D_7$.  From Table \ref{levie8}, we see that $V_{D_7}(\lambda_1) \downarrow Z$ occurs as a multiplicity two composition factor of $L(E_8) \downarrow Z$. This is a contradiction, because $L(E_8) \downarrow X$ does not have a set of composition factors that form two isomorphic 14-dimensional modules. Now suppose $L' = A_7$. By considering the even multiplicity $8$-dimensional composition factors of $X$, it follows that $Z$ acts as $10 = 3 \otimes 1^{[1]}$ on $V_{A_7}(\lambda_1)$. Since $V_{A_7}(\lambda_2)$ occurs a composition factor of $L(E_8) \downarrow A_7$, it follows that $Z$ has a composition factor of high weight $18$ on $L(E_8)$. This is a contradiction. Suppose $L' = A_1 E_6$. Then $L(E_8) \downarrow A_1 E_6$ has a $2$-dimensional composition factor and therefore $Z$ does, a contradiction. Now suppose $L' = A_1 A_6$. Then the projection of $Z$ to $A_6$ acts as $6$ on $V_{A_6}(\lambda_1)$, by Lemma \ref{class}. From Table \ref{levie8}, we have $L(E_8) \downarrow A_1 A_6$ has a composition factor $(V_{A_1}(0),V_{A_6}(\lambda_1))$ and therefore $L(E_8) \downarrow Z$ has a composition factor of dimension 7, a contradiction.  Now let $L' = A_2 D_5$. The only composition factors of $L(E_8) \downarrow A_2 D_5$ with dimension at least 35 are $(V_{A_2}(00),V_{D_5}(\lambda_2))$, $(V_{A_2}(10),V_{D_5}(\lambda_4))$ and $(V_{A_2}(01),V_{D_5}(\lambda_5))$. The composition factor of high weight $34$ of $L(E_8) \downarrow X$ has dimension 35 and so one of the three composition factors of dimension at least 35 has $34$ as a composition factor when restricted to $Z$. Since $\lambda_4^* = \lambda_5$ and $10^* = 01$, we have $(V_{A_2}(10),V_{D_5}(\lambda_4)) \downarrow Z = (V_{A_2}(01),V_{D_5}(\lambda_5)) \downarrow Z$ but $34$ occurs with multiplicity one. Therefore, $(V_{A_2}(00),V_{D_5}(\lambda_2)) \downarrow Z = 34 / M_1 / \ldots / M_k$. The sum of the dimensions of $M_1, \ldots, M_k$ is 10 but no set of composition factors of $L(E_8) \downarrow Z$ have dimensions that sum to 10. This is a contradiction, ruling out $A_2 D_5$. Finally, suppose $L' = A_3 A_4$. Then $L(E_8) \downarrow A_3 A_4$ has $(V_{A_3}(000),V_{A_4}(0100))$ as a composition factor. This has dimension 10, but we just noted that $L(E_8) \downarrow Z$ has no set of composition factors whose dimensions sum to 10. This final contradiction shows that $Z$ does not exist and therefore Lemma \ref{wrongcomps} shows $X$ is $E_8$-irreducible when $p=7$. This yields $E_8(\#32)$ in Table \ref{E8tab}.

Next we consider the case where $Y$ (the projection of $X$ to $E_7$) is $E_7(\#17)$ and so contained in $A_1 G_2$. Then $X \hookrightarrow \bar{A}_1 A_1 A_1 < \bar{A}_1 A_1 G_2$ via $(1^{[r]},1^{[s]},1^{[t]})$ $(p \geq 7; rst=0; s \neq t)$ where the third factor $A_1$ is maximal in $G_2$. We see that $X$ is $E_8$-irreducible by Corollary \ref{notrivs}, yielding $E_8(\#33)$. Similarly, when $Y$ is $E_7(\#18)$ we have $X \hookrightarrow \bar{A}_1 A_1 A_1 < \bar{A}_1 A_1 F_4$ via $(1^{[r]},1^{[s]},1^{[t]})$  $(p \geq 13; rst=0)$ where the third factor $A_1$ is maximal in $F_4$. Again, $X$ is $E_8$-irreducible by Corollary \ref{notrivs}, giving $E_8(\#34)$ in Table \ref{E8tab}. 

Suppose $Y$ is $E_7(\#19)$ and so $X \hookrightarrow \bar{A}_1 A_1 A_1$ via $(1^{[r]},1^{[s]},1^{[t]})$ $(p \geq 5; st=0; s \neq t)$ where $A_1 A_1 < E_7$ is maximal. When $p > 5$, Corollary \ref{notrivs} applies. When $p = 5$, there is one trivial composition factor on $L(E_8) \downarrow X$ when $(r,s,t) = (0,1,0)$ or $(0,0,1)$ and none otherwise. We can use Lemma \ref{wrongcomps} in exactly the same way as for $E_8(\#32)$ $(p=7)$ to show $X$ is $E_8$-irreducible. This gives $E_8(\#35)$ in Table \ref{E8tab}. 

Finally, suppose $Y$ is conjugate to $E_7(\#20)$ or $E_7(\#21)$. Then $X$ is $E_8$-irreducible by Corollary \ref{notrivs}, yielding $E_8(\#36)$ and $E_8(\#37)$, respectively. This concludes the $M=\bar{A}_1 E_7$ case. 

Now let $M = A_2 E_6$. The projection of $X$ to $A_2$ acts as $2$ on $V_{A_2}(10)$ and $p \neq 2$. Let $Y$ be the projection of $X$ to $E_6$. By Theorem \ref{thmE6}, we see that $Y$ is contained in $\bar{A}_1 A_5$, $A_2^3$, $A_2 G_2$, $F_4$ or $C_4$. We claim that in all of the cases $X$ is contained in either $D_8$ or $\bar{A}_1 E_7$ and has therefore already been considered. If $Y$ is contained in $\bar{A}_1 A_5$ then $X$ is contained in $\bar{A}_1 A_2 A_5$, which is a subgroup of $\bar{A}_1 E_7$. If $Y$ is contained in $A_2^3$ then it is also contained in $C_4$ by the proof of Theorem \ref{thmE6}. So when $Y$ is either $E_6(\#3)$ or $E_6(\#6)$, $X$ is contained in $A_1 C_4$. The irreducible subgroup $A_1$ of $A_2$ is the centraliser of a graph automorphism of $A_2$ and similarly, $C_4$ is the centraliser in $E_6$ of a graph automorphism of $E_6$. By \cite[Table 11]{car}, we have $N_{E_8}(A_2 E_6) = (A_2 E_6).2$ where the involution acts as a graph automorphism on both the $A_2$ and the $E_6$ factors. Therefore, there exists an involution $t$ in $E_8$ such that $A_1 C_4 < C_{E_8}(t)^\circ$. By \cite[Table 4.3.1]{gls3}, we have $C_{E_8}(t)^\circ$ is either $D_8$ or $\bar{A}_1 E_7$ and hence $X$ is contained in $D_8$ or $\bar{A}_1 E_7$. In fact, we have $C_{E_8}(t)^\circ = D_8$.  

Next, suppose that $Y$ is contained in $A_2 G_2$. Then the factor $G_2$ is generated by root subgroups of $E_8$ and hence $X$ is contained in $G_2 C_{E_8}(G_2)^\circ = G_2 F_4$. In particular, the projection of $X$ to $F_4$ is contained in $A_2 A_2$. The proof of Theorem \ref{thmF4} shows that the $A_2 A_2$-irreducible $A_1$ subgroups of $A_2 A_2$ are also contained in $\bar{A}_1 C_3 < F_4$. Therefore, $X$ is contained in $\bar{A}_1 G_2 C_3 < \bar{A}_1 E_7$, as required. Finally, suppose $Y$ is contained in $F_4$. Then $X < F_4 C_{E_8}(F_4)^\circ = F_4 G_2$. Moreover, the projection of $X$ to $G_2$ is contained in the maximal subgroup $A_2$. By Theorem \ref{thmG2}, this is also contained in $\bar{A}_1 A_1 < G_2$. Therefore $X < \bar{A}_1 A_1 F_4$ and is hence a subgroup of $\bar{A}_1 E_7$.  

Now let $M = A_4^2$. Using Lemma \ref{class}, the only $A_4$-irreducible $A_1$ subgroups act as $4$ on $V_{A_4}(\lambda_1)$ ($p \geq 5)$ and are hence contained in $B_2 < A_4$. Therefore $X$ is contained in $B_2^2 < A_4^2$. By \cite[p. 63]{LS3}, we have that $B_2^2$ is also contained in $D_8$. Hence $X$ has already been considered in the $D_8$ case and is in fact conjugate to $E_8(\#3)$. 

Now let $M = G_2 F_4$. By Theorem \ref{thmG2}, the projection of $X$ to $G_2$ is either contained in $\bar{A}_1 A_1$ or is maximal with $p \geq 7$. In the first case, $X$ is contained in $\bar{A}_1 E_7$ and has already been considered. So suppose the projection of $X$ to $G_2$ is maximal. Now consider the projection of $X$ to $F_4$. By Theorem \ref{thmF4}, this is contained in $B_4$, $\bar{A}_1 C_3$, $A_1 G_2$ or is maximal with $p \geq 13$. In the first case $X$ is contained in $D_8$, since $B_4 C_{E_8}(B_4)^\circ = B_4 B_3 < D_8$ and in the second case $X$ is contained in $\bar{A}_1 E_7$. Now suppose the projection of $X$ to $F_4$ is $F_4(\#11)$ and hence contained in $A_1 G_2$. The factor $G_2$ of $M$ and the factor $G_2$ of $A_1 G_2$ are conjugate in $E_8$. Furthermore, $N_{E_8}(A_1 G_2^2)$ contains an involution swapping the $G_2$ factors. Thus, up to $E_8$-conjugacy, $X \hookrightarrow A_1 A_1^2 < A_1 G_2^2$ via $(1^{[r]},1^{[s]},1^{[t]})$ ($rs = 0$; $r \neq s$; $r \neq t$; $s \leq t$). We claim that if $s = t$ then $X$ is contained in $\bar{A}_1 E_7$. Indeed, $X$ is contained in the centraliser of an involution in $N_{E_8}(A_1 G_2^2)$ when $s=t$ and the connected centraliser of that involution is $\bar{A}_1 E_7$. In fact, $X$ is conjugate to $E_8(\#32^{\{r,s,r\}})$. When $s \neq t$, then $X$ is $E_8$-irreducible by Corollary \ref{notrivs}, yielding $E_8(\#39)$. 

The last case to consider is when the projection of $X$ to $F_4$ is maximal and hence conjugate to $F_4(\#10)$ $(p \geq 13)$. Then $X \hookrightarrow A_1 A_1 < G_2 F_4$ via $(1^{[r]},1^{[s]})$ $(rs=0)$ with the first factor $A_1$ maximal in $G_2$ and the second maximal in $F_4$. If $r \neq s$, then $X$ is $E_8$-irreducible by Corollary \ref{notrivs}, yielding $E_8(\#38)$. If $r = s$ then Theorem \ref{A1samecomp} shows that $X$ is conjugate to $E_8(\#11^{\{0,0\}})$ and has already been considered.   

Let $M = B_2$ $(p \geq 5)$. There are two cases to consider. Either $X$ is contained in $A_1^2$ and acts as $1 \otimes 1^{[r]} + 0$ $(r \neq 0)$ on $V_{B_2}(10)$ or $X$ is maximal in $M$ and acts as $4$ on $V_{B_2}(10)$. In the first case, $X$ is contained in the connected centraliser of an involution in $B_2$. Hence $X$ is contained in the connected centraliser of an involution in $E_8$, which is either $D_8$ or $\bar{A}_1 E_7$ and so $X$ has been considered already. 

Now consider the second case, in which $X$ is a maximal subgroup $A_1$ of $M$. If $p \geq 11$ then Theorem \ref{A1samecomp} shows that $X$ is conjugate to $E_8(\#10)$. When $p=7$, we have $X$ is contained in a parabolic subgroup of $E_8$, as proved in \cite[3.3]{LS1}. When $p=5$, we will show that $X$ is contained in an $A_7$-parabolic subgroup of $E_8$ and is hence $E_8$-reducible. To do this, we will use the same method as \cite[Lemma 7.9]{tho1}; we show that $S = A_1(25) < X$ fixes the same subspaces as $X$ on $L(E_8)$ and then show that $X$ fixes an $8$-dimensional abelian subalgebra that is ad-nilpotent of exponent 3 i.e. $(\text{ad } a)^3 = 0$ for all $a$. 

Firstly, Lemma \ref{wrongcomps} along with Table \ref{levie8} shows that the only parabolic subgroup $X$ can be contained in is an $A_7$-parabolic. 

To show $S$ and $X$ fix the same subspaces of $L(E_8)$ we use Lemma \ref{subspaces}. We have $L(E_8) \downarrow X = 18^2 /$ $\!\!  16 /$ $\!\!  14^3 /$ $\!\!  12^4 /$ $\!\!  10^5 /$ $\!\!  8^6 /$ $\!\!  6^8 /$ $\!\!  4 /$ $\!\!  2^3 /$ $\!\!  0^3$ and therefore conditions (i) and (iii) hold. To show condition (ii) holds it suffices to check that the Weyl modules of high weight $18$, $16$, $12$, $10$, $8$ and $6$ are still indecomposable when restricted to $S$. We then check this in Magma \cite{magma}: We construct $S \cong \text{PSL}(2,25) \cong \text{PSL}(V)$ and the $S$-modules $S^n(V)$. In each case, we use the inbuilt ``Socle'' function in Magma to find that the socle of $S^n(V)$ is irreducible and thus $S^n(V)$ is indecomposable for each integer $n$ in the list of high weights. Therefore, $X$ and $S$ fix the same subspaces of $L(E_8)$. 

The existence of $M = B_2$ when $p=5$ is proved in \cite[Lemma 5.1.6]{LS1} using \cite[6.7]{se2}. In particular, if $\alpha$ is the long simple root and $\beta$ is the short simple root in a basis for $M$ then the $A_1$ generated by $x_{\pm \alpha}(t)$ is contained in the subsystem subgroup $\bar{A}_1 A_5$ and the $A_1$ generated by $x_{\pm \beta}(t)$ is contained in the subsystem subgroup $A_2 D_5$. Using this, we can write down the generators $x_{\pm \alpha}(t)$, $x_{\pm \beta}(t)$ of $M$ in terms of generators of $E_8$. From these generators we construct $B_2(25)$ in Magma as a subgroup of the inbuilt finite group of Lie type $E_8(25)$ and then construct $S$ as a maximal subgroup of $B_2(25)$. We now use the inbuilt functionality of Magma to construct the Lie algebra $L(E_8)$ as a module for $S$ and then, again using the inbuilt ``Submodules'' function, we find all $8$-dimensional $S$-submodules of $L(E_8)$. We find that there is a unique such $S$-submodule that is an abelian subalgebra, and it is ad-nilpotent of exponent 3.     

So $S$ and therefore $X$ fixes an $8$-dimensional abelian subalgebra of $L(E_8)$ that is ad-nilpotent of exponent 3. Exponentiating this subalgebra yields an $8$-dimensional unipotent subgroup of $E_8$, normalised by $X$. Therefore $X$ is contained in a parabolic subgroup of $E_8$, as required. 

Now let $M = A_1 A_2$ $(p \geq 5)$. Then the projection of $X$ to $A_2$ acts as $2$ on $V_{A_2}(10)$ and is the centraliser of a graph automorphism of $A_2$. By \cite[Table 11]{car}, we have $N_{E_7}(A_1 A_2) = (A_1 A_2).2$ and therefore $X$ is contained in the centraliser of an involution in $E_8$. Using \cite[Table 4.3.1]{gls3} we find this centraliser is either $D_8$ or $\bar{A}_1 E_7$. One calculates that it is the latter and $X$ is conjugate to $E_8(\#34^{\{r,s,r\}})$.  

Let $M$ be one of the classes of maximal $A_1$ subgroups. Then they are $E_8$-irreducible by Theorem \ref{maximalexcep}, yielding $E_8(\#40)$, $E_8(\#41)$ and $E_8(\#42)$. 

Finally, we check there are no more $E_8$-conjugacies between any of the irreducible $A_1$ subgroups by comparing the composition factors in Table \ref{E8tabcomps}. This completes the proof of Theorem \ref{thmE8}.  
\end{proof}

\section{Corollaries} \label{cors}

In this section we give the proofs of Corollaries \ref{A1subgroups} to \ref{nongcr}. Let $G$ be an exceptional algebraic group over an algebraically closed field of characteristic $p$. Corollary \ref{A1subgroups} is an immediate consequence of Theorems \ref{thmG2} to \ref{thmE8}, which prove $G$ contains a $G$-irreducible subgroup $A_1$ unless $G = E_6$ and $p=2$. 

For Corollary \ref{A1overgroups}, we consider Tables \ref{G2tab} to \ref{E8tab} when $p =2,3$. We see that all $G_2$-irreducible $A_1$ subgroups of $G_2$ are contained in $A_1 \tilde{A}_1$ from Table \ref{G2tab}. When $p=2$, Table \ref{F4tab} shows that $B_4$ contains all of the $F_4$-irreducible $A_1$ subgroups of $F_4$ (in fact the same is true for $C_4$). If $p=3$, we notice that $B_4$ and $\bar{A}_1 C_3$ both contain $F_4$-irreducible $A_1$ subgroups by Theorem \ref{thmF4}. Now suppose $G = E_6$. Then if $p=2$ there are no $E_6$-irreducible $A_1$ subgroups. When $p=3$, the $E_6$-irreducible $A_1$ subgroups $E_6(\#2)$ and $E_6(\#3)$ are listed as subgroups of $A_1 A_5$ and $A_2^3$, respectively. However, $E_6(\#2) < \bar{A}_1 C_3 < \bar{A}_1 A_5$ and $\bar{A}_1 C_3$ is also contained in $C_4$. Similarly, in the $M = A_2^3$ case of the proof of Theorem \ref{thmE6}, $E_6(\#3)$ is proved to be contained in $C_4$ (acting as $1 \otimes 1^{[r]} \otimes 1^{[s]}$ on $V_{C_4}(\lambda_1)$). Therefore, $C_4$ contains a conjugacy class representative of each $E_6$-irreducible subgroup when $p=3$. For $G = E_7$ and $E_8$ the result follows immediately from Tables \ref{E7tab} and \ref{E8tab}, respectively. 

Corollaries \ref{corconj} and \ref{corconjmin} follow from careful consideration of the composition factors listed in Tables \ref{G2tabcomps} to \ref{E8tabcomps}.

For the proof of Corollary \ref{nongcr} we first need the following three results. 

\begin{lem} \label{tensorsquare}

Let $X$ be an algebraic group of type $A_1$ over an algebraically closed field of characteristic $p=2$. Let $W$ be the $X$-module $1^{[a]} \otimes 1^{[b]} \otimes 1^{[c]} \otimes 1^{[d]}$ with $a,b,c,d$ distinct. Then the socle of $W \otimes W$ is a $1$-dimensional trivial module.   

\end{lem}

\begin{proof}
It suffices to show that
\vspace{-0.2cm}
\[ S_n:= \textrm{dim}(\textrm{Hom}_X (V_{A_1}(n),W \otimes W)) = \begin{cases} 0 \textrm{\ if \ } n \neq 0 \\ 1 \textrm{\ if \ } n=0. \end{cases}\] 
\vspace{-0.1cm}
Firstly, suppose $n=0$. Then $\textrm{Hom}_X(0,W \otimes W) \cong \textrm{Hom}_X(W,W) \cong K$ since $W$ is self-dual and irreducible. Therefore $S_0 = 1$ as required. Now suppose $n \neq 0$ and $S_n \neq 0$. Then $V_{A_1}(n)$ is a composition factor of $W \otimes W = (1^{[a]} \otimes 1^{[a]}) \otimes  (1^{[b]} \otimes 1^{[b]}) \otimes  (1^{[c]} \otimes 1^{[c]}) \otimes  (1^{[d]} \otimes 1^{[d]})$. Since $1 \otimes 1 = T(2) = 0 | 2 | 0$ it follows that $V_{A_1}(n)$ is isomorphic to $1^{[a+1]}$, $1^{[a+1]} \otimes 1^{[b+1]}$, $1^{[a+1]} \otimes 1^{[b+1]} \otimes 1^{[c+1]}$ or $1^{[a+1]} \otimes 1^{[b+1]} \otimes 1^{[c+1]} \otimes 1^{[d+1]}$, up to relabelling of $a,b,c,d$. It remains to check that none of these modules have non-zero homomorphisms to $W \otimes W$. First consider $V_{A_1}(n) \cong 1^{[a+1]}$. Then $\textrm{Hom}_X(1^{[a+1]}, W \otimes W) \cong \textrm{Hom}_X(1^{[a+1]} \otimes W, W)$. If $a+1 \not \in \{b,c,d\}$ then $1^{[a+1]} \otimes W$ is irreducible and not isomorphic to $W$, hence $S_n = 0$. So $a+1 \in \{b,c,d\}$ and we may assume that $a+1=b$. Noting that $1 \otimes 1 \otimes 1 = 3 \oplus 1^2$ (the composition factors are clear and there are no non-trivial extensions between any of them) we have
\vspace{-0.05cm}
\begin{align*}
\textrm{Hom}_X(1^{[a+1]}, W \otimes W) \cong & \ \textrm{Hom}_X(1^{[a+1]} \otimes 1^{[a+1]} \otimes 1^{[a+1]}, 1^{[a]} \otimes 1^{[a]} \otimes 1^{[c]} \otimes 1^{[c]} \otimes 1^{[d]} \otimes 1^{[d]}) \\  
\cong & \ \textrm{Hom}_X((1^{[a+1]} \otimes 1^{[a+2]}) \oplus (1^{[a+1]})^2, 1^{[a]} \otimes 1^{[a]} \otimes 1^{[c]} \otimes 1^{[c]} \otimes 1^{[d]} \otimes 1^{[d]}) \\
\cong & \ \textrm{Hom}_X(1^{[a+1]} \otimes 1^{[a+2]}, 1^{[a]} \otimes 1^{[a]} \otimes 1^{[c]} \otimes 1^{[c]} \otimes 1^{[d]} \otimes 1^{[d]}) \oplus \\
& \ \textrm{Hom}_X(1^{[a+1]}, 1^{[a]} \otimes 1^{[a]} \otimes 1^{[c]} \otimes 1^{[c]} \otimes 1^{[d]} \otimes 1^{[d]})^2 \\
\cong & \ \textrm{Hom}_X(1^{[a]} \otimes 1^{[a+1]} \otimes 1^{[a+2]} \otimes 1^{[c]} \otimes 1^{[d]}, 1^{[a]} \otimes 1^{[c]} \otimes 1^{[d]}) \oplus \\
& \ \textrm{Hom}_X(1^{[a]} \otimes 1^{[a+1]} \otimes 1^{[c]} \otimes 1^{[d]}, 1^{[a]} \otimes 1^{[c]} \otimes 1^{[d]})^2 \\
=: & \ A \oplus B^2.
\end{align*}
We have $1^{[a]} \otimes 1^{[c]} \otimes 1^{[d]}$ and $1^{[a]} \otimes 1^{[a+1]} \otimes 1^{[c]} \otimes 1^{[d]}$ (recall that $a+1 = b$ and that $a,b,c,d$ are distinct) are irreducible non-isomorphic modules and so $B = 0$. Furthermore, $1^{[a]} \otimes 1^{[a+1]} \otimes 1^{[a+2]} \otimes 1^{[c]} \otimes 1^{[d]}$ is irreducible if and only if $a+2 \not \in \{c,d\}$. Therefore $A = 0$ unless $a+2 \in \{c,d\}$. We may therefore assume $a+2 = c$ and just consider $A$:
\begin{align*}
A \cong & \ \textrm{Hom}_X(1^{[a+2]} \otimes 1^{[a+2]} \otimes 1^{[a+2]}, 1^{[a]} \otimes 1^{[a]} \otimes 1^{[a+1]} \otimes 1^{[d]} \otimes 1^{[d]})& \ \\
 \cong & \ \textrm{Hom}_X(1^{[a+2]} \otimes 1^{[a+3]}, 1^{[a]} \otimes 1^{[a]} \otimes 1^{[a+1]} \otimes 1^{[d]} \otimes 1^{[d]}) \oplus \\
& \ \textrm{Hom}_X(1^{[a+2]}, 1^{[a]} \otimes 1^{[a]} \otimes 1^{[a+1]} \otimes 1^{[d]} \otimes 1^{[d]})^2 \\
 \cong & \ \textrm{Hom}_X(1^{[a]} \otimes 1^{[a+1]} \otimes 1^{[a+2]} \otimes 1^{[a+3]} \otimes 1^{[d]}, 1^{[a]} \otimes 1^{[d]}) \oplus \\
& \ \textrm{Hom}_X(1^{[a]} \otimes 1^{[a+1]} \otimes 1^{[a+2]} \otimes 1^{[d]}, 1^{[a]} \otimes 1^{[d]})^2 \\
=: & \ C \oplus D^2. 
\end{align*}
As before, $D = 0$ and $C=0$ unless $a+3 = d$. So finally, we consider $C$ when $a+3 =d$.  
\begin{align*}
C \cong & \ \textrm{Hom}_X(1^{[a+3]} \otimes 1^{[a+3]} \otimes 1^{[a+3]}, 1^{[a]} \otimes 1^{[a]} \otimes 1^{[a+1]} \otimes 1^{[a+2]})& \ \\
 \cong & \ \textrm{Hom}_X(1^{[a+3]} \otimes 1^{[a+4]}, 1^{[a]} \otimes 1^{[a]} \otimes 1^{[a+1]} \otimes 1^{[a+2]}) \oplus \\
& \ \textrm{Hom}_X(1^{[a+3]}, 1^{[a]} \otimes 1^{[a]} \otimes 1^{[a+1]} \otimes 1^{[a+2]})^2 \\
 \cong & \ \textrm{Hom}_X(1^{[a]} \otimes 1^{[a+1]} \otimes 1^{[a+2]} \otimes 1^{[a+3]} \otimes 1^{[a+4]}, 1^{[a]}) \oplus \\
& \ \textrm{Hom}_X(1^{[a]} \otimes 1^{[a+1]} \otimes 1^{[a+2]} \otimes 1^{[a+3]}, 1^{[a]})^2 \\
= & \ 0.  
\end{align*}
We have proved that $\textrm{Hom}_X(1^{[a+1]}, W \otimes W) = 0$ and so $S_n= 0$, a contradiction. The three remaining cases are similar and in each one we find that $S_n = 0$, a contradiction. This completes the proof.  
\end{proof}

\begin{lem} \label{A1sinA7}
Let $X$ be a simple algebraic group  of type $A_1$ over an algebraically closed field of characteristic $p=2$. Let $W_1$ and $W_2$ be the $X$-modules $1 \otimes 1^{[r]} \otimes 1^{[s]}$ and $0 | (2 + 2^{[r]} + 2^{[s]}) | 0$, respectively. Then $\bigwedge^2(W_1)$ has a $1$-dimensional socle, whereas $\bigwedge^2(W_2)$ has a $7$-dimensional socle. In particular, $\bigwedge^2(W_1)$ is not isomorphic to $\bigwedge^2(W_2)$.  
\end{lem}

\begin{proof}
First, we claim that the socle of $W_1 \otimes W_1$ is $1$-dimensional. This follows from a similar calculation to that in the previous lemma. Since $\bigwedge^2(W_1)$ is a submodule of $W_1 \otimes W_1$, it follows that the socle of $\bigwedge^2(W_1)$ is also $1$-dimensional. Now we consider $\bigwedge^2(W_2)$. We claim that the socle is isomorphic to $2 + 2^{[r]} + 2^{[s]} + 0$. To prove this we consider $X$ as a diagonal subgroup of $Y \cong A_1^3$ via $(1,1^{[r]},1^{[s]})$ and let $W_3$ be the $Y$-module $(0,0,0) | ((2,0,0) + (0,2,0) + (0,0,2))| (0,0,0)$, so $W_3 \downarrow X = W_2$. The first step is to show that the socle of $W_3$ is isomorphic to $(2,0,0) + (0,2,0) + (0,0,2) + (0,0,0)$. Define a finite subgroup $S$ of $Y$, as the direct product of subgroups $\text{SL}(2,4) = \text{SL}(V) < A_1$ for each of the three simple factors of $Y$. We then construct $W_3 \downarrow S$ in Magma: Given the natural module $(V,0,0)$ for the first factor of $S$, we have that the tensor product $U_1 = (V,0,0) \otimes (V,0,0) = (0,0,0) | (V^{[1]},0,0) |(0,0,0)$. We then do the same for the second and third factors, forming $U_2$ and $U_3$ respectively. It is then straightforward to form the module $W_3 \downarrow S = (0,0,0) | ((V^{[1]},0,0) + (0,V^{[1]},0) + (0,0,V^{[1]})) | (0,0,0)$ from the direct sum of $U_1$, $U_2$ and $U_3$. We then use the inbuilt functions for the wedge square and the socle of a module in Magma to conclude that $\text{Soc}(\bigwedge^2(W_3) \downarrow S) \cong (V^{[1]},0,0) + (0,V^{[1]},0) + (0,0,V^{[1]}) + (0,0,0)$. Now we check the three conditions of Lemma \ref{subspaces}, applied to $S < Y$ acting on $\bigwedge^2(W_3)$. The $Y$-composition factors of $\bigwedge^2(W_3)$ are $(2,2,0) /$ $\!\!  (2,0,2) /$ $\!\!  (0,2,2) /$ $\!\!  (2,0,0)^2 /$ $\!\!  (0,2,0)^2 /$ $\!\!  (0,0,2)^2 /$ $\!\!  (0,0,0)^4$. Therefore conditions (i) and (iii) hold and it remains to show that the restriction map $\text{Ext}^1_Y(M,N) \rightarrow \text{Ext}^1_S(M,N)$ is injective for all pairs of $Y$-composition factors of $\bigwedge^2(W_3)$. Using the K\"{u}nneth formula \cite[10.85]{rot} and Lemma \ref{h1fora1}, we have $\text{Ext}^1_Y(M,N) = 0$ unless $M = (2,0,0)$ and $N = (2,2,0)$, $(2,0,2)$ or $(0,0,0)$ (up to swapping $M$, $N$ and cycling the three $A_1$ factors), in which case it is $1$-dimensional. The map $\text{Ext}^1_Y((2,0,0),(0,0,0)) \rightarrow \text{Ext}^1_S((2,0,0),(0,0,0))$ is injective since the non-trivial extension $(2,0,0)|(0,0,0)$ is found in the tensor product $(1,0,0) \otimes (1,0,0)$ for both $Y$ and $S$. Similarly, the non-trivial extension $(2,0,0)|(2,2,0)$ is found in $(2,1,0) \otimes (0,1,0)$ and $(2,0,0)|(2,0,2)$ is found in $(2,0,1) \otimes (0,0,1)$ for both $Y$ and $S$. We now apply Lemma \ref{subspaces} to conclude that $Y$ and $S$ fix the same subspaces of $\bigwedge^2(W_3)$. Therefore the socle of $\bigwedge^2(W_3)$ as a $Y$-module is $(2,0,0) + (0,2,0) + (0,0,2) + (0,0,0)$. It follows that the socle of $\bigwedge^2(W_2)$ as an $X$-module has dimension at least $7$. In fact, using Lemma \ref{subspaces} again, this time applied to $X < Y$ acting on $\bigwedge^2(W_3)$ implies that the socle of $\bigwedge^2(W_2)$ as an $X$-module has dimension $7$. 
\end{proof}

\begin{lem} \label{bada1b2p5}
Let $G = E_8$, $p=5$ and $M$ be a maximal subgroup $B_2$. Suppose $X$ is a maximal subgroup $A_1$ of $M$. Then $X$ is conjugate to $Y < A_8$ acting as $W(8) = 8 | 0$ on $V_{A_8}(\lambda_1)$.    
\end{lem}

\begin{proof} 
In the $M = B_2$ case of the proof of Theorem \ref{thmE8} it is proved that $X$ is contained in an $A_7$-parabolic subgroup of $G$. The composition factors of $X$ acting on $L(E_8)$ are $18^2 /$ $\!\!  16 /$ $\!\!  14^3 /$ $\!\!  12^4 /$ $\!\!  10^5 /$ $\!\!  8^6 /$ $\!\!  6^8 /$ $\!\!  4 /$ $\!\!  2^3 /$ $\!\!  0^3$ and it follows that the projection of $X$ to $A_7$ acts as $3 \otimes 1^{[1]}$ on $V_{A_7}(\lambda_1)$. Let $P = Q L$ be an $A_7$-parabolic subgroup of $E_8$ containing $X$ and let $Z$ be an subgroup $A_1$ of $L' = A_7$ acting as $3 \otimes 1^{[1]}$ on $V_{A_7}(\lambda_1)$, so the projection of $X$ to $A_7$ is $A_7$-conjugate to $Z$. By definition, $Y$ is a subgroup of an $A_7$-parabolic subgroup of $A_8$ and hence of $E_8$, with the projection of $Y$ to $A_7$ also $A_7$-conjugate to $Z$. By using the construction of $X$ in Magma from the $M = B_2$ case of the proof of Theorem \ref{thmE8} and the fact that $L(E_8) \downarrow A_8 = (\lambda_1 + \lambda_7) \oplus \lambda_3 \oplus \lambda_6$ we calculate that $X$ and $Y$ act the same way on $L(E_8)$, specifically as $W(18) + W(18)^* + T(16) + T(12)^2 + T(10)^3 + 14^3 + T(6)^3 + 4$. In particular, neither $X$ nor $Y$ have a trivial submodule on $L(E_8)$ and are thus not contained in an $A_7$ Levi subgroup. Therefore, both $X$ and $Y$ are non-$E_8$-cr. To prove that $X$ and $Y$ are conjugate it remains to show that there is only one $E_8$-conjugacy class of non-$E_8$-cr $A_1$ subgroups in $Q Z$. To do this we use the results and methods described in \cite{dav} and \cite{littho}.      

We first consider the action of $Z$ on the levels of $Q$. The action of $L'$ on the levels of $Q$ are as follows
\begin{align*}
Q/Q(2) \downarrow A_7 &= \lambda_5, \\
Q(2)/Q(3) \downarrow A_7 &= \lambda_2, \\
Q(3) \downarrow A_7 &= \lambda_7,
\end{align*}
and restricting to $Z < L'$ we have
\begin{align*}
Q/Q(2) \downarrow Z &= 18 + T(12) + T(10), \\
Q(2)/Q(3) \downarrow Z &= 14 + 10 + T(6), \\
Q(3) \downarrow Z &= 8. 
\end{align*}
In particular, using Lemma \ref{h1fora1} we see $H^1(Z,Q(i) / Q(i+1)) = 0$ for $i=1,2$ and $H^1(Z,Q(3)) \cong K$. It follows from \cite[Proposition 3.2.6, Lemma 3.2.11]{dav} that $H^1(Z,Q) \cong K$. Moreover, by \cite[Lemma 3.20]{littho}, there is at most one $E_8$-conjugacy class of non-$E_8$-cr $A_1$ subgroups contained in $QZ$ by considering the action of the $1$-dimensional non-trivial torus $Z(L)$. Since $X$ and $Y$ are non-$E_8$-cr and contained in $QZ$, there is exactly one class and so $X$ and $Y$ are $E_8$-conjugate. 
\end{proof}

It remains to prove Corollary \ref{nongcr}. The strategy for the proof is as follows. For each exceptional algebraic group $G$ we find all $M$-irreducible $A_1$ subgroups that are not $G$-irreducible from the proofs of Theorems \ref{thmG2} to \ref{thmE8}. Given such a subgroup $X$ we then check whether it satisfies the hypothesis of Corollary \ref{nongcr}. That is to say, we check whether $X$ is contained reducibly in another reductive, maximal connected subgroup $M'$, or $X$ is contained in a Levi subgroup of $G$. To do this we use the composition factors of $X$ on the minimal or adjoint module for $G$, using restriction from $M$. Of course, since $X$ is $G$-reducible there must exist some subgroup $Y \cong X$ inside a Levi factor $L'$ having the same composition factors as $X$. Therefore, we will require the exact module structure of $X$ acting on either the minimal module or adjoint module for $G$ to prove that $X$ is not contained in $L'$. 

\begin{pf}{Proof of Corollary \ref{nongcr}} 

First consider $G = G_2$. The proof of Theorem \ref{thmG2} shows that the only $M$-irreducible subgroup $A_1$ that is $G_2$-reducible is $X = A_1 \hookrightarrow M = A_1 \tilde{A}_1$ via $(1,1)$ when $p=2$. We need to check whether $X$ satisfies the hypothesis of Corollary \ref{nongcr}. The only subgroups of $G_2$ that have the same composition factors as $X$ are a Levi $\bar{A}_1$ and $A_1 < A_2$ embedded via $W(2)$. However, \cite[Theorem 1]{g2dav} shows that $X$ is not conjugate to either of these subgroups. Therefore $X$ is not contained reducibly in another reductive, maximal connected subgroup nor is it contained in a Levi subgroup of $G$. Hence $X$ satisfies the hypothesis of Corollary \ref{nongcr} and is listed in Table \ref{cortab}. 

Now let $G = F_4$. By the proof of Theorem \ref{thmF4}, there are no conjugacy classes of $M$-irreducible $A_1$ subgroups which are $F_4$-reducible. Indeed, the $B_4$-irreducible subgroups acting as $1 \otimes 1^{[r]} \otimes 1^{[s]}$ on $V_{B_4}(\lambda_1)$ when $p=2$ are shown to be conjugate to $B_4$-reducible subgroups acting as $0 | (2 + 2^{[r]} + 2^{[s]}) | 0$. Similarly for the $C_4$-irreducible subgroups acting as $1 \otimes 1^{[r]} \otimes 1^{[s]}$ on $V_{C_4}(\lambda_1)$ when $p=2$. Therefore there are no $A_1$ subgroups satisfying the hypothesis of Corollary \ref{nongcr}.  

Next, suppose $G=E_6$ and consider the $M$-irreducible $A_1$ subgroups that are $E_6$-reducible. These are all found in the proof of Theorem \ref{thmE6}. Let $X$ be such a subgroup. If $X$ is contained in a maximal $A_2$ then $p=5$ and we claim that $X$ is contained in $\bar{A}_1 A_5$ via $(1,W(5))$, hence $X$ does not satisfy the hypothesis of Corollary \ref{nongcr}. This is proved in \cite[Section 4.1]{littho} by showing there is only one conjugacy class of non-$E_6$-cr subgroups of type $A_1$ with the same composition factors as $X$ on $V_{56}$. Now suppose $X$ is contained in $C_4$ $(p \neq 2)$. If $X$ is contained in $\bar{A}_1 C_3 < C_4$ then $X$ is also contained in $\bar{A}_1 A_5$. By the proof of Theorem \ref{thmE6}, every $\bar{A}_1 A_5$-irreducible subgroup $A_1$ is $E_6$-irreducible. Hence $X$ is contained reducibly in $\bar{A}_1 A_5$ and so does not satisfy the hypothesis of the corollary. Similarly, if $X < C_2^2$ then $X$ is contained in a $D_5$ Levi subgroup and does not satisfy the hypothesis. Finally, suppose $X$ is contained in $F_4$. Then $X$ is contained in $B_4$ and hence contained in a $D_5$ Levi subgroup. Therefore there are no $A_1$ subgroups satisfying the hypothesis of Corollary \ref{nongcr}. 

For $G = E_7$ the result is checked in the same way as for $E_6$. First consider $A_7$-irreducible $A_1$ subgroups, all of which are $E_7$-reducible. If $p>2$ then such subgroups are contained in $C_4$, which is contained in an $E_6$ Levi subgroup by \cite[Table 8.2]{LS3} and so do not satisfy the hypothesis of the corollary . If $p=2$, then we need to consider a subgroup $X$ acting on $V_{A_7}(\lambda_1)$ as $W_1 = 1 \otimes 1^{[r]} \otimes 1^{[s]}$ $(0 < r < s)$. We claim that $X$ satisfies the hypothesis of Corollary \ref{nongcr}. To prove this, we start by considering the action of $X$ on $V_{56}$. We have $V_{56} \downarrow A_7 = V_{A_7}(\lambda_2) + V_{A_7}(\lambda_6)$ and hence $V_{56} \downarrow X = (\bigwedge^2(W_1))^2$. By Lemma \ref{A1sinA7}, we have $\bigwedge^2(W_1)$ is indecomposable and hence $X$ has two direct summands of dimension $28$ on $V_{56}$. From this, it follows that $X$ satisfies the hypothesis of Corollary \ref{nongcr} or $X$ is conjugate to an $M$-reducible subgroup $A_1$ of $A_7$. Suppose the latter is true. Considering the $X$-composition factors of $V_{56}$, it follows that $X$ is conjugate to $Y$, where $V_{A_7}(\lambda_1) \downarrow Y = W_2 = 0| (2 + 2^{[r]} + 2^{[s]}) | 0$. Now, we have $V_{56} \downarrow Y = (\bigwedge^2(W_2))^2$. Lemma \ref{A1sinA7} shows that $X$ and $Y$ are non-$\text{GL}(56,K)$-conjugate, and thus non-$E_7$-conjugate. This contradiction proves that $X$ satisfies the hypothesis of the corollary.

Now let $X$ be the irreducible subgroup $A_1$ contained in the maximal subgroup $A_2$ when $p=5,7$. Then $X$ is $E_7$-reducible and is in fact non-$E_7$-cr and conjugate to $Y = A_1 < A_7$ with $V_{A_7}(\lambda_1) \downarrow A_1 = W(7)$ for both $p=5, 7$. Indeed, by considering its composition factors on $V_{56}$, we find that the only Levi subgroup that can contain $X$ is $E_6$. However, in the case $M=A_2$ in the proof of Theorem \ref{thmE7}, we calculated that $V_{56} \downarrow X = (0|12|0)^2 + (4|8|4)^2$ when $p=7$ and therefore $X$ is not contained in $E_6$ since $E_6$ has a trivial direct summand on $V_{56}$. When $p=5$, we have $L(E_7) \downarrow A_2 = 44 + 11$. We know $V_{A_2}(11) \downarrow X = 4 + 2$ and need to find $V_{A_2}(44) \downarrow X$. To do this, we first use the fact that $40$ and $04$ are tilting, so their tensor product is also tilting and thus $40 \otimes 04 = 44 + T(33) + T(22)$.  Furthermore, $40 = S^4(10)$ and hence $40 \downarrow X = S^4(2) =  T(8) + 4$. Combining this with usual high weight calculations yields that $L(E_7) \downarrow X = T(16) +  14 + T(12) + T(10)^2 + T(8) + 4^3 + 2$.  In particular, we find that $X$ has no trivial direct summands on $L(E_7)$ and is therefore not contained in $E_6$. Thus $X$ is non-$E_7$-cr. To prove that $X$ is conjugate to $Y < A_7$, we need to prove there is only one $E_7$-conjugacy class of $A_1$ subgroups contained in an $E_6$-parabolic when $p=5,7$. This is done in \cite[Sections 5.1, 6.1]{littho}. 

Next, we consider the maximal subgroup $A_1 A_1$ when $p=5$. Then $X \hookrightarrow A_1 A_1$ via $(1,1)$ is $E_7$-reducible. In fact, $X$ is non-$E_7$-cr and conjugate to $Y = A_1 \hookrightarrow A_1 A_1< A_2 A_5$ via $(1,1)$ where the first factor $A_1$ acts as $2$ on $V_{A_2}(10)$ and the second factor $A_1$ acts as $W(5)$ on $V_{A_5}(\lambda_1)$. We will prove that $X$ is non-$E_7$-cr and the second statement is proved in \cite[Section 5.3]{littho}. Looking for a contradiction we suppose $X$ is $E_7$-cr. Then $X$ is contained irreducibly in some Levi factor $L'$. By considering the composition factors of $X$ on $V_{56}$, we have $L' = A_1 A_2 A_3$. From \cite[Table 10.2]{LS1}, we have $V_{56} \downarrow A_1 A_1 = ((2,3) |((6,3) + (2,5))| (2,3)) + (4,1)$. We can construct such a module for $A_1(25) \times A_1(25)$ in Magma and use it to show that $V_{56} \downarrow X = 9 + W(7) + W(7)^* + T(5)^3$. The action of $A_1 A_2 A_3$ on $V_{56}$ is completely reducible and has a direct summand of dimension 4. All of the direct summands of $X$ on $V_{56}$ have dimension at least 8, a contradiction. It follows that $X$ is non-$E_7$-cr but is contained reducibly in $A_2 A_5$ and so does not satisfy the hypothesis of Corollary \ref{nongcr}. 

The last $E_7$-reducible subgroups to consider are $X_1 \hookrightarrow A_1 A_1 < A_1 G_2$ via $(1,1)$ where the second factor $A_1$ is maximal in $G_2$ and $X_2 \hookrightarrow A_1 A_1 < G_2 C_3$ via $(1,1)$ where the first factor $A_1$ is maximal in $G_2$ and the second is maximal in $C_3$, both when $p=7$. From the proof of Theorem \ref{thmE7}, we see that $X_1$ and $X_2$ are $E_7$-reducible and by considering their composition factors on $V_{56}$, they are contained in an $A_1 A_2 A_3$-parabolic subgroup. Both $X_1$ and $X_2$ act on $V_{56}$ as $T(11) + T(9)^2 + T(7)$ , checked using the restriction of $V_{56}$ to $A_1 G_2$ and $G_2 C_3$ from \cite[Table 10.2]{LS1}. As $V_{56} \downarrow A_1 A_2 A_3$ has no direct summands of dimension at least $14$, it follows that neither $X_1$ nor $X_2$ are contained in a $A_1 A_2 A_3$ Levi subgroup and both satisfy the hypothesis of the corollary. Furthermore, it is shown in \cite[Section 6.2]{littho} that $X_1$ is conjugate to $X_2$, and hence only $X_1$ appears in Table \ref{cortab}.          

Finally, suppose $G=E_8$. First consider $X = A_1 < D_8$ acting as $1 \otimes 1^{[r]} \otimes 1^{[s]} \otimes 1^{[t]}$ $(0 < r < s < t)$ when $p=2$, where $X$ is the $E_8$-reducible class of such subgroups contained in $B_4 (\ddagger)$, as in \cite[Lemma 7.1]{tho1}. We claim that $X$ satisfies the hypothesis of Corollary \ref{nongcr}. To prove this, we will first show that $X$ has a $120$-dimensional indecomposable summand on $L(E_8)$. When $p=2$ we have $L(E_8) \downarrow D_8 = (0|\lambda_2|0) + \lambda_7$ and in particular, $\bigwedge^2(\lambda_1) = 0| \lambda_2 | 0$ is a direct summand. Furthermore, $\bigwedge^2(\lambda_1)$ is a submodule of $\lambda_1 \otimes \lambda_1$. Lemma \ref{tensorsquare} shows that the socle of  $\lambda_1 \otimes \lambda_1 \downarrow X$ is a $1$-dimensional trivial module and therefore simple. It follows that  $\bigwedge^2(\lambda_1)$ has a simple socle and is thus indecomposable. Therefore $X$ has a 120-dimensional indecomposable summand on $L(E_8)$. Now suppose $X$ does not satisfy the hypothesis of the corollary. Then $X$ is a subgroup of a Levi subgroup or another reductive, maximal connected subgroup of $E_8$. Since $X$ has a $120$-dimensional indecomposable summand, it follows that $X$ is an $E_7$-irreducible subgroup of $E_7$. As $p=2$, it follows from Corollary \ref{A1overgroups} that $X$ is contained in $\bar{A}_1 D_6$. But the largest dimensional indecomposable summand of $L(E_8) \downarrow \bar{A}_1 D_6$ is $(V_{A_1}(1),V_{A_6}(\lambda_6))$, which has dimension $64$. This is a contradiction and hence $X$ satisfies the hypothesis of the corollary.  

Next, we note that the $E_8$-reducible subgroup contained $A_8$-irreducibly in $A_8$ when $p=3$, is shown to be $D_8$-reducible in the proof of Theorem \ref{thmE8} and hence does not satisfy the hypothesis of the corollary.  

The last subgroup $A_1$ to consider is a maximal subgroup $A_1$ of $M = B_2$ when $p=5,7$, let this be $X$. The proof of \cite[Proposition 3.3.3]{LS1} shows that $X$ is contained in $A_8$ acting as $W(8) = 8 | 4$ when $p=7$. When $p=5$, Lemma \ref{bada1b2p5} shows that $X$ is also contained in $A_8$ acting as $W(8) = 8 | 0$. Therefore in both cases $X$ is contained reducibly in another reductive maximal connected subgroup of $E_8$ and so does not satisfy the hypothesis of Corollary \ref{nongcr}.
\qed
\end{pf}

\section{Tables} \label{tabs}

In this section we give the tables of composition factors for the $G$-irreducible $A_1$ subgroups from Theorems \ref{thmG2} to \ref{thmE8} acting on the minimal and adjoint modules for $G$. The tables use the unique identifier given to $G$-irreducible $A_1$ subgroups in Tables \ref{G2tab} to \ref{E8tab} and the composition factors are calculated by restriction from a reductive, maximal connected subgroup $M$. The notation used in the tables is described in Section \ref{nota}. The composition factors of $M$ on the minimal and adjoint modules of $G$ are listed in Theorem \ref{maximalexcep}.     

\begin{longtable}{>{\raggedright\arraybackslash}p{0.07\textwidth - 2\tabcolsep}>{\raggedright\arraybackslash}p{0.25\textwidth - 2\tabcolsep}>{\raggedright\arraybackslash}p{0.34\textwidth-2\tabcolsep}>{\raggedright\arraybackslash}p{0.34\textwidth-\tabcolsep}@{}}

\caption{The composition factors of the irreducible $A_1$ subgroups of $G_2$. \label{G2tabcomps}} \\

\hline

ID & Comp. factors of $V_7 \downarrow X$ & Comp. factors of $L(G_2) \downarrow X$ \\

\hline

1 & $1^{[r]} \otimes 1^{[s]} /$ $\!\! W(2)^{[s]}$ & $W(2)^{[r]} /$ $\!\! 1^{[r]} \otimes W(3)^{[s]} /$ $\!\! W(2)^{[s]}$ \\

2 & $2^2 /$ $\!\! 0$ & $W(4) /$ $\!\! 2^3$ \\

3 & $6$ & $W(10) /$ $\!\! 2$ \\

\hline

\end{longtable}

\begin{longtable}{>{\raggedright\arraybackslash}p{0.05\textwidth - 2\tabcolsep}>{\raggedright\arraybackslash}p{0.35\textwidth - 2\tabcolsep}>{\raggedright\arraybackslash}p{0.6\textwidth-\tabcolsep}@{}}

\caption{The composition factors of the irreducible $A_1$ subgroups of $F_4$. \label{F4tabcomps}} \\

\hline

ID & Comp. factors of $V_{26} \downarrow X$ & Comp. factors of $L(F_4) \downarrow X$ \\

\hline

1 & $ 1 \otimes 1^{[r]} /$  $\!\! 1 \otimes 1^{[s]} /$ $\!\! 1 \otimes 1^{[t]} /$ $\!\! 1^{[r]} \otimes 1^{[s]} /$ $\!\! 1^{[r]} \otimes 1^{[t]} /$ $\!\!  1^{[s]} \otimes 1^{[t]} /$ $\!\! 0^2$ & $W(2) /$ $\!\! 1 \otimes 1^{[r]} \otimes 1^{[s]} \otimes 1^{[t]} /$ $\!\! 1 \otimes 1^{[r]} /$ $\!\! 1 \otimes 1^{[s]} /$ $\!\! 1 \otimes 1^{[t]} /$ $\!\! W(2)^{[r]} /$ $\!\! 1^{[r]} \otimes 1^{[s]} /$ $\!\! 1^{[r]} \otimes 1^{[t]} /$ $\!\! W(2)^{[s]} /$  $\!\! 1^{[s]} \otimes 1^{[t]} /$ $\!\! W(2)^{[t]}$ \\

2 & $2^{[r]}/$ $\!\! 1^{[r]} \otimes 1^{[s]} \otimes 1^{[t]} /$ $\!\! 1^{[r]} \otimes 1^{[s]} \otimes 1^{[u]} /$ $\!\! 2^{[s]} /$ $\!\! 1^{[t]} \otimes 1^{[u]} /$  $\!\! 0^2$ & $2^{[r]} \otimes 2^{[s]} /$  $\!\! 2^{[r]} \otimes 1^{[t]} \otimes 1^{[u]} /$ $\!\! 2^{[r]} /$ $\!\! 1^{[r]} \otimes 1^{[s]} \otimes 1^{[t]} /$ $\!\! 1^{[r]} \otimes 1^{[s]} \otimes 1^{[u]} /$ $\!\! 2^{[s]} \otimes 1^{[t]} \otimes 1^{[u]} /$  $\!\! 2^{[s]} /$ $\!\! 1^{[t]} \otimes 1^{[u]} /$  $\!\! 2^{[t]} /$ $\!\! 2^{[u]} /$  $\!\! 0^4$  \\

3 & $2 /$ $\!\! 1 \otimes 1^{[r]} \otimes 1^{[s]} \otimes 1^{[t]} /$ $\!\! 2^{[r]} /$ $\!\! 2^{[s]} /$ $\!\! 2^{[t]} /$  $\!\! 0^2$ & $2 \otimes 2^{[r]} /$ $\!\! 2 \otimes 2^{[s]} /$  $\!\! 2 \otimes 2^{[t]} /$ $\!\! 2 /$ $\!\! 1 \otimes 1^{[r]} \otimes 1^{[s]} \otimes 1^{[t]} /$  $\!\! 2^{[r]} \otimes 2^{[s]} /$  $\!\! 2^{[r]} \otimes 2^{[t]} /$  $\!\! 2^{[r]} /$ $\!\! 2^{[s]} \otimes 2^{[t]} /$ $\!\! 2^{[s]} /$ $\!\! 2^{[t]} /$  $\!\! 0^4$ \\

4 & $2 /$ $\!\! (1 \otimes 1^{[r]} \otimes 1^{[s]})^2 /$ $\!\! 2^{[r]} /$ $\!\! 2^{[s]} /$  $\!\! 0$ &  $2 /$  $\!\! 2 \otimes 2^{[r]} /$ $\!\! 2 \otimes 2^{[s]} /$ $\!\! (1 \otimes 1^{[r]} \otimes 1^{[s]})^2 /$ $\!\! 2^{[r]} \otimes 2^{[s]} /$ $\!\! 2^{[r]} /$ $\!\! 2^{[s]}$ \\

5 & $W(3) \otimes 1^{[r]} /$ $\!\! 2 \otimes 2^{[r]} /$ $\!\! 1 \otimes W(3)^{[r]} /$ $\!\! 0$ & $\!\! W(4) \otimes 2^{[r]} /$  $\!\! W(3) \otimes 1^{[r]} /$ $\!\! 2 \otimes W(4)^{[r]} /$  $\!\! 2 /$  $\!\! 1 \otimes W(3)^{[r]} /$ $\!\! 2^{[r]}$ \\

6 & $4^{[r]} /$ $\!\! 3^{[r]} \otimes 1^{[s]} /$ $\!\! 3^{[r]} \otimes 1^{[t]} /$ $\!\! 1^{[s]} \otimes 1^{[t]} /$ $\!\! 0$ & $W(6)^{[r]} /$ $\!\! 4^{[r]} \otimes 1^{[s]} \otimes 1^{[t]} /$ $\!\! 3^{[r]} \otimes 1^{[s]} /$ $\!\! 3^{[r]} \otimes 1^{[t]} /$  $\!\! 2^{[r]} /$  $\!\! 2^{[s]} /$ $\!\! 2^{[t]} $ \\

7 & $10 /$ $\!\! 8 /$ $\!\! 4 /$ $\!\! 0$ & $W(14) /$ $\!\! 10^2 /$ $\!\! 6 /$ $\!\! 4 /$ $\!\! 2$ \\

8 & $1^{[r]} \otimes 5^{[s]} /$ $\!\! W(8)^{[s]} /$ $\!\! 4^{[s]}$ & $2^{[r]} /$ $\!\! 1^{[r]} \otimes W(9)^{[s]} /$ $\!\! 1^{[r]} \otimes 3^{[s]} /$ $\!\! W(10)^{[s]} /$ $\!\! 6^{[s]} /$ $\!\! 2^{[s]}$  \\

9 & $1^{[r]} \otimes 2^{[s]} \otimes 1^{[t]} /$ $\!\! W(4)^{[s]} /$ $\!\! 2^{[s]} \otimes 2^{[t]}$  &  $2^{[r]} /$ $\!\! 1^{[r]} \otimes W(4)^{[s]} \otimes 1^{[t]} /$ $\!\! 1^{[r]} \otimes W(3)^{[t]} /$  $\!\! W(4)^{[s]} \otimes 2^{[t]} /$ $\!\! 2^{[s]} /$ $\!\! 2^{[t]}$ \\

10 & $W(16) /$ $\!\! 8$ & $W(22) /$ $\!\! W(14) /$ $\!\! 10 /$ $\!\! 2$  \\

11 & $4^{[r]} /$ $\!\! 2^{[r]} \otimes 6^{[s]}$ & $4^{[r]} \otimes 6^{[s]} /$  $\!\! 2^{[r]} /$ $\!\! W(10)^{[s]} /$ $\!\! 2^{[s]}$ \\

\hline

\end{longtable}

\pagebreak
\begin{longtable}{>{\raggedright\arraybackslash}p{0.05\textwidth - 2\tabcolsep}>{\raggedright\arraybackslash}p{0.35\textwidth - 2\tabcolsep}>{\raggedright\arraybackslash}p{0.6\textwidth-\tabcolsep}@{}}

\caption{The composition factors of the irreducible $A_1$ subgroups of $E_6$. \label{E6tabcomps}} \\

\hline

ID & Comp. factors of $V_{27} \downarrow X$ & Comp. factors of $L(E_6) \downarrow X$ \\

\hline

1 &  $1^{[r]} \otimes 5^{[s]} /$ $\!\! W(8)^{[s]} /$ $\!\! 4^{[s]} /$ $\!\! 0$ & $2^{[r]} /$  $\!\! 1^{[r]} \otimes W(9)^{[s]} /$  $\!\! 1^{[r]} \otimes 5^{[s]} /$  $\!\! 1^{[r]} \otimes 3^{[s]} /$ $\!\! W(10)^{[s]} /$ $\!\! W(8)^{[s]} /$ $\!\! 6^{[s]} /$ $\!\! 4^{[s]} /$ $\!\! 2^{[s]}$ \\

2 & $1^{[r]} \otimes 2^{[s]} \otimes 1^{[t]} /$ $\!\! W(4)^{[s]} /$ $\!\! 2^{[s]} \otimes 2^{[t]} /$ $\!\! 0$ & $2^{[r]} /$ $\!\! 1^{[r]} \otimes W(4)^{[s]} \otimes 1^{[t]} /$  $\!\! 1^{[r]} \otimes 2^{[s]} \otimes 1^{[t]} /$ $\!\! 1^{[r]} \otimes W(3)^{[t]} /$  $\!\! W(4)^{[s]} \otimes 2^{[t]} /$ $\!\! W(4)^{[s]} /$ $\!\! 2^{[s]} \otimes 2^{[t]} /$ $\!\! 2^{[s]} /$ $\!\! 2^{[t]}$   \\

3 & $2 \otimes 2^{[r]} /$ $\!\! 2 \otimes 2^{[s]} /$ $\!\! 2^{[r]} \otimes 2^{[s]}$ & $W(4) /$ $\!\! (2 \otimes 2^{[r]} \otimes 2^{[s]})^2 /$ $\!\! 2 /$ $\!\! W(4)^{[r]} /$ $\!\! 2^{[r]} /$ $\!\! W(4)^{[s]} /$ $\!\! 2^{[s]}$  \\

4 & $4^{[r]} /$ $\!\! 2^{[r]} \otimes 6^{[s]} /$ $\!\! 0$ & $4^{[r]} \otimes 6^{[s]} /$ $\!\! 4^{[r]} /$ $\!\! 2^{[r]} \otimes 6^{[s]} /$  $\!\! 2^{[r]} /$ $\!\! W(10)^{[s]} /$ $\!\! 2^{[s]}$  \\

5 & $W(16) /$ $\!\! 8 /$ $\!\! 0$ & $W(22) /$ $\!\! W(16) /$ $\!\! W(14) /$ $\!\! 10 /$ $\!\! 8 /$ $\!\! 2$ \\

6 & $W(12) /$ $\!\! 8 /$ $\!\! 4 $ & $W(16) /$ $\!\! W(14) /$ $\!\! 10^2 /$ $\!\! 8 /$ $\!\! 6 /$ $\!\! 4 /$ $\!\! 2$ \\

\hline

\end{longtable}

\begin{longtable}{>{\raggedright\arraybackslash}p{0.05\textwidth - 2\tabcolsep}>{\raggedright\arraybackslash}p{0.35\textwidth - 2\tabcolsep}>{\raggedright\arraybackslash}p{0.6\textwidth-\tabcolsep}@{}}

\caption{The composition factors of the irreducible $A_1$ subgroups of $E_7$. \label{E7tabcomps}} \\

\hline

ID & Comp. factors of $V_{56} \downarrow X$ & Comp. factors of $L(E_7) \downarrow X$ \\

\hline

1 & $1^{[r]} \otimes 5^{[s]} \otimes 1^{[t]} /$ $\!\! W(8)^{[s]} \otimes 1^{[t]} /$ $\!\! 4^{[s]} \otimes 1^{[t]} /$ $\!\! 3^{[t]}$ & $2^{[r]} /$ $\!\! 1^{[r]} \otimes W(9)^{[s]} /$ $\!\! 1^{[r]} \otimes 5^{[s]} \otimes 2^{[t]} /$  $\!\! 1^{[r]} \otimes 3^{[s]} /$ $\!\! W(10)^{[s]} /$ $\!\! W(8)^{[s]} \otimes 2^{[t]} /$ $\!\! 6^{[s]} /$   $\!\! 4^{[s]} \otimes 2^{[t]} /$ $\!\! 2^{[s]} /$ $\!\! 2^{[t]}$    \\

2 & $1^{[r]} \otimes 5^{[s]} \otimes 1^{[t]} /$ $\!\! W(9)^{[s]} /$ $\!\! 5^{[s]} \otimes 2^{[t]} /$  $\!\! 3^{[s]}$ & $2^{[r]} /$ $\!\! 1^{[r]} \otimes W(8)^{[s]} \otimes 1^{[t]} /$ $\!\! 1^{[r]} \otimes 4^{[s]} \otimes 1^{[t]} /$ $\!\! 1^{[r]} \otimes 3^{[s]} /$ $\!\! W(10)^{[s]} /$ $\!\! W(8)^{[s]} \otimes 2^{[t]} /$ $\!\! 6^{[s]} /$ $\!\! 4^{[s]} \otimes 2^{[t]} /$ $\!\! 2^{[s]} /$ $\!\! 2^{[t]}$ \\

3 & $1^{[r]} \otimes 2^{[s]} \otimes 1^{[t]} \otimes 1^{[u]} /$ $\!\!W(4)^{[s]} \otimes 1^{[u]} /$ $\!\! 2^{[s]} \otimes 2^{[t]} \otimes 1^{[u]} /$ $\!\! W(3)^{[u]}$ & $2^{[r]} /$ $\!\! 1^{[r]} \otimes W(4)^{[s]} \otimes 1^{[u]} /$ $\!\! 1^{[r]} \otimes 2^{[s]} \otimes 2^{[t]} \otimes 1^{[u]} /$ $\!\! 1^{[r]} \otimes W(3)^{[u]} /$ $\!\! W(4)^{[s]} \otimes 2^{[t]} /$ $\!\! W(4)^{[s]} \otimes 2^{[u]} /$ $\!\! 2^{[s]} \otimes 2^{[t]} \otimes 2^{[u]} /$ $\!\! 2^{[s]} /$ $\!\! 2^{[t]} /$ $\!\! 2^{[u]}$  \\

4 & $1^{[r]} \otimes 10^{[s]} /$ $\!\!  1^{[r]} /$ $\!\!  W(15)^{[s]} /$ $\!\!  9^{[s]} /$ $\!\!  5^{[s]} $ &  $2^{[r]} /$ $\!\! 1^{[r]} \otimes W(15)^{[s]} /$ $\!\! 1^{[r]} \otimes 9^{[s]} /$ $\!\! 1^{[r]} \otimes 5^{[s]} /$ $\!\! W(18)^{[s]} /$ $\!\! W(14)^{[s]} /$ $\!\! (10^{[s]})^2 /$ $\!\! 6^{[s]} /$ $\!\! 2^{[s]}$ \\

5 & $1^{[r]} \otimes 8^{[s]} /$ $\!\!  1^{[r]} \otimes 2^{[t]} /$ $\!\!  10^{[s]} \otimes 1^{[t]} /$ $\!\!  4^{[s]} \otimes 1^{[t]}$  & $2^{[r]} /$ $\!\! 1^{[r]} \otimes 10^{[s]} \otimes 1^{[t]} /$ $\!\! 1^{[r]} \otimes 4^{[s]} \otimes 1^{[t]} /$ $\!\! W(14)^{[s]} /$ $\!\! 10^{[s]} /$  $\!\! 8^{[s]} \otimes 2^{[t]} /$ $\!\! 6^{[s]} /$ $\!\! 2^{[s]} /$ $\!\! 2^{[t]}$ \\

6 & $1^{[r]} \otimes 6^{[s]} /$ $\!\!  1^{[r]} \otimes 4^{[t]} /$ $\!\!  6^{[s]} \otimes 3^{[t]} /$ $\!\!  3^{[t]}$  & $2^{[r]} /$ $\!\! 1^{[r]} \otimes 6^{[s]} \otimes 3^{[t]} /$ $\!\! 1^{[r]} \otimes 3^{[t]} /$ $\!\! W(10)^{[s]} /$ $\!\! 6^{[s]} /$ $\!\! 6^{[s]} \otimes 4^{[t]} /$ $\!\! 2^{[s]} /$ $\!\! 6^{[t]} /$ $\!\! 2^{[t]}$ \\

7 & $1^{[r]} \otimes 6^{[s]} /$ $\!\!  1^{[r]} \otimes 1^{[t]} \otimes 1^{[u]} /$ $\!\!  1^{[r]} /$ $\!\!  6^{[s]} \otimes 1^{[t]} /$ $\!\!  6^{[s]} \otimes 1^{[u]} /$ $\!\!  1^{[t]} /$ $\!\!  1^{[u]}$  &  $2^{[r]} /$ $\!\! 1^{[r]} \otimes 6^{[s]} \otimes  1^{[t]} /$ $\!\! 1^{[r]} \otimes 6^{[s]} \otimes 1^{[u]} /$ $\!\! 1^{[r]} \otimes 1^{[t]} /$  $\!\! 1^{[r]} \otimes 1^{[u]} /$ $\!\! W(10)^{[s]} /$  $\!\! 6^{[s]} \otimes 1^{[t]} \otimes 1^{[u]} /$ $\!\! (6^{[s]})^2 /$  $\!\! 2^{[s]} /$ $\!\! 1^{[t]} \otimes 1^{[u]} /$ $\!\! 2^{[t]} /$ $\!\! 2^{[u]} /$ \\

8 & $1^{[r]} \otimes 4^{[s]} /$ $\!\!  1^{[r]} \otimes 2^{[t]} /$ $\!\!  1^{[r]} \otimes 1^{[u]} \otimes 1^{[v]} /$ $\!\! 3^{[s]} \otimes 1^{[t]} \otimes 1^{[u]} /$ $\!\! 3^{[s]} \otimes 1^{[t]} \otimes 1^{[v]}$  &  $W(6)^{[s]} /$  $\!\! 2^{[r]} /$ $\!\! 1^{[r]} \otimes 3^{[s]} \otimes 1^{[t]} \otimes 1^{[u]} /$ $\!\! 1^{[r]} \otimes 3^{[s]} \otimes 1^{[t]} \otimes 1^{[v]} /$ $\!\! 4^{[s]} \otimes 2^{[t]} /$ $\!\! 4^{[s]} \otimes 1^{[u]} \otimes 1^{[v]} /$ $\!\! 2^{[s]} /$ $\!\! 2^{[t]} \otimes 1^{[u]} \otimes 1^{[v]} /$ $\!\! 2^{[t]} /$ $\!\! 2^{[u]} /$ $\!\!  2^{[v]}$  \\

9 & $1^{[r]} \otimes 4^{[s]} /$ $\!\!  1^{[r]} \otimes 2^{[s]} /$ $\!\!1^{[r]} \otimes 1^{[t]} \otimes 1^{[u]} /$ $\!\! 4^{[s]} \otimes 1^{[t]} /$  $\!\! 4^{[s]} \otimes 1^{[u]} /$  $\!\! 2^{[s]} \otimes 1^{[t]} /$ $\!\!  2^{[s]} \otimes 1^{[u]}$ &  $2^{[r]} /$ $\!\! 1^{[r]} \otimes 4^{[s]} \otimes 1^{[t]} /$ $\!\! 1^{[r]} \otimes 4^{[s]} \otimes  1^{[u]} /$ $\!\! 1^{[r]} \otimes 2^{[s]} \otimes 1^{[t]} /$  $\!\! 1^{[r]} \otimes 2^{[s]} \otimes  1^{[u]} /$  $\!\! (W(6)^{[s]})^2 /$  $\!\! 4^{[s]} \otimes 1^{[t]} \otimes 1^{[u]} /$ $\!\! 2^{[s]} \otimes 1^{[t]} \otimes 1^{[u]} /$ $\!\! 4^{[s]} /$  $\!\! (2^{[s]})^3 /$ $\!\! 2^{[t]} /$ $\!\! 2^{[u]}$ \\

10 & $1^{[r]} \otimes 2^{[s]} \otimes 2^{[t]} /$ $\!\!  1^{[r]} \otimes 2^{[u]} /$ $\!\!  W(3)^{[s]} \otimes 1^{[t]} \otimes 1^{[u]} /$ $\!\!  W(3)^{[t]} \otimes 1^{[s]} \otimes 1^{[u]}$  &  $2^{[r]} /$ $\!\! 1^{[r]} \otimes W(3)^{[s]} \otimes 1^{[t]} \otimes 1^{[u]} /$ $\!\! 1^{[r]} \otimes 1^{[s]} \otimes W(3)^{[t]} \otimes 1^{[u]} /$ $\!\! W(4)^{[s]} \otimes 2^{[t]} /$ $\!\! 2^{[s]} \otimes W(4)^{[t]} /$  $\!\! 2^{[s]} \otimes 2^{[t]} \otimes 2^{[u]} /$ $\!\! 2^{[s]} /$  $\!\! 2^{[t]} /$ $\!\! 2^{[u]}$     \\

11 & $1^{[r]} \otimes 2^{[s]} /$ $\!\!  1^{[r]} \otimes 2^{[t]} /$ $\!\!  1^{[r]} \otimes 2^{[u]} /$ $\!\!  1^{[r]} \otimes 2^{[v]} /$ $ (1^{[s]} \otimes 1^{[t]} \otimes 1^{[u]} \otimes 1^{[v]})^2$ & $2^{[r]} /$ $\!\! (1^{[r]} \otimes 1^{[s]} \otimes 1^{[t]} \otimes 1^{[u]} \otimes 1^{[v]})^2 /$ $\!\! 2^{[s]} /$ $\!\! 2^{[s]} \otimes 2^{[t]} /$ $\!\! 2^{[s]} \otimes 2^{[u]} /$ $\!\! 2^{[s]} \otimes 2^{[v]} /$ $\!\! 2^{[t]} /$  $\!\! 2^{[t]} \otimes 2^{[u]} /$ $\!\! 2^{[t]} \otimes 2^{[v]} /$ $\!\! 2^{[u]} /$ $\!\! 2^{[u]} \otimes 2^{[v]} /$ $\!\! 2^{[v]}$ \\

12 &  $1 \otimes 1^{[r]} \otimes 1^{[s]} /$ $\!\!  1 \otimes 1^{[t]} \otimes 1^{[u]} /$ $\!\!  1 \otimes 1^{[v]} \otimes 1^{[w]} /$ $\!\!  1^{[r]} \otimes 1^{[t]} \otimes 1^{[v]} /$ $\!\!  1^{[r]} \otimes 1^{[u]} \otimes 1^{[w]} /$ $\!\!  1^{[s]} \otimes 1^{[t]} \otimes 1^{[w]} /$ $\!\!  1^{[s]} \otimes 1^{[u]} \otimes 1^{[v]}$ & $W(2) /$ $\!\! 1 \otimes 1^{[r]} \otimes 1^{[u]} \otimes 1^{[v]} /$ $\!\! 1 \otimes 1^{[r]} \otimes 1^{[t]} \otimes 1^{[w]} /$ $\!\! 1 \otimes 1^{[s]} \otimes 1^{[t]} \otimes 1^{[v]} /$ $\!\! 1 \otimes 1^{[s]} \otimes 1^{[u]} \otimes 1^{[w]} /$ $\!\! W(2)^{[r]} /$ $\!\! 1^{[r]} \otimes 1^{[s]} \otimes 1^{[t]} \otimes 1^{[u]} /$ $\!\! 1^{[r]} \otimes 1^{[s]} \otimes 1^{[v]} \otimes 1^{[w]} /$ $\!\! W(2)^{[s]} /$  $\!\! W(2)^{[t]} /$ $\!\! 1^{[t]} \otimes 1^{[u]} \otimes 1^{[v]} \otimes 1^{[w]} /$ $\!\! W(2)^{[u]} /$ $\!\! W(2)^{[v]} /$ $\!\! W(2)^{[w]} /$ \\

13 & $1^{[r]} \otimes 2^{[s]} /$ $\!\! 1^{[r]} \otimes 2^{[t]} /$ $\!\! 1^{[r]} \otimes 2^{[u]} /$ $\!\! 1^{[r]} \otimes 1^{[v]} \otimes 1^{[w]} /$ $\!\! (1^{[r]})^2 /$ $\!\! 1^{[s]} \otimes 1^{[t]} \otimes 1^{[u]} \otimes 1^{[v]} /$ $\!\! 1^{[s]} \otimes 1^{[t]} \otimes 1^{[u]} \otimes 1^{[w]}$  &  $2^{[r]} /$ $\!\! 1^{[r]} \otimes 1^{[s]} \otimes 1^{[t]} \otimes 1^{[u]} \otimes 1^{[v]} /$ $\!\! 1^{[r]} \otimes 1^{[s]} \otimes 1^{[t]} \otimes 1^{[u]} \otimes 1^{[w]} /$ $\!\! 2^{[s]} \otimes 2^{[t]} /$ $\!\! 2^{[s]} \otimes 2^{[u]} /$ $\!\! 2^{[s]} \otimes 1^{[v]} \otimes 1^{[w]} /$ $\!\! (2^{[s]})^2 /$ $\!\! 2^{[t]} \otimes 2^{[u]} /$  $\!\! 2^{[t]} \otimes 1^{[v]} \otimes 1^{[w]} /$ $\!\! (2^{[t]})^2 /$ $\!\! 2^{[u]} \otimes 1^{[v]} \otimes 1^{[w]} /$  $\!\! (2^{[u]})^2 /$ $\!\! 2^{[v]} /$  $\!\! (1^{[v]} \otimes 1^{[w]})^2 /$  $\!\! 2^{[w]} /$ $\!\!  0^6$ \\

14 & $1^{[r]} \otimes 2^{[s]} /$ $\!\!  1^{[r]} \otimes 2^{[t]} /$ $\!\!  1^{[r]} \otimes 2^{[u]} /$ $\!\!  1^{[r]} \otimes 2^{[v]} /$ $\!\!  1^{[r]} \otimes 2^{[w]} /$ $\!\!  (1^{[r]})^2 /$ $\!\!  1^{[s]} \otimes 1^{[t]} \otimes 1^{[u]} \otimes 1^{[v]} \otimes 1^{[w]}$ & $2^{[r]} /$ $\!\! 1^{[r]} \otimes 1^{[s]} \otimes 1^{[t]} \otimes 1^{[u]} \otimes 1^{[v]} \otimes 1^{[w]}  /$  $\!\!  2^{[s]} \otimes 2^{[t]} /$ $\!\! 2^{[s]} \otimes 2^{[u]} /$ $\!\!  2^{[s]} \otimes 2^{[v]} /$ $\!\! 2^{[s]} \otimes 2^{[w]} /$  $\!\! (2^{[s]})^2 /$  $\!\! 2^{[t]} \otimes 2^{[u]} /$ $\!\! 2^{[t]} \otimes 2^{[v]} /$ $\!\! 2^{[t]} \otimes 2^{[w]} /$  $\!\! (2^{[t]})^2 /$ $\!\!  2^{[u]} \otimes 2^{[v]} /$ $\!\! 2^{[u]} \otimes 2^{[w]} /$ $\!\! (2^{[u]})^2 /$ $\!\! 2^{[v]} \otimes 2^{[w]} /$  $\!\! (2^{[v]})^2 /$ $\!\! (2^{[w]})^2 /$ $\!\! 0^6 $ \\
 
15 &  $6^{[r]} \otimes 5^{[s]} /$ $\!\! W(9)^{[s]} /$ $\!\! 3^{[s]}$  & $W(10)^{[r]} /$ $\!\! 6^{[r]} \otimes W(8)^{[s]} /$ $\!\! 6^{[r]} \otimes 4^{[s]} /$ $\!\! 2^{[r]} /$ 
  $\!\! W(10)^{[s]} /$ $\!\! 6^{[s]} /$  $\!\! 2^{[s]}$  \\

16 & $6^{[r]} \otimes 2^{[s]} \otimes 1^{[t]} /$ $\!\! 4^{[s]} \otimes 1^{[t]} /$ $\!\! 3^{[t]}$ & $W(10)^{[r]} /$  $\!\! 6^{[r]} \otimes 4^{[s]} /$  $\!\! 6^{[r]} \otimes 2^{[s]} \otimes 2^{[t]} /$ $\!\! 2^{[r]} /$  $\!\! 4^{[s]} \otimes 2^{[t]} /$ $\!\! 2^{[s]} /$ $\!\! 2^{[t]}$ \\ 

17 & $3^{[r]} \otimes 6^{[s]} /$ $\!\! 1^{[r]} \otimes W(10)^{[s]} /$ $\!\! 1^{[r]} \otimes 2^{[s]}$  & $4^{[r]} \otimes 6^{[s]} /$  $\!\! 2^{[r]} \otimes W(12)^{[s]} /$ $\!\! 2^{[r]} \otimes 8^{[s]} /$ $\!\! 2^{[r]} \otimes 4^{[s]} /$ $\!\! 2^{[r]} /$ $\!\! W(10)^{[s]} /$  $\!\! 2^{[s]}$ \\

18 & $3^{[r]} /$ $\!\! 1^{[r]} \otimes W(16)^{[s]} /$ $\!\! 1^{[r]} \otimes 8^{[s]}$ & $2^{[r]} \otimes W(16)^{[s]} /$ $\!\! 2^{[r]} \otimes 8^{[s]} /$ $\!\! 2^{[r]} /$  $\!\! W(22)^{[s]} /$ $\!\! W(14)^{[s]} /$ $\!\! W(10)^{[s]} /$ $\!\! 2^{[s]}$  \\

19 &  $W(6)^{[r]} \otimes 3^{[s]} /$ $\!\! 4^{[r]} \otimes 1^{[s]} /$ $\!\! 2^{[r]} \otimes 5^{[s]}$  & $W(6)^{[r]} \otimes 4^{[s]} /$ $\!\! 4^{[r]} \otimes W(6)^{[s]} /$ $\!\! 4^{[r]} \otimes 2^{[s]} /$ $\!\! 2^{[r]} \otimes W(8)^{[s]} /$ $\!\! 2^{[r]} \otimes 4^{[s]} /$ $\!\! 2^{[r]} /$ $\!\! 2^{[s]}$  \\

20 & $W(21) / 15 / 11 / 5$  & $W(26) /$ $\!\! W(22) /$ $\!\! W(18) /$ $\!\! 16 /$ $\!\! 14 /$ $\!\! 10^2 /$ $\!\! 6 /$ $\!\! 2$ \\

21 & $W(27) / 17 / 9$  & $W(34) /$ $\!\! W(26) /$ $\!\! W(22) /$ $\!\! 18 /$ $\!\! 14 /$ $\!\! 10 /$ $\!\! 2$ \\

\hline

\end{longtable}

\begin{longtable}{>{\raggedright\arraybackslash}p{0.05\textwidth - 2\tabcolsep}>{\raggedright\arraybackslash}p{0.95\textwidth-\tabcolsep}@{}}

\caption{The composition factors of the irreducible $A_1$ subgroups of $E_8$. \label{E8tabcomps}} \\

\hline

ID & Composition Factors of $L(E_8) \downarrow X$ \\

\hline

1 & $W(16)^{[r]} /$ $\!\! W(14)^{[r]} /$ $\!\! (W(12)^{[r]} \otimes 2^{[s]})^2 /$ $\!\! (10^{[r]})^2 /$ $\!\! (8^{[r]} \otimes 2^{[s]})^2 /$ $\!\! 8^{[r]} /$ $\!\! 6^{[r]} /$ $\!\! (4^{[r]} \otimes 2^{[s]})^2 /$ $\!\! 4^{[r]} /$ $\!\! 2^{[r]} /$ $\!\! 4^{[s]} /$ $\!\! 2^{[s]}$ \\

2 & $W(15)^{[r]} \otimes 1^{[s]} /$ $\!\! W(14)^{[r]} /$ $\!\! W(12)^{[r]} \otimes 2^{[s]} /$ $\!\! W(11)^{[r]} \otimes 1^{[s]} /$ $\!\! 10^{[r]} /$ $\!\! 9^{[r]} \otimes 1^{[s]} /$ $\!\! 8^{[r]} \otimes 2^{[s]} /$ $\!\! 7^{[r]} \otimes 3^{[s]} /$ $\!\! 6^{[r]} /$ $\!\! 5^{[r]} \otimes 1^{[s]} /$ $\!\! 4^{[r]} \otimes 2^{[s]} /$ $\!\! 3^{[r]} \otimes 1^{[s]} /$ $\!\! 2^{[r]} /$ $\!\! 2^{[s]}$ \\

3 & $W(8) /$ $\!\! (W(6) \otimes 4^{[r]})^2 /$ $\!\! W(6) /$ $\!\! (4 \otimes W(6)^{[r]})^2 /$ $\!\! (4 \otimes 2^{[r]})^2 /$ $\!\! 4 /$ $\!\! (2 \otimes 4^{[r]})^2 /$ $\!\! 2 /$ $\!\! W(8)^{[r]} /$ $\!\! W(6)^{[r]} /$ $\!\! 4^{[r]} /$ $\!\! 2^{[r]} $  \\

4 & $W(7) \otimes 3^{[r]} /$ $\!\!  W(6) \otimes 4^{[r]} /$ $\!\!  W(6) /$ $\!\!  W(5) \otimes 3^{[r]} /$ $\!\!  4 \otimes W(6)^{[r]} /$ $\!\!  4 \otimes 2^{[r]} /$ $\!\!  3 \otimes W(7)^{[r]} /$ $\!\!  3 \otimes W(5)^{[r]} /$ $\!\! 3 \otimes 1^{[r]} /$ $\!\! 2 \otimes 4^{[r]} /$ $\!\!  2 /$ $\!\!  1 \otimes 3^{[r]} /$ $\!\!  W(6)^{[r]} /$ $\!\!  2^{[r]}$  \\

5 & $W(10)^{[r]} /$ $\!\! W(9)^{[r]} \otimes 1^{[t]} /$ $\!\! W(8)^{[r]} \otimes 2^{[s]} /$ $\!\! W(8)^{[r]} \otimes 1^{[s]} \otimes 1^{[u]} /$ $\!\! 6^{[r]} /$ $\!\! 5^{[r]} \otimes 2^{[s]} \otimes 1^{[t]}/$ $\!\! 5^{[r]} \otimes 1^{[s]} \otimes 1^{[t]} \otimes 1^{[u]} /$ $\!\! 4^{[r]} \otimes 2^{[s]} /$ $\!\! 4^{[r]} \otimes 1^{[s]} \otimes 1^{[u]}/$ $\!\! 3^{[r]} \otimes 1^{[t]}/$ $\!\! 2^{[r]} /$ $\!\! 3^{[s]} \otimes 1^{[u]}/$ $\!\! 2^{[s]} /$ $\!\! 2^{[t]} /$ $\!\! 2^{[u]} $ \\

6 & $W(4)^{[r]} \otimes 2^{[s]} /$ $\!\!  W(4)^{[r]} \otimes 1^{[s]} \otimes 1^{[u]} /$ $\!\!  W(4)^{[r]} \otimes 2^{[t]} /$ $\!\!  W(4)^{[r]} \otimes 1^{[t]} \otimes 1^{[v]} /$ $\!\!  2^{[r]} \otimes 2^{[s]} \otimes 2^{[t]} /$ $\!\!  2^{[r]} \otimes 2^{[s]} \otimes 1^{[t]} \otimes 1^{[v]} /$ $\!\!  2^{[r]} \otimes 1^{[s]} \otimes 2^{[t]} \otimes 1^{[v]} /$ $\!\!  2^{[r]} \otimes 1^{[s]} \otimes 1^{[t]} \otimes 1^{[u]} \otimes 1^{[v]} /$ $\!\!  2^{[r]} /$ $\!\!  3^{[s]} \otimes 1^{[v]} /$ $\!\!  2^{[s]} /$ $\!\!  3^{[t]} \otimes 1^{[u]} /$ $\!\!  2^{[t]} /$ $\!\!  2^{[u]} /$ $\!\!  2^{[v]}$  \\

7 & $W(9)^{[r]} \otimes 1^{[s]} /$ $\!\!  W(8)^{[r]} \otimes 2^{[s]} /$ $\!\!  W(7)^{[r]} \otimes 3^{[s]} /$ $\!\!  W(6)^{[r]} \otimes 4^{[s]} /$ $\!\!  W(6)^{[r]} /$ $\!\!  W(5)^{[r]} \otimes 3^{[s]} /$ $\!\!  W(5)^{[r]} \otimes 1^{[s]} /$ $\!\!  (4^{[r]} \otimes 2^{[s]})^2 /$ $\!\!  3^{[r]} \otimes W(5)^{[s]} /$ $\!\!  3^{[r]} \otimes 1^{[s]} /$ $\!\!  2^{[r]} \otimes 4^{[s]} /$ $\!\!  2^{[r]} /$ $\!\!  1^{[r]} \otimes 3^{[s]} /$ $\!\! 2^{[s]}$     \\

8 & $W(3)  \otimes 1^{[r]} \otimes 1^{[s]} \otimes 1^{[t]} /$ $\!\!  W(2)  \otimes W(2)^{[r]} \otimes W(2)^{[s]} /$ $\!\!  W(2)  \otimes W(2)^{[r]} \otimes W(2)^{[t]} /$ $\!\!  W(2)  \otimes W(2)^{[s]} \otimes W(2)^{[t]} /$  $\!\!  W(2) /$ $\!\!  1  \otimes W(3)^{[r]} \otimes 1^{[s]} \otimes 1^{[t]} /$ $\!\!  1  \otimes 1^{[r]} \otimes W(3)^{[s]} \otimes 1^{[t]} /$ $\!\!  1  \otimes 1^{[r]} \otimes 1^{[s]} \otimes W(3)^{[t]} /$ $\!\!  W(2)^{[r]} \otimes W(2)^{[s]} \otimes W(2)^{[t]} /$ $\!\!  W(2)^{[r]} /$ $\!\!  W(2)^{[s]} /$ $\!\!  W(2)^{[t]}$ \\

9 & $W(4) /$ $\!\!  (2 \otimes 2^{[r]} \otimes 2^{[s]})^2 /$ $\!\!  (2 \otimes 2^{[r]} \otimes 2^{[t]})^2 /$ $\!\!  (2 \otimes 2^{[s]} \otimes 2^{[t]})^2 /$ $\!\!  2 /$ $\!\!  W(4)^{[r]} /$ $\!\!  (2^{[r]} \otimes 2^{[s]} \otimes 2^{[t]})^2 /$ $\!\! 2^{[r]} /$ $\!\!  W(4)^{[s]} /$ $\!\!  2^{[s]} /$ $\!\!  W(4)^{[t]} /$ $\!\!  2^{[t]} $ \\

10 & $W(28) /$ $\!\!  W(26) /$ $\!\!  W(22)^2 /$ $\!\!  W(18)^2 /$ $\!\!  16 /$ $\!\!  14^3 /$ $\!\!  10^2 /$ $\!\!  8 /$ $\!\!  6 /$ $\!\!  4 /$ $\!\!  2 $ \\

11 & $W(22)^{[r]} /$ $\!\!  W(21)^{[r]} \otimes 1^{[s]} /$ $\!\!  W(18) /$ $\!\!  W(15)^{[r]} \otimes 1^{[s]} /$ $\!\!  W(14)^{[r]} /$ $\!\!  12^{[r]} \otimes 2^{[s]} /$ $\!\!  11^{[r]} \otimes 1^{[s]} /$ $\!\!  10^{[r]} /$ $\!\!  9^{[r]} \otimes 1^{[s]} /$ $\!\!  6^{[r]} /$ $\!\!  3^{[r]} \otimes 1^{[s]} /$ $\!\!  2^{[r]} /$ $\!\!  2^{[s]}$  \\

12 & $W(18)^{[r]} /$ $\!\!  W(15)^{[r]} \otimes 3^{[s]} /$ $\!\!  W(14)^{[r]} /$ $\!\!  10^{[r]} \otimes 4^{[s]} /$ $\!\!  10^{[r]} /$ $\!\!  9^{[r]} \otimes 3^{[s]} /$ $\!\!  6^{[r]} /$ $\!\!  5^{[r]} \otimes 3^{[s]} /$ $\!\!  2^{[r]} /$ $\!\!  6^{[s]} /$ $\!\!  2^{[s]}$  \\

13 & $W(18)^{[r]} /$ $\!\!  W(15)^{[r]} \otimes 1^{[s]} /$ $\!\!  W(15)^{[r]} \otimes 1^{[t]} /$ $\!\!  W(14)^{[r]} /$ $\!\!  10^{[r]} \otimes 1^{[s]} \otimes 1^{[t]} /$ $\!\!  (10^{[r]})^2 /$ $\!\!  9^{[r]} \otimes 1^{[s]} /$ $\!\!  9^{[r]} \otimes 1^{[t]} /$ $\!\!  6^{[r]} /$ $\!\!  5^{[r]} \otimes 1^{[s]} /$ $\!\!  5^{[r]} \otimes 1^{[t]} /$ $\!\!  2^{[r]} /$ $\!\!  2^{[s]} /$ $\!\!  1^{[s]} \otimes 1^{[t]} /$ $\!\!  2^{[t]}$  \\

14 & $W(14)^{[r]} /$ $\!\!  10^{[r]} \otimes 6^{[s]} /$ $\!\!  (10^{[r]})^2 /$ $\!\!  8^{[r]} \otimes 6^{[s]} /$ $\!\!  6^{[r]} /$ $\!\!  4^{[r]} \otimes 6^{[s]} /$ $\!\!  4^{[r]} /$ $\!\!  2^{[r]} /$ $\!\!  10^{[s]} /$ $\!\!  6^{[s]} /$ $\!\!  2^{[s]}$ \\

15 & $W(14)^{[r]} /$ $\!\!  10^{[r]} \otimes 1^{[s]} \otimes 1^{[t]} /$ $\!\!  10^{[r]} \otimes 1^{[s]} \otimes 1^{[u]} /$ $\!\!  10^{[r]} /$ $\!\!  8^{[r]} \otimes 2^{[s]} /$ $\!\!  8^{[r]} \otimes 1^{[t]} \otimes 1^{[u]} /$ $\!\!  6^{[r]} /$ $\!\!  4^{[r]} \otimes 1^{[s]} \otimes 1^{[t]} /$ $\!\!  4^{[r]} \otimes 1^{[s]} \otimes 1^{[u]} /$ $\!\!  2^{[r]} /$ $\!\!  2^{[s]} \otimes 1^{[t]} \otimes 1^{[u]} /$ $\!\!  2^{[s]} /$ $\!\! 2^{[t]} /$ $\!\!  2^{[u]} $  \\

16 & $W(10)^{[r]} /$ $\!\! 6^{[r]} \otimes 3^{[s]} \otimes 1^{[t]} /$ $\!\! 6^{[r]} \otimes 2^{[s]} \otimes 2^{[t]} /$ $\!\! 6^{[r]} \otimes 1^{[s]} \otimes 3^{[t]} /$ $\!\! 6^{[r]} /$ $\!\! 2^{[r]} /$ $\!\! 4^{[s]} \otimes 2^{[t]} /$ $\!\! 3^{[s]} \otimes 1^{[t]} /$ $\!\! 2^{[s]} \otimes 4^{[t]} /$ $\!\! 2^{[s]} /$ $\!\! 1^{[s]} \otimes 3^{[t]} /$ $\!\! 2^{[t]}$ \\

17 & $W(10)^{[r]} /$ $\!\!  6^{[r]} \otimes 4^{[s]} /$ $\!\!  6^{[r]} \otimes 3^{[s]} \otimes 1^{[t]} /$ $\!\!  6^{[r]} \otimes 3^{[s]} \otimes 1^{[u]} /$ $\!\!  6^{[r]} \otimes 1^{[t]} \otimes 1^{[u]} /$ $\!\!  6^{[r]} /$ $\!\!  2^{[r]} /$ $\!\!  6^{[s]} /$ $\!\!  4^{[s]} \otimes 1^{[t]} \otimes 1^{[u]} /$ $\!\!  3^{[s]} \otimes 1^{[t]} /$ $\!\!  3^{[s]} \otimes 1^{[u]} /$ $\!\!  2^{[s]} /$ $\!\!  2^{[t]} /$ $\!\!  2^{[u]}$ \\

18 & $W(10)^{[r]} /$ $\!\!  6^{[r]} \otimes 1^{[s]} \otimes 1^{[t]} /$ $\!\!  6^{[r]} \otimes 1^{[s]} \otimes 1^{[u]} /$ $\!\!  6^{[r]} \otimes 1^{[s]} \otimes 1^{[v]} /$ $\!\!  6^{[r]} \otimes 1^{[t]} \otimes 1^{[u]} /$ $\!\!  6^{[r]} \otimes 1^{[t]} \otimes 1^{[v]} /$ $\!\!  6^{[r]} \otimes 1^{[u]} \otimes 1^{[v]} /$ $\!\!  (6^{[r]})^2 /$ $\!\!  2^{[r]} /$ $\!\!  2^{[s]} /$ $\!\!  1^{[s]} \otimes 1^{[t]} \otimes 1^{[u]} \otimes 1^{[v]} /$ $\!\!  1^{[s]} \otimes 1^{[t]} /$ $\!\!  1^{[s]} \otimes 1^{[u]} /$ $\!\!  1^{[s]} \otimes 1^{[v]} /$ $\!\!  2^{[t]} /$ $\!\!  1^{[t]} \otimes 1^{[u]} /$ $\!\!  1^{[t]} \otimes 1^{[v]} /$ $\!\!  2^{[u]} /$ $\!\!  1^{[u]} \otimes 1^{[v]} /$ $\!\!  2^{[v]}$ \\

19 & $W(10)^{[r]} /$ $\!\!  6^{[r]} \otimes 2^{[s]} /$ $\!\!  (6^{[r]} \otimes 1^{[s]} \otimes 1^{[t]} \otimes 1^{[u]})^2 /$ $\!\!  6^{[r]} \otimes 2^{[t]} /$ $\!\!  6^{[r]} \otimes 2^{[u]} /$ $\!\!  6^{[r]} /$ $\!\!  2^{[r]} /$ $\!\!  2^{[s]} \otimes 2^{[t]} /$ $\!\!  2^{[s]} \otimes 2^{[u]} /$ $\!\!  2^{[s]} /$ $\!\!  (1^{[s]} \otimes 1^{[t]} \otimes 1^{[u]})^2 /$ $\!\!  2^{[t]} \otimes 2^{[u]} /$ $\!\!  2^{[t]} /$ $\!\!  2^{[u]}$ \\

20 & $W(6) /$ $\!\!  4 \otimes 4^{[r]} /$ $\!\!  4 \otimes 4^{[s]} /$ $\!\!  4 /$ $\!\!  (3 \otimes 3^{[r]} \otimes 3^{[s]})^2 /$ $\!\!  2 /$ $\!\!  W(6)^{[r]} /$ $\!\!  4^{[r]} \otimes 4^{[s]} /$ $\!\!  4^{[r]} /$ $\!\!  2^{[r]} /$ $\!\!  W(6)^{[s]} /$ $\!\!  4^{[s]} /$ $\!\!  2^{[s]}$  \\

21 & $W(6)^{[r]} /$ $\!\!  4^{[r]} \otimes 4^{[s]} /$ $\!\!  4^{[r]} \otimes 2^{[t]} /$ $\!\!  4^{[r]} \otimes 2^{[u]} /$ $\!\!  (3^{[r]} \otimes 3^{[s]} \otimes 1^{[t]} \otimes 1^{[u]})^2 /$ $\!\!  2^{[r]} /$ $\!\!  W(6)^{[s]} /$ $\!\!  4^{[s]} \otimes 2^{[t]} /$ $\!\!  4^{[s]} \otimes 2^{[u]} /$ $\!\!  2^{[s]} /$ $\!\!  2^{[t]} \otimes 2^{[u]} /$ $\!\!  2^{[t]} /$ $\!\!  2^{[u]}$ \\

22 & $W(6)^{[r]} /$ $\!\!  4^{[r]} \otimes 2^{[s]} /$ $\!\!  4^{[r]} \otimes 3^{[t]} \otimes 1^{[u]} /$ $\!\!  3^{[r]} \otimes 1^{[s]} \otimes 4^{[t]} /$ $\!\!  3^{[r]} \otimes 1^{[s]} \otimes 3^{[t]} \otimes 1^{[u]} /$ $\!\!  3^{[r]} \otimes 1^{[s]} \otimes 2^{[u]} /$ $\!\!  2^{[r]} /$ $\!\!  2^{[s]} \otimes 3^{[t]} \otimes 1^{[u]} /$ $\!\!  2^{[s]} /$ $\!\!  W(6)^{[t]} /$ $\!\!  4^{[t]} \otimes 2^{[u]} /$ $\!\!  2^{[t]} /$ $\!\!  2^{[u]}$ \\

23 & $W(6)^{[r]} /$ $\!\!  4^{[r]} \otimes 2^{[s]} /$ $\!\!  4^{[r]} \otimes 1^{[t]} \otimes 1^{[u]} /$ $\!\!  4^{[r]} \otimes 1^{[v]} \otimes 1^{[w]} /$ $\!\!  3^{[r]} \otimes 1^{[s]} \otimes 1^{[t]} \otimes 1^{[v]} /$ $\!\!   3^{[r]} \otimes 1^{[s]} \otimes 1^{[t]} \otimes 1^{[w]} /$ $\!\!   3^{[r]} \otimes 1^{[s]} \otimes 1^{[u]} \otimes 1^{[v]} /$ $\!\!   3^{[r]} \otimes 1^{[s]} \otimes 1^{[u]} \otimes 1^{[w]} /$ $\!\!  2^{[r]} /$ $\!\!  2^{[s]} \otimes 1^{[t]} \otimes 1^{[u]} /$ $\!\!  2^{[s]} \otimes 1^{[v]} \otimes 1^{[w]} /$ $\!\!  2^{[s]} /$ $\!\!  2^{[t]} /$ $\!\!  1^{[t]} \otimes 1^{[u]} \otimes 1^{[v]} \otimes 1^{[w]} /$ $\!\!  2^{[u]} /$ $\!\!  2^{[v]} /$ $\!\!  2^{[w]}$ \\

24 & $W(4)^{[r]} \otimes 2^{[s]} /$ $\!\!  W(3)^{[r]} \otimes 1^{[s]} \otimes 1^{[t]} \otimes 1^{[u]} /$ $\!\!   W(3)^{[r]} \otimes 1^{[s]} \otimes 1^{[t]} \otimes 1^{[v]} /$ $\!\!  2^{[r]} \otimes W(4)^{[s]} /$ $\!\!  2^{[r]} \otimes 2^{[s]} \otimes 2^{[t]} /$ $\!\!   2^{[r]} \otimes 2^{[s]} \otimes 1^{[u]} \otimes 1^{[v]} /$ $\!\!  2^{[r]} /$ $\!\!  1^{[r]} \otimes W(3)^{[s]} \otimes 1^{[t]} \otimes 1^{[u]} /$ $\!\!  1^{[r]} \otimes W(3)^{[s]} \otimes 1^{[t]} \otimes 1^{[v]} /$ $\!\!  2^{[s]} /$ $\!\!  2^{[t]} \otimes 1^{[u]} \otimes 1^{[v]} /$ $\!\!  2^{[t]} /$ $\!\!  2^{[u]} /$ $\!\!  2^{[v]}$ \\

25 & $2^{[r]} \otimes 2^{[s]} /$ $\!\!  2^{[r]} \otimes 2^{[t]} /$ $\!\!  2^{[r]} \otimes 2^{[u]} /$ $\!\!  2^{[r]} \otimes 1^{[v]} \otimes 1^{[w]} /$ $\!\!  2^{[r]} /$ $\!\!  (1^{[r]} \otimes 1^{[s]} \otimes 1^{[t]} \otimes 1^{[u]} \otimes 1^{[v]})^2 /$ $\!\!  (1^{[r]} \otimes 1^{[s]} \otimes 1^{[t]} \otimes 1^{[u]} \otimes 1^{[w]})^2 /$ $\!\!  2^{[s]} \otimes 2^{[t]} /$ $\!\!  2^{[s]} \otimes 2^{[u]} /$ $\!\!  2^{[s]} \otimes 1^{[v]} \otimes 1^{[w]} /$ $\!\!  2^{[s]} /$ $\!\!  2^{[t]} \otimes 2^{[u]} /$ $\!\!  2^{[t]} \otimes 1^{[v]} \otimes 1^{[w]} /$ $\!\!  2^{[t]} /$ $\!\!  2^{[u]} \otimes 1^{[v]} \otimes 1^{[w]} /$ $\!\!  2^{[u]} /$ $\!\!  2^{[v]} /$ $\!\!  2^{[w]} $  \\

26 & $W(2) /$ $\!\!  1 \otimes 1^{[r]} \otimes 1^{[s]} \otimes 1^{[t]} /$ $\!\!  1 \otimes 1^{[r]} \otimes 1^{[u]} \otimes 1^{[v]} /$ $\!\!  1 \otimes 1^{[r]} \otimes 1^{[w]} \otimes 1^{[x]} /$ $\!\!  1 \otimes 1^{[s]} \otimes 1^{[u]} \otimes 1^{[w]} /$ $\!\!  1 \otimes 1^{[s]} \otimes 1^{[v]} \otimes 1^{[x]} /$ $\!\!  1 \otimes 1^{[t]} \otimes 1^{[u]} \otimes 1^{[x]} /$ $\!\!  1 \otimes 1^{[t]} \otimes 1^{[v]} \otimes 1^{[w]} /$ $\!\!  W(2)^{[r]} /$ $\!\!  1^{[r]} \otimes 1^{[s]} \otimes 1^{[u]} \otimes 1^{[x]} /$ $\!\!  1^{[r]} \otimes 1^{[s]} \otimes 1^{[v]} \otimes 1^{[w]} /$ $\!\!  1^{[r]} \otimes 1^{[t]} \otimes 1^{[u]} \otimes 1^{[w]} /$ $\!\!  1^{[r]} \otimes 1^{[t]} \otimes 1^{[v]} \otimes 1^{[x]} /$ $\!\!  W(2)^{[s]} /$ $\!\!  1^{[s]} \otimes 1^{[t]} \otimes 1^{[u]} \otimes 1^{[v]} /$ $\!\!  1^{[s]} \otimes 1^{[t]} \otimes 1^{[w]} \otimes 1^{[x]} /$ $\!\!  W(2)^{[t]} /$ $\!\!  W(2)^{[u]} /$ $\!\!  1^{[u]} \otimes 1^{[v]} \otimes 1^{[w]} \otimes 1^{[x]} /$ $\!\!  W(2)^{[v]} /$ $\!\!  W(2)^{[w]} /$ $\!\!  W(2)^{[x]}   $ \\

27 & $2^{[r]} \otimes 2^{[s]} /$ $\!\!  2^{[r]} \otimes 2^{[t]} /$ $\!\!  2^{[r]} \otimes 1^{[u]} \otimes 1^{[v]} /$ $\!\!  2^{[r]} \otimes 1^{[w]} \otimes 1^{[x]} /$ $\!\!  (2^{[r]})^2 /$ $\!\!  1^{[r]} \otimes 1^{[s]} \otimes 1^{[t]} \otimes 1^{[u]} \otimes 1^{[w]} /$ $\!\!  1^{[r]} \otimes 1^{[s]} \otimes 1^{[t]} \otimes 1^{[v]} \otimes 1^{[x]} /$ $\!\!  1^{[r]} \otimes 1^{[s]} \otimes 1^{[t]} \otimes 1^{[u]} \otimes 1^{[x]} /$ $\!\!  1^{[r]} \otimes 1^{[s]} \otimes 1^{[t]} \otimes 1^{[v]} \otimes 1^{[w]} /$ $\!\!  2^{[s]} \otimes 2^{[t]} /$ $\!\!  2^{[s]} \otimes 1^{[u]} \otimes 1^{[v]} /$ $\!\!  2^{[s]} \otimes 1^{[w]} \otimes 1^{[x]} /$ $\!\!  (2^{[s]})^2 /$ $\!\!  2^{[t]} \otimes 1^{[u]} \otimes 1^{[v]} /$ $\!\!  2^{[t]} \otimes 1^{[w]} \otimes 1^{[x]} /$ $\!\!  (2^{[t]})^2 /$ $\!\! 2^{[u]} /$ $\!\!  1^{[u]} \otimes 1^{[v]} \otimes 1^{[w]} \otimes 1^{[x]} /$ $\!\!  (1^{[u]} \otimes 1^{[v]})^2 /$ $\!\! 2^{[v]} /$ $\!\! 2^{[w]} /$ $\!\!  (1^{[w]} \otimes 1^{[x]})^2 /$ $\!\! 2^{[x]} /$ $\!\!  0^8$ \\

28 & $2 \otimes 2^{[r]} /$ $\!\! 2 \otimes 2^{[s]} /$ $\!\! 2 \otimes 1^{[u]} \otimes 1^{[v]} \otimes 1^{[w]} /$ $\!\! 2^2 /$ $\!\! 1 \otimes 1^{[r]} \otimes 1^{[s]} \otimes 2^{[t]} /$ $\!\! 1 \otimes 1^{[r]} \otimes 1^{[s]} \otimes 1^{[t]} \otimes 1^{[u]} \otimes 1^{[v]} /$  $\!\! 1 \otimes 1^{[r]} \otimes 1^{[s]} \otimes 2^{[u]} /$ $\!\! 1 \otimes 1^{[r]} \otimes 1^{[s]} \otimes 2^{[v]} /$ $\!\! (1 \otimes 1^{[r]} \otimes 1^{[s]})^2 /$ $\!\! 2^{[r]} \otimes 2^{[s]}  /$ $\!\! (2^{[r]})^2 /$ $\!\! 2^{[r]} \otimes 1^{[t]} \otimes 1^{[u]} \otimes 1^{[v]} /$ $\!\! 2^{[s]} \otimes 1^{[t]} \otimes 1^{[u]} \otimes 1^{[v]} /$ $\!\! (2^{[s]})^2 /$ $\!\! 2^{[t]} \otimes 2^{[u]} /$ $\!\! 2^{[t]} \otimes 2^{[v]} /$ $\!\! (2^{[t]})^2 /$ $\!\! (1^{[t]} \otimes 1^{[u]} \otimes 1^{[v]})^2 /$ $\!\! 2^{[u]} \otimes 2^{[v]} /$ $\!\! (2^{[u]})^2 /$ $\!\! (2^{[v]})^2 /$ $\!\! 0^8$ \\

29 & $2^{[r]} \otimes 2^{[s]} /$ $\!\! 2^{[r]} \otimes 2^{[t]} /$ $\!\! 2^{[r]} \otimes 2^{[u]} /$ $\!\! 2^{[r]} \otimes 2^{[v]} /$ $\!\! 2^{[r]} \otimes 1^{[w]} \otimes 1^{[x]} /$ $\!\! (2^{[r]})^2 /$ $\!\! 1^{[r]} \otimes 1^{[s]} \otimes 1^{[t]} \otimes 1^{[u]} \otimes 1^{[v]} \otimes 1^{[w]} /$ $\!\! 1^{[r]} \otimes 1^{[s]} \otimes 1^{[t]} \otimes 1^{[u]} \otimes 1^{[v]} \otimes 1^{[x]} /$ $\!\! 2^{[s]} \otimes 2^{[t]} /$ $\!\! 2^{[s]} \otimes 2^{[u]} /$ $\!\! 2^{[s]} \otimes 2^{[v]} /$ $\!\! 2^{[s]} \otimes 1^{[w]} \otimes 1^{[x]} /$ $\!\! (2^{[s]})^2 /$ $\!\!  2^{[t]} \otimes 2^{[u]} /$ $\!\! 2^{[t]} \otimes 2^{[v]} /$ $\!\! 2^{[t]} \otimes 1^{[w]} \otimes 1^{[x]} /$ $\!\! (2^{[t]})^2 /$ $\!\! 2^{[u]} \otimes 2^{[v]} /$ $\!\! 2^{[u]} \otimes 1^{[w]} \otimes 1^{[x]} /$ $\!\! (2^{[u]})^2 /$ $\!\! 2^{[v]} \otimes 1^{[w]} \otimes 1^{[x]} /$ $\!\! (2^{[v]})^2 /$ $\!\! 2^{[w]} /$ $\!\! (1^{[w]} \otimes 1^{[x]})^2 /$ $\!\! 2^{[x]} /$ $\!\! 0^8$ \\

30 & $2 \otimes 2^{[r]} /$ $\!\! 2 \otimes 2^{[s]} /$ $\!\! 2 \otimes 2^{[t]} /$ $\!\! 2 \otimes 2^{[u]} /$ $\!\! 2 \otimes 2^{[v]} /$ $\!\! 2 \otimes 2^{[w]} /$ $\!\! 2^2 /$ $\!\! 1 \otimes 1^{[r]} \otimes 1^{[s]} \otimes 1^{[t]} \otimes 1^{[u]} \otimes 1^{[v]} \otimes 1^{[w]} /$ $\!\! 2^{[r]} \otimes 2^{[s]} /$ $\!\! 2^{[r]} \otimes 2^{[t]} /$ $\!\! 2^{[r]} \otimes 2^{[u]} /$ $\!\! 2^{[r]} \otimes 2^{[v]} /$ $\!\! 2^{[r]} \otimes 2^{[w]} /$ $\!\! (2^{[r]})^2 /$ $\!\! 2^{[s]} \otimes 2^{[t]} /$ $\!\! 2^{[s]} \otimes 2^{[u]} /$ $\!\! 2^{[s]} \otimes 2^{[v]} /$ $\!\! 2^{[s]} \otimes 2^{[w]} /$ $\!\! (2^{[s]})^2 /$ $\!\!  2^{[t]} \otimes 2^{[u]} /$ $\!\! 2^{[t]} \otimes 2^{[v]} /$ $\!\! 2^{[t]} \otimes 2^{[w]} /$ $\!\! (2^{[t]})^2 /$ $\!\! 2^{[u]} \otimes 2^{[v]} /$ $\!\! 2^{[u]} \otimes 2^{[w]} /$ $\!\! (2^{[u]})^2 /$ $\!\! 2^{[v]} \otimes 2^{[w]} /$ $\!\! (2^{[v]})^2 /$ $\!\! (2^{[w]})^2 /$ $\!\! 0^8$ \\

31 & $2^{[r]} /$ $\!\!1^{[r]} \otimes 6^{[s]} \otimes 5^{[t]} /$ $\!\! 1^{[r]} \otimes W(9)^{[t]} /$ $\!\! 1^{[r]} \otimes 3^{[t]} /$ $\!\! W(10)^{[s]} /$  $\!\! 6^{[s]} \otimes W(8)^{[t]} /$ $\!\! 6^{[s]} \otimes 4^{[t]} /$ $\!\! 2^{[s]} /$ $\!\! W(10)^{[t]} /$ $\!\! 6^{[t]} /$ $\!\! 2^{[t]}$ \\

32 & $2^{[r]} /$ $\!\! 1^{[r]} \otimes 6^{[s]} \otimes 2^{[t]} \otimes 1^{[u]} /$ $\!\!1^{[r]} \otimes  4^{[t]} \otimes 1^{[u]} /$ $\!\! 1^{[r]} \otimes  3^{[u]} /$ $\!\! W(10)^{[s]} /$ $\!\! 6^{[s]} \otimes 4^{[t]} /$ $\!\! 6^{[s]} \otimes 2^{[t]} \otimes 2^{[u]} /$  $\!\! 2^{[s]} /$ $\!\! 4^{[t]} \otimes 2^{[u]} /$ $\!\! 2^{[t]} /$ $\!\! 2^{[u]}$     \\

33 & $2^{[r]} /$ $\!\!1^{[r]} \otimes  3^{[s]} \otimes 6^{[t]} /$  $\!\! 1^{[r]} \otimes 1^{[s]} \otimes W(10)^{[t]} /$ $\!\!1^{[r]} \otimes 1^{[s]} \otimes 2^{[t]} /$  $\!\! 4^{[s]} \otimes 6^{[t]}/ $ $\!\! 2^{[s]} \otimes W(12)^{[t]} /$  $\!\! 2^{[s]} \otimes W(8)^{[t]} /$ $\!\! 2^{[s]} \otimes 4^{[t]} /$  $\!\! 2^{[s]} /$ $\!\! W(10)^{[t]} /$ $\!\! 2^{[t]}$ \\

34 & $2^{[r]} /$  $\!\! 1^{[r]} \otimes 3^{[s]} /$ $\!\!1^{[r]} \otimes 1^{[s]} \otimes W(16)^{[t]} /$ $\!\!1^{[r]} \otimes  1^{[s]} \otimes 8^{[t]} /$ $\!\! 2^{[s]} \otimes W(16)^{[t]} /$ $\!\! 2^{[s]} \otimes 8^{[t]} /$  $\!\! 2^{[s]} /$ $\!\! W(22)^{[t]} /$ $\!\! W(14)^{[t]} /$ $\!\! 10^{[t]} /$ $\!\! 2^{[t]}$   \\

35 & $2^{[r]} /$ $\!\! 1^{[r]} \otimes W(6)^{[s]} \otimes 3^{[t]} /$ $\!\! 1^{[r]} \otimes 4^{[s]} \otimes 1^{[t]} /$ $\!\! 1^{[r]} \otimes 2^{[s]} \otimes W(5)^{[t]} /$ $\!\! W(6)^{[s]} \otimes 4^{[t]} /$ $\!\! 4^{[s]} \otimes W(6)^{[t]} /$ $\!\! 4^{[s]} \otimes 2^{[t]} /$ $\!\! 2^{[s]} \otimes W(8)^{[t]} /$  $\!\! 2^{[s]} \otimes 4^{[t]} /$ $\!\! 2^{[s]} /$ $\!\! 2^{[t]}$ \\

36 & $2^{[r]} /$ $\!\! 1^{[r]} \otimes W(21)^{[s]} /$ $\!\! 1^{[r]} \otimes 15^{[s]} /$ $\!\! 1^{[r]} \otimes 11^{[s]} /$ $\!\! 1^{[r]} \otimes 5^{[s]} /$ $\!\! W(26)^{[s]} /$ $\!\! W(22)^{[s]} /$ $\!\! W(18)^{[s]} /$ $\!\! 16^{[s]} /$ $\!\! 14^{[s]} /$ $\!\! (10^{[s]})^2 /$ $\!\! 6^{[s]} /$ $\!\! 2^{[s]} $ \\

37 & $2^{[r]} /$ $\!\! 1^{[r]} \otimes W(27)^{[s]} /$ $\!\! 1^{[r]} \otimes  17^{[s]} /$ $\!\! 1^{[r]} \otimes  9^{[s]} /$ $\!\! W(34)^{[s]} /$ $\!\! W(26)^{[s]} /$ $\!\! W(22)^{[s]} /$ $\!\! 18^{[s]} /$ $\!\! 14^{[s]} /$ $\!\! 10^{[s]} /$  $\!\! 2^{[s]}$   \\

38 & $10^{[r]} /$ $\!\! 6^{[r]} \otimes W(16)^{[s]} /$ $\!\! 6^{[r]} \otimes 8^{[s]} /$ $\!\! 2^{[r]} /$ $\!\! W(22)^{[s]} /$ $\!\! W(14)^{[s]} /$ $\!\! 10^{[s]} /$ $\!\! 2^{[s]}$   \\

39 & $W(10)^{[r]} /$  $\!\! 6^{[r]} \otimes 4^{[s]} /$ $\!\! 6^{[r]} \otimes 2^{[s]} \otimes 6^{[t]} /$ $\!\! 2^{[r]} /$ $\!\! 4^{[s]} \otimes 6^{[t]} /$ $\!\! 2^{[s]} /$ $\!\! W(10)^{[t]} /$ $\!\! 2^{[t]}$   \\

40 & $W(38) /$ $\!\! W(34) /$ $\!\! W(28) /$ $\!\! W(26) /$ $\!\! 22^2 /$ $\!\! 18 /$ $\!\! 16 /$ $\!\! 14 /$ $\!\! 10 /$ $\!\! 6 /$ $\!\! 2 $ \\

41 & $W(46) /$ $\!\! W(38) /$ $\!\! W(34) /$ $\!\! 28 /$ $\!\! 26 /$ $\!\! 22 /$ $\!\!18 /$ $\!\! 14 /$ $\!\!10 /$ $\!\! 2$ \\

42 & $W(58) /$ $\!\! W(46) /$ $\!\! W(38) /$ $\!\! W(34) /$ $\!\! 26 /$ $\!\! 22 /$ $\!\! 14 /$ $\!\! 2$ \\

\hline

\end{longtable}

\section{Conditions for conjugacy class representatives of $G$-irreducible $A_1$ subgroups of $G = E_7$ and $E_8$}

In this section we present the tables referred to within Table \ref{E7tab} and \ref{E8tab}. They give the extra restrictions on the field twists in certain embeddings of $M$-irreducible $A_1$ subgroups of $M = A_1 D_6$ when $G = E_7$ and $M = D_8$ when $G = E_8$, namely $E_7(\#12)$, $E_8(\#23)$, $E_8(\#26)$ and $E_8(\#27)$. These restrictions ensure there is no repetition of conjugacy classes and further, that each conjugacy class is $G$-irreducible. The restrictions are given in rows of the tables: the first column lists all equalities amongst a subset of the field twists; the second column lists any further requirements. So an ordered set $\{0, r, \ldots \}$ is permitted if it satisfies the conditions in the first and second column of a row of the table (and the set of field twists satisfying each row are mutually exclusive).

We give an example. Consider $X = E_7(\#12)$. Then $X$ is a diagonal subgroup of $A_1^7 < A_1 D_6$ via $(1,1^{[r]}, 1^{[s]}, 1^{[t]}, 1^{[u]}, 1^{[v]}, 1^{[w]})$. Table \ref{conditions1} gives the conditions that an ordered set $0, r, s, t, u , v, w$ needs to satisfy. So $0, r, \ldots, w$ satisfy the conditions of the first row if: $0, r, \ldots, w$ are all distinct;  $r$ is the smallest integer of $r, s, \ldots, w$; and $t$ is the smallest integer of $t, u, v, w$. Similarly, $0, r, \ldots, w$ satisfy the conditions of the ninth row if: $r=0$; $s=t$; $0,s,u,v,w$ are distinct; and $u < v$. 

\begin{longtable}{>{\raggedright\arraybackslash}p{0.4\textwidth-2\tabcolsep}>{\raggedright\arraybackslash}p{0.43\textwidth-\tabcolsep}@{}}
\caption{Conditions on field twists for $E_7(\#12)$} \label{conditions1} \\

\hline \noalign{\smallskip}

All equalities among $0, r, \ldots, w$ & Further requirements on $0, r, \ldots, w$ \\

\hline 

none & $r < \text{min}\{s,t,u,v,w\}$ and $t<\text{min}\{u,v,w\}$ \\

$r=s$ & $t<\text{min}\{u,v,w\}$ and $v<w$ \\

$r=t$ & $s < v$ \\

$r=s=t$ & none \\

$r=t=w$ & $s < u,v$ \\

$r=s=t=v$ & $u < w$ \\

$r=t$ and $u=v=w$ & none \\

$r = 0$ & $t < \text{min}\{u,v,w\}$ and $v < w$ \\

$r=0$ and $s=t$ & $u < v$ \\

$r=0$ and $t=u$ & $v < w$ \\

$r=0$ and $s=t=u$ & $v < w$ \\

$r=0$ and $s=t$ and $u=w$ & none \\

$r=s=0$ & $t < u < v < w$ \\

$r=t=0$ & $s < u < v$ \\

$r=t=0$ and $s=w$ & $u<v$ \\

$r=s=t=0$ & $u < v < w$ \\

\end{longtable}

In the following table, note that the first column lists all equalities occurring among the elements $r, s$ and all equalities among the elements $t, u, v, w$ only. So $r$ is still permitted to be equal to $t$, for example, in all of the conditions in the first column (this is ruled out by the conditions in the second column for the third and fourth rows).   

\begin{longtable}{>{\raggedright\arraybackslash}p{0.5\textwidth-2\tabcolsep}>{\raggedright\arraybackslash}p{0.5\textwidth-\tabcolsep}@{}}
\caption{Conditions on field twists for $E_8(\#23)$} \label{23conditions} \\

\hline \noalign{\smallskip}

All equalities among $r, s$ and all equalities among $t, u, v, w$ & Further requirements on $r, \ldots, w$ \\

\hline \noalign{\smallskip}

none & $t < u$ and $t < v$ and $v < w$ \\

$r=s$ & $t < u < v < w$ \\

$r=s$ and $t=u$ & $r < t$ and $v < w$ \\

$r = s$ and $t = u = v$ & $r \neq t$ \\

$t=u$ & $s < t$ and $v < w$ \\

$t=v$ & $u < w$ \\

$t=u=v$ & none \\

\end{longtable}

\begin{longtable}{>{\raggedright\arraybackslash}p{0.5\textwidth-2\tabcolsep}>{\raggedright\arraybackslash}p{0.5\textwidth-\tabcolsep}@{}}
\caption{Conditions on field twists for $E_8(\#26)$} \label{8conditions1} \\

\hline \noalign{\smallskip}

All equalities among $0, r, \ldots, w$ & Further requirements on $0, r, \ldots, w$ \\

\hline \noalign{\smallskip}

none & $r < s$ and $s < \text{min}\{t, u\}$ and $u < \text{min}\{v, w, x\}$ \\

$r=s$ & $u<v, w, x$ and $w<x$ \\

$r=s=t$ & $u < v < w < x$ \\

$r=s=u$ & $t < v < w$ \\

$r=s=t=u$ & $v < w < x$ \\

$r=s$ and $t=u$ & $v < w$ \\

$r=s$ and $u=v$ & $w < x$ \\

$r=s$ and $t=u=v$ & $w < x$ \\

$r=s$ and $t=u$ and $v = w$ & none \\

$r = 0$ & $s < \text{min}\{t,u\}$ and $u < \text{min}\{v, w\}$ and $w < x$ \\

$r = 0$ and $s=u$ & $t<v$ and $w<x$ \\

$r = 0$ and $s=u=w$ & $t < v < x$ \\

$r = 0$ and $s=u$ and $t=w$ & none \\

$r = 0$ and $s = u$ and $t = w$ and $v = x$ & none \\

$r = s = 0$ & $u < v < w< x$ \\

$r=s=0$ and $t=u$ & $v < w < x$ \\

$r = s = u = 0$ & $t < v < w < x$ \\

\end{longtable}

\begin{longtable}{>{\raggedright\arraybackslash}p{0.35\textwidth-2\tabcolsep}>{\raggedright\arraybackslash}p{0.45\textwidth-\tabcolsep}@{}}
\caption{Conditions on field twists for $E_8(\#27)$} \label{27conditions} \\

\hline \noalign{\smallskip}

All equalities amongst $u, v, w, x$ & Further requirements on $r, \ldots, x$ \\

\hline \noalign{\smallskip}

none & $r < s < t$ and $u < v$ and $u < w$ and $w < x$ \\

$u=w$ & $r < s < t$ and $v < x$ \\

\end{longtable}

\subsection*{Acknowledgements} 

This paper was partially prepared as part of the author's Ph.D. qualification under the supervision of \linebreak Prof.\ M.\ Liebeck, with financial support from the EPSRC. The author would like to thank \linebreak Prof.\ M.\ Liebeck for suggesting the problem and his help in solving it. He would also like to thank Dr\ D.\ Stewart and Dr\ A.\ Litterick for many valuable conversations along the way.  Finally, he would like to thank the anonymous referee for their helpful corrections and suggestions. 

\begin{appendices} 
\section{Levi subgroups} \label{app}

Let $G$ be an exceptional algebraic group over an algebraically closed field. This section contains tables of composition factors for Levi subgroups of $G$ on the minimal and adjoint modules for $G$. If $L'$ is simple then these are found in \cite[Tables 8.1--8.7]{LS3}. If $L'$ is not simple then the composition factors are deduced from those of a maximal subsystem subgroup containing $L'$.

\begin{longtable}{p{0.1\textwidth - 2\tabcolsep}>{\raggedright\arraybackslash}p{0.3\textwidth-2\tabcolsep}>{\raggedright\arraybackslash}p{0.6\textwidth-\tabcolsep}@{}}

\caption{The composition factors for the action of Levi subgroups of $G_2$ on $V_{7}$ and $L(G_2)$. \label{levig2}} \\

\hline \noalign{\smallskip}

Levi $L'$ & Comp. factors of $V_{7} \downarrow L'$ & Comp. factors of $L(G_2) \downarrow L'$ \\

\hline \noalign{\smallskip}

$A_1$ & $1^2 /$ $\!\! 0^{3}$ & $W(2) /$ $\!\! 1^{4} /$ $\!\! 0^{3}$ \\

$\tilde{A}_1$ & $W(2) /$ $\!\! 1^2$ & $W(3)^2 /$ $\!\! W(2) /$ $\!\! 0^{3}$ \\

\hline

\end{longtable}

\begin{longtable}{p{0.1\textwidth - 2\tabcolsep}>{\raggedright\arraybackslash}p{0.3\textwidth-2\tabcolsep}>{\raggedright\arraybackslash}p{0.6\textwidth-\tabcolsep}@{}}

\caption{The composition factors for the action of Levi subgroups of $F_4$ on $V_{26}$ and $L(F_4)$. \label{levif4}} \\

\hline \noalign{\smallskip}

Levi $L'$ & Comp. factors of $V_{26} \downarrow L'$ & Comp. factors of $L(F_4) \downarrow L'$ \\

\hline \noalign{\smallskip}

$B_3$ & $W(100) /$ $\!\! 001^2 /$ $\!\! 000^3$ & $W(100)^2 /$ $\!\! W(010) /$ $\!\! 001^2 /$ $\!\! 000$ \\

$B_2$ & $W(10) /$ $\!\! 01^4 /$ $\!\! 00^5$ & $W(10)^4 /$ $\!\! W(02) /$ $\!\! 01^4 /$ $\!\! 00^6$ \\

$C_3$ & $100^2 /$ $\!\! W(010)$ & $W(200) /$ $\!\! W(001)^2 /$ $\!\! 000^3$ \\

$A_2 \tilde{A}_1$ & $(10,1) /$ $\!\! (10,0) /$ $\!\! (01,1) /$ $\!\! (01,0) /$ $\!\! (00,W(2)) /$ $\!\! (00,1)^2 /$ $\!\! (00,0)$ & $(W(11),0) /$ $\!\! (10,W(2)) /$ $\!\! (10,1) /$ $\!\! (10,0) /$ $\!\! (01,W(2)) /$ $\!\! (01,1) /$ $\!\! (01,0) /$ $\!\! (00,W(2)) /$ $\!\! (00,1)^2 /$ $\!\! (00,0)$  \\

$\tilde{A}_2 A_1$ & $(10,1) /$ $\!\! (10,0) /$ $\!\! (01,1) /$ $\!\! (01,0) /$ $\!\! (W(11),0)$ & $(W(20),1) /$ $\!\! (W(20),0) /$ $\!\! (W(11),0) /$ $\!\! (W(02),1) /$ $\!\! (W(02),0) /$ $\!\! (00,W(2)) /$ $\!\! (00,1)^2 /$ $\!\! (00,0)$ \\

$A_2$ & $10^3 /$ $\!\! 01^3 /$ $\!\! 00^8$ & $W(11) /$ $\!\! 10^6 /$ $\!\! 01^6 /$ $\!\! 00^8$ \\

$\tilde{A}_2$ & $10^3 /$ $\!\! 01^3 /$ $\!\! W(11)$ & $W(20)^3 /$ $\!\! W(11) /$ $\!\! W(02)^3 /$ $\!\! 00^8$ \\

$A_1 \tilde{A}_1$ & $(1,1)^2 /$ $\!\! (1,0)^2 /$ $\!\! (0,W(2)) /$ $\!\! (0,1)^4 /$  $\!\! (0,0)^3$ & $(W(2),0) /$ $\!\! (1,W(2))^2 /$ $\!\! (1,1)^2 /$ $\!\! (1,0)^4 /$ $\!\! (0,W(2))^3 /$  $\!\! (0,1)^4 /$ $\!\! (0,0)^4$ \\

$A_1$ & $1^6 /$ $\!\! 0^{14}$ & $W(2) /$ $\!\! 1^{14} /$ $\!\! 0^{21}$ \\

$\tilde{A}_1$ & $W(2) /$ $\!\! 1^8 /$ $\!\! 0^7$ & $W(2)^7 /$ $\!\! 1^8 /$ $\!\! 0^{15}$ \\

\hline

\end{longtable}

\begin{longtable}{p{0.1\textwidth - 2\tabcolsep}>{\raggedright\arraybackslash}p{0.3\textwidth-2\tabcolsep}>{\raggedright\arraybackslash}p{0.6\textwidth-\tabcolsep}@{}}

\caption{The composition factors for the action of Levi subgroups of $E_6$ on $V_{27}$ and $L(E_6)$. \label{levie6}} \\

\hline \noalign{\smallskip}

Levi $L'$ & Comp. factors of $V_{27} \downarrow L'$ & Comp. factors of $L(E_6) \downarrow L'$ \\

\hline \noalign{\smallskip}

$D_5$ & $\lambda_1 /$ $\!\! \lambda_4 /$ $\!\! 0 $ & $W(\lambda_2) /$ $\!\! \lambda_4 /$ $\!\! \lambda_5 /$ $\!\! 0$ \\

$D_4$ & $1000 /$ $\!\! 0010 /$ $\!\! 0001 /$ $\!\! 0000^3$  & $1000^2 /$ $\!\! W(0100) /$  $\!\! 0010^2 /$ $\!\! 0001^2 /$ $\!\! 0000^3$ \\

$A_5$ & $\lambda_1^2 /$ $\!\! \lambda_4$ & $W(\lambda_1 + \lambda_5) /$ $\!\! \lambda_3^2 /$ $\!\! 0^3$ \\

$A_1 A_4$ & $(1,1000) /$ $\!\! (1,0000) /$ $\!\! (0,0010) /$ $\!\! (0,0001) $ & $(W(2),0000) /$ $\!\! (1,0100) /$ $\!\! (1,0010) /$ $\!\! (0,W(1001)) /$ $\!\! (0,1000) /$  $\!\! (0,0001) /$ $\!\! (0,0000)$ \\

$A_1 A_2^2$ & $(1,01,00) /$  $\!\! (1,00,10) /$ $\!\! (0,10,01) /$ $\!\! (0,01,00) /$ $\!\! (0,00,10)$  & $(W(2),00,00) /$ $\!\! (1,10,10) /$ $\!\! (1,01,01) /$ $\!\! (1,00,00)^2 /$ $\!\! (0,W(11),00) /$ $\!\! (0,10,10) /$  $\!\! (0,01,01) /$  $\!\! (0,00,W(11)) /$    $\!\! (0,00,00)$ \\

$A_4$ & $1000^2 /$ $\!\! 0010 /$ $\!\! 0001 /$ $\!\! 0000^2$ & $W(1001) /$ $\!\! 1000 /$ $\!\! 0100^2 /$ $\!\! 0010^2 /$ $\!\! 0001 /$ $\!\! 0000^4$ \\

$A_1 A_3$ & $(1,100) /$ $\!\! (1,000)^2 /$ $\!\! (0,010) /$ $\!\! (0,001)^2 /$ $\!\! (0,000)$ & $(W(2),000) /$ $\!\! (1,100) /$ $\!\! (1,010)^2 /$ $\!\! (1,001) /$ $\!\! (0,W(101)) /$  $\!\! (0,100)^2 /$ $\!\! (0,001)^2 /$   $\!\! (0,000)^4$ \\

$A_2^2$ & $(10,01) /$ $\!\! (01,00)^3 /$ $\!\! (00,10)^3$  & $(W(11),00) /$ $\!\! (10,10)^3 /$ $\!\! (01,01)^3 /$ $\!\! (00,W(11)) /$ $\!\! (00,00)^8$ \\

$A_1^2 A_2$ & $(1,1,00) /$ $\!\! (1,0,10) /$ $\!\! (1,0,00) /$ $\!\! (0,1,01) /$ $\!\! (0,1,00) /$ $\!\! (0,0,10) /$ $\!\! (0,0,01) /$ $\!\! (0,0,00)$ & $(W(2),0,00) /$ $\!\! (1,1,10) /$ $\!\! (1,1,01) /$  $\!\! (1,0,10) /$ $\!\! (1,0,01) /$ $\!\! (1,0,00)^2 /$ $\!\! (0,W(2),00) /$ $\!\! (0,1,10) /$ $\!\! (0,1,01) /$    $\!\! (0,1,00)^2 /$ $\!\! (0,0,W(11)) /$   $\!\! (0,0,10) /$  $\!\! (0,0,01) /$  $\!\! (0,0,00)^2$ \\ 

$A_3$ & $100^2 /$ $\!\! 010 /$ $\!\! 001^2 /$ $\!\! 000^5$ & $W(101) /$ $\!\! 100^4 /$ $\!\! 010^4 /$ $\!\! 001^4 /$ $\!\! 000^7$ \\

$A_1 A_2$ & $(0,10)^3 /$ $\!\! (1,00)^3 /$ $\!\! (1,01) /$ $\!\! (0,01) /$ $\!\! (0,00)^3$ & $ (W(2),00) /$ $\!\! (0,W(11)) /$ $\!\! (1,10)^3 /$ $\!\! (1,01)^3 /$  $\!\! (1,00)^2 /$ $\!\! (0,10)^3 /$  $\!\! (0,01)^3 /$ $\!\! (0,00)^9$ \\

$A_1^3$ & $(1,1,0) /$ $\!\! (1,0,1) /$ $\!\! (1,0,0)^2 /$ $\!\! (0,1,1) /$  $\!\! (0,1,0)^2 /$ $\!\! (0,0,1)^2 /$ $\!\! (0,0,0)^3$ & $(W(2),0,0) /$ $\!\! (1,1,1)^2 /$ $\!\! (1,1,0)^2 /$ $\!\! (1,0,1)^2 /$ $\!\! (1,0,0)^4 /$ $\!\! (0,W(2),0) /$ $\!\! (0,1,1)^2 /$  $\!\! (0,1,0)^4 /$ $\!\! (0,0,W(2)) /$   $\!\! (0,0,1)^4 /$ $\!\! (0,0,0)^5$ \\ 

$A_2$ & $10^3 /$ $\!\! 01^3 /$ $\!\! 00^9$ & $W(11) /$ $\!\! 10^9 /$ $\!\! 01^9 /$ $\!\! 00^{16}$ \\

$A_1^2$ & $(1,1) /$ $\!\! (1,0)^4 /$ $\!\! (0,1)^4 /$ $\!\! (0,0)^7$ & $(W(2),0) /$ $\!\! (1,1)^6 /$ $\!\! (1,0)^8 /$ $\!\! (0,W(2)) /$  $\!\! (0,1)^8 /$ $\!\! (0,0)^{16}$ \\ 

$A_1$ & $1^6 /$ $\!\! 0^{15}$ & $W(2) /$ $\!\! 1^{20} /$ $\!\! 0^{35}$ \\

\hline

\end{longtable}

\begin{longtable}{p{0.1\textwidth - 2\tabcolsep}>{\raggedright\arraybackslash}p{0.38\textwidth-2\tabcolsep}>{\raggedright\arraybackslash}p{0.52\textwidth-\tabcolsep}@{}}

\caption{The composition factors for the action of Levi subgroups of $E_7$ on $V_{56}$ and $L(E_7)$. \label{levie7}} \\

\hline \noalign{\smallskip}

Levi $L'$ & Comp. factors of $V_{56} \downarrow L'$ & Comp. factors of $L(E_7) \downarrow L'$ \\

\hline \noalign{\smallskip}

$E_6$ & $\lambda_1 /$ $\!\! \lambda_6 /$ $\!\! 0^2$ & $\lambda_1 /$ $\!\! W(\lambda_2) /$  $\!\! \lambda_6 /$ $\!\! 0$  \\

$D_6$ & $\lambda_1^2 /$ $\!\! \lambda_5$ & $W(\lambda_2) /$ $\!\! \lambda_6^2 /$ $\!\! 0^3$ \\

$A_1 D_5$ & $(1, \lambda_1) /$ $\!\! (1,0)^2 /$ $\!\! (0,\lambda_4) /$ $\!\! (0,\lambda_5)$ & $(W(2),0) /$ $\!\! (1,\lambda_4) /$ $\!\! (1,\lambda_5) /$ $\!\! (0,\lambda_1)^2 /$ $\!\! (0,W(\lambda_2) /$ $\!\! (0,0)$  \\

$D_5$ & $\lambda_1^2 /$ $\!\! \lambda_4 /$ $\!\! \lambda_5 /$ $\!\! 0^4$ & $\lambda_1^2 /$ $\!\! W(\lambda_2) /$  $\!\! \lambda_4^2 /$ $\!\! \lambda_5^2 /$ $\!\! 0^4$  \\

$A_1 D_4$ & $(1,1000) /$ $\!\! (1,0000)^4 /$ $\!\! (0,0010)^2 /$ $\!\! (0,0001)^2$ & $(W(2),0000) /$ $\!\! (1,0010)^2 /$ $\!\! (1,0001)^2 /$ $\!\! (0,1000)^4 /$ $\!\! (0,W(0100)) /$ $\!\! (0,0000)^6$ \\

$D_4$ & $1000^2 /$ $\!\! 0010^2 /$ $\!\! 0001^2 /$ $\!\! 0000^8$ & $1000^4 /$ $\!\! W(0100) /$  $\!\! 0010^4 /$ $\!\! 0001^4 /$ $\!\! 0000^9$ \\

$A_6$ & $\lambda_1 /$ $\!\! \lambda_2 /$ $\!\! \lambda_5 /$ $\!\! \lambda_6$ & $W(\lambda_1 + \lambda_6) /$ $\!\! \lambda_1 /$ $\!\! \lambda_3 /$ $\!\! \lambda_4 /$ $\!\! \lambda_6 /$ $\!\! 0$  \\

$A_1 A_5$ & $(1,\lambda_1) /$ $\!\! (1,\lambda_5) /$ $\!\! (0,\lambda_1) /$ $\!\! (0,\lambda_3) /$ $\!\! (0,\lambda_5)$ & $(W(2),0) /$ $\!\! (1,\lambda_2) /$ $\!\! (1,\lambda_4) /$ $\!\! (1,0)^2 /$ $\!\! (0,W(\lambda_1 + \lambda_5)) /$  $\!\! (0,\lambda_2) /$  $\!\! (0,\lambda_4) /$ $\!\! (0,0)$ \\

$A_2 A_4$ & $(10,1000) /$ $\!\! (10,0000) /$ $\!\! (01,0001) /$ $\!\! (01,0000) /$ $\!\! (00,0100) /$ $\!\! (00,0010)$ & $(W(11),0000) /$ $\!\! (10,0010) /$ $\!\! (10,0001) /$$\!\! (01,1000) /$ $\!\! (01,0100) /$  $(00,W(1001)) /$ $\!\! (00,1000) /$ $\!\! (00,0001) /$  $\!\! (00,0000)$  \\

$A_1 A_2 A_3$ & $(1,10,000) /$ $\!\! (1,01,000) /$ $\!\! (1,00,010) /$ $\!\! (0,10,100) /$ $\!\! (0,01,001) /$ $\!\! (0,00,100) /$  $\!\! (0,00,001)$ & $(W(2),00,000) /$ $\!\! (1,10,001) /$ $\!\! (1,01,100) /$ $\!\! (1,00,100) /$ $\!\! (1,00,001) /$ $\!\! (0,W(11),000) /$  $\!\! (0,10,010) /$ $\!\! (0,10,000) /$ $\!\! (0,01,010) /$  $\!\! (0,01,000) /$ $\!\! (0,00,W(101)) /$  $\!\! (0,00,000)$ \\

$A_5$ & $\lambda_1^3 /$ $\!\! \lambda_3 /$ $\!\! \lambda_5^3$  & $W(\lambda_1 + \lambda_5) /$ $\!\! \lambda_2^3 /$ $\!\! \lambda_4^3 /$ $\!\! 0^8$   \\

$A_5'$ & $\lambda_1^2 /$ $\!\! \lambda_2 /$ $\!\! \lambda_4 /$ $\!\! \lambda_5^2 /$ $\!\! 0^2$ & $W(\lambda_1 + \lambda_5) /$ $\!\! \lambda_1^2 /$ $\!\! \lambda_2 /$ $\!\! \lambda_3^2 /$ $\!\! \lambda_4 /$ $\!\! \lambda_5^2 /$ $\!\! 0^4$ \\

$A_1 A_4$ & $(1,1000) /$ $\!\! (1,0000)^2 /$ $\!\! (1,0001) /$ $\!\! (0,1000) /$  $\!\! (00,0100) /$ $\!\! (00,0010) /$ $\!\! (0,0001) /$  $\!\! (0,0000)^2$ & $(W(2),0000) /$ $\!\! (1,0000)^2 /$ $\!\! (1,1000) /$  $\!\! (1,0100) /$  $\!\! (1,0010) /$ $\!\! (1,0001) /$ $\!\! (0,W(1001)) /$ $\!\! (0,1000)^2 /$ $\!\! (0,0100) /$ $\!\! (0,0010) /$ $\!\! (0,0001)^2 /$  $\!\! (0,0000)^2$  \\

$A_2 A_3$ & $(10,100) /$ $\!\! (10,000)^2 /$ $\!\! (01,001) /$ $\!\! (01,000)^2 /$ $\!\! (00,100) /$ $\!\! (00,010)^2 /$  $\!\! (00,001)$ & $(W(11),000) /$  $\!\! (10,010) /$ $\!\! (10,001)^2 /$ $\!\! (10,000) /$ $\!\! (01,100)^2 /$ $\!\! (01,010) /$  $\!\! (01,000) /$ $\!\! (00,W(101)) /$ $\!\! (00,100)^2 /$ $\!\! (00,001)^2 /$  $\!\! (00,000)^4$ \\

$A_1^2 A_3$ & $(1,1,000)^2 /$ $\!\! (1,0,010) /$ $\!\! (1,0,000)^2 /$  $\!\! (0,1,100) /$  $\!\! (0,1,001) /$ $\!\! (0,0,100)^2 /$ $\!\! (0,0,001)^2$ & $(W(2),0,000) /$ $\!\! (1,1,100) /$ $\!\! (1,1,001) /$ $\!\! (1,0,100)^2 /$ $\!\! (1,0,001)^2 /$ $\!\! (0,W(2),000) /$   $\!\! (0,1,010)^2 /$ $\!\! (0,1,000)^4 /$ $\!\! (0,0,W(101)) /$ $\!\! (0,0,010)^2 /$  $\!\! (0,0,000)^4$ \\

$A_4$ & $ 1000^3 /$ $\!\! 0100 /$ $\!\! 0010 /$ $\!\! 0001^3 /$ $\!\! 0000^6 $ & $W(1001) /$ $\!\! 1000^4 /$ $\!\! 0100^3 /$ $\!\! 0010^3 /$ $\!\! 0001^4 /$ $\!\! 0000^9$  \\

$A_1 A_3$ & $(1,010) /$  $\!\! (1,000)^6 /$  $\!\! (0,100)^4 /$ $\!\! (0,001)^4$ & $(W(2),000) /$ $\!\! (0,W(101)) /$ $\!\! (1,100)^4 /$ $\!\! (1,001)^4 /$ $\!\! (0,010)^6 /$ $\!\! (0,000)^{15} $ \\

$(A_1 A_3)'$ & $(1,100) /$ $\!\! (1,001) /$  $\!\! (1,000)^4 /$ $\!\! (0,100)^2 /$ $\!\! (0,010)^2 /$   $\!\! (0,001)^2 /$ $\!\! (0,000)^4$ & $(W(2),000) /$ $\!\! (1,100)^2 /$ $\!\! (1,010)^2 /$ $\!\! (1,001)^2 /$ $\!\! (1,000)^4 /$ $\!\! (0,W(101)) /$  $\!\! (0,100)^4 /$ $\!\! (0,010)^2 /$ $\!\! (0,001)^4 /$  $\!\! (0,000)^7$ \\

$A_2^2$ & $(10,10) /$ $\!\! (10,00)^3 /$  $\!\! (01,01) /$ $\!\! (01,00)^3 /$ $\!\! (00,10)^3 /$  $\!\! (00,01)^3 /$  $\!\! (00,00)^2$ & $(W(11),00) /$ $\!\! (10,10)^3 /$ $\!\! (10,01) /$ $\!\! (10,00)^3 /$ $\!\! (01,10) /$ $\!\! (01,01)^3 /$ $\!\! (01,00)^3 /$ $\!\! (00,W(11)) /$   $\!\! (00,10)^3 /$ $\!\! (00,01)^3 /$ $\!\! (00,00)^{9} $ \\

$A_1^2 A_2$ &  $(1,1,00)^2 /$ $\!\! (1,0,10) /$ $\!\! (1,0,01) /$ $\!\! (1,0,00)^2 /$  $\!\! (0,1,10) /$ $\!\! (0,1,01) /$ $\!\! (0,1,00)^2 /$ $\!\! (0,0,10)^2 /$   $\!\! (0,0,01)^2 /$ $\!\! (0,0,00)^4$ & $(W(2),0,00) /$ $\!\! (1,1,10) /$ $\!\! (1,1,01) /$ $\!\! (1,1,00)^2 /$ $\!\! (1,0,10)^2 /$ $\!\! (1,0,01)^2 /$ $\!\! (1,0,00)^4 /$ $\!\! (0,W(2),00) /$   $\!\! (0,1,10)^2 /$ $\!\! (0,1,01)^2 /$  $\!\! (0,1,00)^4 /$  $\!\! (0,0,W(11)) /$   $\!\! (0,0,10)^3 /$ $\!\! (0,0,01)^3 /$ $\!\! (0,0,00)^5$ \\

$A_1^4$ & $(1,1,1,0) /$ $\!\! (1,0,0,1)^2 /$ $\!\! (1,0,0,0)^4 /$ $\!\! (0,1,0,1)^2 /$ $\!\! (0,1,0,0)^4 /$ $\!\! (0,0,1,1)^2 /$   $\!\! (0,0,1,0)^4$ & $(W(2),0,0,0) /$ $\!\! (1,1,0,1)^2 /$  $\!\! (1,1,0,0)^4 /$ $\!\! (1,0,1,1)^2 /$ $\!\! (1,0,1,0)^4 /$ $\!\! (0,W(2),0,0) /$ $\!\! (0,1,1,1)^2 /$  $\!\! (0,1,1,0)^4 /$ $\!\! (0,0,W(2),0) /$ $\!\! (0,0,0,W(2)) /$ $\!\! (0,0,0,1)^8 / (0,0,0,0)^9$ \\

$A_3$ & $100^4 /$ $\!\! 010^2 /$ $\!\! 001^4 /$  $\!\! 000^{12}$ & $W(101) /$ $\!\! 100^8 /$ $\!\! 010^{6} /$ $\!\! 001^8 /$  $\!\! 000^{18}$  \\

$A_1 A_2$ &  $ (1,10) /$ $\!\! (1,01) /$ $\!\! (1,00)^6 /$ $\!\! (0,10)^4 /$ $\!\! (0,01)^4 /$ $\!\! (0,00)^8$ & $(W(2),00) /$    $\!\! (1,10)^4 /$ $\!\! (1,01)^4 /$ $\!\! (1,00)^8 /$ $\!\! (0,W(11)) /$ $\!\! (0,10)^7 /$ $\!\! (0,01)^7 /$ $\!\! (0,00)^{16}$ \\

$A_1^3$ & $(1,1,0)^2 /$ $\!\! (1,0,1)^2 /$ $\!\! (1,0,0)^4 /$ $\!\! (0,1,1)^2 /$  $\!\! (0,1,0)^4 /$ $\!\! (0,0,1)^4 /$ $\!\! (0,0,0)^8$ & $(W(2),0,0) /$ $\!\! (1,1,1)^2 /$ $\!\! (1,1,0)^4 /$ $\!\! (1,0,1)^4 /$ $\!\! (1,0,0)^8 /$ $\!\! (0,W(2),0) /$  $\!\! (0,1,1)^4 /$   $\!\! (0,1,0)^8 /$ $\!\! (0,0,W(2)) /$ $\!\! (0,0,1)^8 /$ $ \!\! (0,0,0)^{12}$ \\

$(A_1^3)'$ & $(1,1,1) /$ $\!\! (1,0,0)^8 /$ $\!\! (0,1,0)^8 /$ $\!\! (0,0,1)^8$ & $(W(2),0,0) /$ $\!\! (1,1,0)^8 /$ $\!\! (1,0,1)^8 /$ $\!\! (0,W(2),0) /$  $\!\! (0,1,1)^8 /$ $\!\! (0,0,W(2)) /$  $\!\! (0,0,0)^{28}$  \\

$A_2$ & $10^6 /$ $\!\! 01^6 /$ $\!\! 00^{20}$ & $W(11) /$ $\!\! 10^{15} /$ $\!\! 01^{15} /$ $\!\! 00^{35}$ \\

$A_1^2$ & $(1,1)^2 /$ $\!\! (1,0)^8 /$ $\!\! (0,1)^8 /$ $\!\! (0,0)^{16}$ & $(W(2),0) /$ $\!\! (1,1)^7 /$ $\!\! (1,0)^{16} /$ $\!\! (0,W(2)) /$  $\!\! (0,1)^{16} /$ $\!\! (0,0)^{31}$ \\

$A_1$ & $1^{12} /$ $\!\! 0^{32}$ & $W(2) /$ $\!\! 1^{32} /$ $\!\! 0^{66}$ \\

\hline

\end{longtable}

\begin{longtable}{p{0.1\textwidth - 2\tabcolsep}>{\raggedright\arraybackslash}p{0.9\textwidth-\tabcolsep}@{}}

\caption{The composition factors for the action of Levi subgroups of $E_8$ on $L(E_8)$. \label{levie8}} \\

\hline \noalign{\smallskip}

Levi $L'$ & Composition factors of $L(E_8) \downarrow L'$ \\

\hline \noalign{\smallskip}

$E_7$ & $W(\lambda_1) /$ $\!\! \lambda_7^2 /$ $\!\! 0^3$ \\

$A_1 E_6$ & $(W(2),0) /$ $\!\! (1,\lambda_1) /$ $\!\! (1,\lambda_6) /$ $\!\! (1,0)^2 /$ $\!\! (0,\lambda_1) /$ $\!\! (0,W(\lambda_2)) /$  $\!\! (0,\lambda_6) /$ $\!\! (0,0)$ \\

$E_6$ & $\lambda_1^3 /$ $\!\! W(\lambda_2) /$  $\!\! \lambda_6^3 /$ $\!\! 0^8$ \\

$D_7$ & $\lambda_1^2 /$ $\!\! W(\lambda_2) /$  $\!\! \lambda_6 /$ $\!\! \lambda_7 /$ $\!\! 0$  \\

$A_2 D_5$ & $(W(11),0) /$ $\!\! (10,\lambda_1) /$  $\!\! (10,\lambda_4) /$ $\!\! (10,0) /$ $\!\! (01,\lambda_1) /$ $\!\! (01,\lambda_5) /$ $\!\! (01,0) /$ $\!\! (00,W(\lambda_2)) /$  $\!\! (00,\lambda_4) /$  $\!\! (00,\lambda_5) /$ $\!\! (00,0)$ \\

$D_6$ & $ \lambda_1^4 /$ $\!\! W(\lambda_2) /$  $\!\! \lambda_5^2 /$ $\!\! \lambda_6^2 /$ $\!\! 0^6$ \\

$A_1 D_5$ & $(W(2),0) /$ $\!\! (1,\lambda_1)^2 /$ $\!\! (1,\lambda_4) /$ $\!\! (1,\lambda_5) /$ $\!\! (1,0)^4 /$  $\!\! (0,\lambda_1)^2 /$ $\!\! (0,W(\lambda_2)) /$  $\!\! (0,\lambda_4)^2 /$  $\!\! (0,\lambda_5)^2 /$ $\!\! (0,0)^4$ \\

$A_2 D_4$ & $(W(11),0000) /$ $\!\! (10,1000) /$ $\!\! (10,0010) /$ $\!\! (10,0001) /$ $\!\! (10,0000)^3 /$ $\!\! (01,1000) /$ $\!\! (01,0010) /$ $\!\! (01,0001) /$ $\!\! (01,0000)^3 /$  $\!\! (00,1000)^2 /$ $\!\! (00,W(0100)) /$  $\!\! (00,0010)^2 /$ $\!\! (00,0001)^2 /$ $\!\! (00,0000)^{2}$ \\

$D_5$ & $\lambda_1^6 /$ $\!\! W(\lambda_2) /$  $\!\! \lambda_4^4 /$ $\!\! \lambda_5^4 /$ $\!\! 0^{15}$ \\

$A_1 D_4$ & $(W(2),0000) /$  $\!\! (1,1000)^2 /$ $\!\! (1,0010)^2 /$ $\!\! (1,0001)^2 /$ $\!\! (1,0000)^8 /$ $\!\! (0,1000)^4 /$ $\!\! (0,W(0100)) /$ $\!\! (0,0010)^4 /$ $\!\! (0,0001)^4 /$ $\!\! (0,0000)^{9}$ \\

$D_4$ & $1000^8 /$ $\!\! W(0100) /$  $\!\! 0010^8 /$ $\!\! 0001^8 /$ $\!\! 0000^{28}$ \\

$A_7$ & $W(\lambda_1 + \lambda_7) /$ $\!\! \lambda_1 /$ $\!\! \lambda_2 /$ $\!\! \lambda_3 /$ $\!\! \lambda_5 /$ $\!\! \lambda_6 /$ $\!\! \lambda_7 /$ $\!\! 0$ \\

$A_3 A_4$ & $(W(101),0000) /$ $\!\! (100,1000) /$ $\!\! (100,0010) /$ $\!\! (100,0000) /$ $\!\! (010,1000) /$ $\!\! (010,0001) /$ $\!\! (001,0100) /$ $\!\! (001,0001) /$ $\!\! (001,0000) /$ $\!\! (000,W(1001)) /$  $ \!\! (000,0100) /$ $\!\! (000,0010) /$ $\!\! (000,0000)$ \\

$A_1 A_6$ & $(W(2),0) /$ $\!\! (1,\lambda_1) /$ $\!\! (1,\lambda_2) /$ $\!\! (1,\lambda_5) /$ $\!\! (1,\lambda_6) /$ $\!\! (0,W(\lambda_1 + \lambda_6)) /$  $\!\! (0,\lambda_1) /$ $\!\! (0,\lambda_3)/$ $\!\! (0,\lambda_4)/$ $\!\! (0,\lambda_6) /$ $\!\! (0,0)$ \\

$A_1 A_2 A_4$ & $(W(2),00,0000) /$  $\!\! (1,10,0001) /$ $\!\! (1,10,0000) /$ $\!\! (1,01,1000) /$ $\!\! (1,01,0000) /$ $\!\! (1,00,0100) /$ $\!\! (1,00,0010) /$ $\!\! (0,W(11),0000) /$  $\!\! (0,10,1000) /$ $\!\! (0,10,0100) /$ $\!\! (0,01,0010) /$ $\!\! (0,01,0001) /$ $\!\! (0,00,W(1001)) /$ $\!\! (0,00,1000) /$ $\!\! (0,00,0001) /$ $\!\! (0,00,0000)$ \\

$A_6$ & $W(\lambda_1 + \lambda_6) /$ $\!\! \lambda_1^3 /$ $\!\! \lambda_2^2 /$ $\!\! \lambda_3 /$ $\!\! \lambda_4 /$ $\!\! \lambda_5^2 /$ $\!\! \lambda_6^3 /$ $\!\! 0^4$  \\

$A_1 A_5$ & $(W(2),0) /$  $\!\! (1,\lambda_1)^2 /$ $\!\! (1,\lambda_2) /$ $\!\! (1,\lambda_4) /$ $\!\! (1,\lambda_5)^2 /$ $\!\! (1,0)^2 /$ $\!\! (0,W(\lambda_1 + \lambda_5)) /$ $\!\! (0,\lambda_1)^2 /$ $\!\! (0,\lambda_2) /$ $\!\! (0,\lambda_3)^2/$ $\!\! (0,\lambda_4) /$ $\!\! (0,\lambda_5)^2 /$ $\!\! (0,0)^4$ \\

$A_2 A_4$ & $(W(11),0000) /$ $\!\! (10,1000) /$ $\!\! (10,0100) /$ $\!\! (10,0001)^2 /$ $\!\! (10,0000)^2 /$ $\!\! (01,1000)^2 /$   $\!\! (01,0010) /$ $\!\! (01,0001) /$ $\!\! (01,0000)^2 /$ $\!\! (00,W(1001)) /$ $\!\! (00,1000) /$ $\!\! (00,0100)^2 /$ $\!\! (00,0010)^2 /$   $\!\! (00,0001) /$ $\!\! (00,0000)^4$ \\

$A_1^2 A_4$ & $(W(2),0,0000) /$ $\!\! (1,1,1000) /$ $\!\! (1,1,0001) /$ $\!\! (1,1,0000)^2 /$ $\!\! (1,0,1000) /$ $\!\! (1,0,0100) /$   $\!\! (1,0,0010) /$ $\!\! (1,0,0001) /$  $\!\! (1,0,0000)^2 /$ $\!\! (0,W(2),0000) /$ $\!\! (0,1,1000) /$ $\!\! (0,1,0100) /$ $\!\! (0,1,0010) /$ $\!\! (0,1,0001) /$ $\!\! (0,1,0000)^2 /$ $\!\! (0,0,W(1001)) /$ $\!\! (0,0,1000)^2 /$ $\!\! (0,0,0100) /$  $\!\! (0,0,0010) /$ $\!\! (0,0,0001)^2 /$ $\!\! (0,0,0000)^2$ \\

$A_3^2$ & $(W(101),000) /$  $\!\! (100,100) /$ $\!\! (100,010) /$ $\!\! (100,001) /$ $\!\! (100,000)^2 /$ $\!\! (010,100) /$ $\!\! (010,001) /$ $\!\! (010,000)^2  /$ $\!\! (001,100) /$  $\!\! (001,010) /$ $\!\! (001,001) /$ $\!\! (001,000)^2 /$ $\!\! (000,W(101)) /$  $\!\! (000,100)^2 /$ $\!\! (000,010)^2 /$ $\!\! (000,001)^2 /$ $\!\! (000,000)^{2}$ \\

$A_1 A_2 A_3$ & $(W(2),00,000) /$ $\!\! (1,10,001) /$ $\!\! (1,10,000)^2 /$ $\!\! (1,01,100) /$ $\!\! (1,01,000)^2 /$ $\!\! (1,00,100) /$ $\!\! (1,00,010)^2 /$ $\!\! (1,00,001) /$ $\!\! (0,W(11),000) /$   $\!\! (0,10,100)^2 /$ $\!\! (0,10,010) /$ $\!\! (0,10,000) /$ $\!\! (0,01,010) /$ $\!\! (0,01,001)^2 /$ $\!\!(0,01,000) /$ $\!\! (0,00,W(101)) /$ $\!\! (0,00,100)^2 /$ $\!\! (0,00,001)^2 /$ $\!\! (0,00,000)^4$ \\

$A_1^2 A_2^2$ & $(W(2),0,00,00) /$  $\!\!(1,1,10,00) /$ $\!\!(1,1,01,00) /$ $\!\!(1,1,00,10) /$ $\!\!(1,1,00,01) /$ $\!\!(1,0,10,01) /$ $\!\!(1,0,10,00) /$ $\!\!(1,0,01,10) /$ $\!\!(1,0,01,00) /$ $\!\!(1,0,00,10) /$ $\!\!(1,0,00,01) /$ $\!\!(1,0,00,00)^2 /$ $\!\! (0,W(2),00,00) /$ $\!\!(0,1,10,10) /$ $\!\!(0,1,10,00) /$ $\!\!(0,1,01,01) /$ $\!\!(0,1,01,00) /$ $\!\!(0,1,00,10) /$ $\!\!(0,1,00,01) /$ $\!\!(0,1,00,00)^2 /$ $\!\! (0,0,W(11),00) /$ $\!\!(0,0,10,10) /$ $\!\!(0,0,10,01) /$ $\!\!(0,0,10,00) /$ $\!\!(0,0,01,10) /$ $\!\!(0,0,01,01) /$  $\!\!(0,0,01,00) /$ $\!\! (0,0,00,W(11)) /$ $\!\!(0,0,00,10) /$ $\!\!(0,0,00,01) /$ $\!\!(0,0,00,00)^2 $ \\

$A_5$ & $W(\lambda_1 + \lambda_5) /$ $\!\! \lambda_1^6 /$ $\!\! \lambda_2^3 /$ $\!\! \lambda_3^2 /$ $\!\! \lambda_4^3 /$ $\!\! \lambda_5^6 /$ $\!\! 0^{11}$  \\
 
$A_1 A_4$ & $(W(2),0000) /$ $\!\! (1,1000)^3 /$ $\!\! (1,0100) /$  $\!\! (1,0010) /$  $\!\! (1,0001)^3 /$ $\!\! (1,0000)^6 /$ $\!\! (0,W(1001)) /$ $\!\! (0,1000)^4 /$ $\!\! (0,0100)^3 /$ $\!\! (0,0010)^3 /$ $\!\! (0,0001)^4 /$  $\!\! (0,0000)^9$ \\

$A_2 A_3$ & $(W(11),000) /$ $\!\! (10,100)^2 /$ $\!\! (10,010) /$ $\!\! (10,001)^2 /$ $\!\! (10,000)^5 /$ $\!\! (01,100)^2 /$  $\!\! (01,010) /$ $\!\! (01,001)^2 /$  $\!\! (01,000)^5 /$ $\!\! (00,W(101)) /$ $\!\! (00,100)^4 /$ $\!\! (00,010)^4 /$ $\!\! (00,001)^4 /$ $\!\! (00,000)^7$ \\

$A_1^2 A_3$ & $(W(2),0,000) /$ $\!\! (1,1,100) /$   $\!\! (1,1,001) /$ $\!\! (1,1,000)^4 /$ $\!\! (1,0,100)^2 /$ $\!\! (1,0,010)^2 /$ $\!\! (1,0,001)^2 /$ $\!\! (1,0,000)^4 /$ $\!\! (0,W(2),000) /$ $\!\! (0,1,100)^2 /$ $\!\! (0,1,010)^2 /$ $\!\! (0,1,001)^2 /$ $\!\! (0,1,000)^4 /$  $\!\! (0,0,W(101)) /$ $\!\! (0,0,100)^4 /$ $\!\! (0,0,010)^2 /$ $\!\! (0,0,001)^4 /$ $\!\! (0,0,000)^7$ \\

$A_1 A_2^2$ & $(W(2),00,00) /$ $\!\!(1,10,01) /$ $\!\!(1,10,00)^3 /$ $\!\!(1,01,10) /$ $\!\!(1,01,00)^3 /$ $\!\!(1,00,10)^3 /$ $\!\!(1,00,01)^3 /$   $\!\!(1,00,00)^2 /$ $\!\! (0,W(11),00) /$  $\!\!(0,10,10)^3 /$  $\!\!(0,10,01) /$ $\!\!(0,10,00)^3 /$ $\!\!(0,01,10) /$ $\!\!(0,01,01)^3 /$ $\!\!(0,01,00)^3 /$ $\!\! (0,00,W(11)) /$ $\!\!(0,00,10)^3 /$ $\!\!(0,00,01)^3 /$ $\!\!(0,00,00)^9 $ \\

$A_1^3 A_2$ & $(W(2),0,0,00) /$ $\!\!(1,1,1,00)^2 /$  $\!\!(1,1,0,10) /$ $\!\!(1,1,0,01) /$ $\!\!(1,1,0,00)^2 /$ $\!\!(1,0,1,10) /$ $\!\!(1,0,1,01) /$ $\!\!(1,0,1,00)^2 /$ $\!\!(1,0,0,10)^2 /$ $\!\! (1,0,0,01)^2 /$ $\!\!(1,0,0,00)^4 /$ $\!\! (0,W(2),0,00) /$  $\!\!(0,1,1,10) /$ $\!\!(0,1,1,01) /$ $\!\!(0,1,1,00)^2 /$ $\!\!(0,1,0,10)^2 /$ $\!\!(0,1,0,01)^2 /$  $\!\!(0,1,0,00)^4 /$ $\!\! (0,0,W(2),00) /$  $\!\!(0,0,1,10)^2 /$ $\!\!(0,0,1,01)^2 /$ $\!\!(0,0,1,00)^4 /$ $\!\! (0,0,0,W(11)) /$ $\!\!(0,0,0,10)^3 /$ $\!\!(0,0,0,01)^3 /$ $\!\!(0,0,0,00)^5 $ \\

$A_4$ & $W(1001) /$ $\!\! 1000^{10} /$ $\!\! 0100^5 /$ $\!\! 0010^5 /$ $\!\! 0001^{10} /$ $\!\! 0000^{24}$  \\

$A_1 A_3$ & $(W(2),000) /$  $\!\! (1,100)^4 /$  $\!\! (1,010)^2 /$ $\!\! (1,001)^4 /$ $\!\! (1,000)^{12} /$ $\!\! (0,W(101)) /$ $\!\! (0,100)^8 /$ $\!\! (0,010)^6 /$ $\!\! (0,001)^8 /$ $\!\! (0,000)^{18}$ \\

$A_2^2$ & $(W(11),00) /$ $\!\!(10,10)^3 /$ $\!\!(10,01)^3 /$ $\!\!(10,00)^9 /$ $\!\!(01,10)^3 /$ $\!\!(01,01)^3 /$ $\!\!(01,00)^9 /$ $\!\! (00,W(11)) /$ $\!\!(00,10)^9 /$ $\!\!(00,01)^9 /$ $\!\!(00,00)^{16} $ \\

$A_1^2 A_2$ & $(W(2),0,00) /$  $\!\!(1,1,10) /$ $\!\!(1,1,01) /$ $\!\!(1,1,00)^6 /$ $\!\!(1,0,10)^4 /$ $\!\!(1,0,01)^4 /$ $\!\!(1,0,00)^8 /$ $\!\! (0,W(2),00) /$  $\!\!(0,1,10)^4 /$ $\!\!(0,1,01)^4 /$ $\!\!(0,1,00)^8 /$ $\!\! (0,0,W(11)) /$ $\!\!(0,0,10)^7 /$ $\!\!(0,0,01)^7 /$ $\!\!(0,0,00)^{16} $ \\

$A_1^4$ & $(W(2),0,0,0) /$    $\!\!(1,1,1,0)^2 /$  $\!\!(1,1,0,1)^2 /$ $\!\!(1,1,0,0)^4 /$ $\!\!(1,0,1,1)^2 /$ $\!\!(1,0,1,0)^4 /$ $\!\!(1,0,0,1)^4 /$ $\!\!(1,0,0,0)^8 /$ $\!\! (0,W(2),0,0) /$ $\!\!(0,1,1,1)^2 /$ $\!\!(0,1,1,0)^4 /$ $\!\!(0,1,0,1)^4 /$ $\!\!(0,1,0,0)^8 /$ $\!\! (0,0,W(2),0) /$ $\!\!(0,0,1,1)^4 /$ $\!\!(0,0,1,0)^8 /$ $\!\! (0,0,0,W(2)) /$  $\!\!(0,0,0,1)^8 /$ $\!\!(0,0,0,0)^{12} $ \\

$A_3$ & $W(101) /$ $\!\! 100^{16} /$ $\!\! 010^{10} /$ $\!\! 001^{16} /$  $\!\! 000^{45}$ \\

$A_1 A_2$ & $(W(2),00) /$  $\!\! (1,10)^{6} /$ $\!\! (1,01)^{6} /$ $\!\! (1,00)^{20} /$  $\!\! (0,W(11)) /$ $\!\! (0,10)^{15} /$ $\!\! (0,01)^{15} /$ $\!\! (0,00)^{35}$  \\

$A_1^3$ & $(W(2),0,0) /$  $\!\! (1,1,1)^{2} /$ $\!\! (1,1,0)^{8} /$ $\!\! (1,0,1)^{8} /$ $\!\! (1,0,0)^{16} /$ $\!\! (0,W(2),0) /$  $\!\! (0,1,1)^{8} /$   $\!\! (0,1,0)^{16} /$ $\!\! (0,0,W(2)) /$  $\!\! (0,0,1)^{16} /$ $\!\! (0,0,0)^{31}$  \\

$A_2$ & $W(11) /$ $\!\! 10^{27} /$ $\!\! 01^{27} /$ $\!\! 00^{78}$ \\

$A_1^2$ & $(W(2),0) /$ $\!\! (1,1)^{12} /$ $\!\! (1,0)^{32} /$ $\!\! (0,W(2)) /$ $\!\! (0,1)^{32} /$  $\!\! (0,0)^{66}$ \\

$A_1$ &  $W(2) / 1^{56} / 0^{133}$ \\

\hline

\end{longtable}
\end{appendices}

\bibliographystyle{amsplain}

\bibliography{../biblio}

\end{document}